\documentclass[english]{enigtex-article}

\usepackage{enigtex-biblatex}
\usepackage{3-handling}

\newcommand{\attachingboundary}{\partial_\mathrm{a}}
\newcommand{\remainingboundary}{\partial_\mathrm{r}}
\newcommand{\incident}{\delta}
\newcommand{\acts}[2]{{}^{#1}\mkern-2mu{#2}}
\newcommand{\handleenumerationstyle}{\renewcommand{\labelenumi}{\textbf{\theenumi}}}

\title{Evaluating TQFT invariants from $G$-crossed braided spherical fusion categories via Kirby diagrams with 3-handles}

\begin{document}

\maketitle

\begin{abstract}
	A family of TQFTs parametrised by $G$-crossed braided spherical fusion categories has been defined recently as a state sum model and as a Hamiltonian lattice model.
	Concrete calculations of the resulting manifold invariants are scarce
	because of the combinatorial complexity of triangulations, if nothing else.
	Handle decompositions, and in particular Kirby diagrams are known to offer an economic and intuitive description of 4-manifolds.
	We show that if 3-handles are added to the picture,
	the state sum model can be conveniently redefined by translating Kirby diagrams into the graphical calculus of a $G$-crossed braided spherical fusion category.

	This reformulation is very efficient for explicit calculations,
	and the manifold invariant is calculated for several examples.
	It is also shown that in most cases,
	the invariant is multiplicative under connected sum,
	which implies that it does not detect exotic smooth structures.
\end{abstract}

\tableofcontents
\listoffixmes

\section{Introduction}

In the study of 4-dimensional topological quantum field theories (TQFTs),
few interesting examples are known,
and even fewer are defined rigorously as axiomatic TQFTs.
The fewest are economic to calculate.

Often, families of TQFTs like the Reshetikhin-Turaev model \cite{ReshetikhinTuraev}, the Turaev-Viro-Barrett-Westbury state sum \cite{TuraevViro:1992865, BarrettWestbury:1993Invariants} or the Crane-Yetter model \cite{CraneYetterKauffman:1997177} are indexed by an algebraic datum such as a fusion category with extra structure.
More recently \cite{BarkeshliBondersonChengWang2014SymmetryDefectsTopPhases, WilliamsonWangGCrossedTopPhases2017},
$G$-crossed braided fusion categories have been studied in Hamiltonian approaches to topological phases,
and a state sum model (and thus, an axiomatic TQFT) was defined \cite{Cui2016TQFTs}.
This \emph{$G$-crossed model} is thus defined explicitly, but it is hard to calculate concrete values for the partition function.

The Crane-Yetter model is famously invertible if the input datum is a modular category,
but it had been suspected that it is noninvertible for nonmodular categories,
with a rigorous proof only given more than two decades after its definition \cite{BaerenzBarrett2016Dichromatic}.
One reason seems to be that state sum models based on triangulations are easy to define, but hard to calculate.
Once the partition function is defined in terms of handle decompositions,
it is much easier to calculate concrete values,
and indeed the second open question in \cite[Section 7]{Cui2016TQFTs} asks for a description of the $G$-crossed model in terms of Kirby diagrams.

This is achieved here, adapting the techniques from \cite{BaerenzBarrett2016Dichromatic}.
As the central result,
we define for every $G$-crossed braided spherical fusion category $\mathcal{C}$ an invariant $I_\mathcal{C}$ of smooth, closed, oriented manifolds
and show (Theorem \ref{thm:equal to SSM}):
\begin{theorem*}
	Up to a factor involving the Euler characteristic,
	the invariant $I_\mathcal{C}$
	is equal to the state sum $Z_\mathcal{C}$ in \cite[(23)]{Cui2016TQFTs}.
	Explicitly, let $M$ be a connected, smooth, oriented, closed 4-manifold and $\mathcal{T}$ an arbitrary triangulation,
	then:
	\begin{equation*}
		I_\mathcal{C}(M) = Z_\mathcal{C}(M; \mathcal{T}) \cdot \qdim{\Omega_\mathcal{C}}^{1 - \chi(M)} \cdot \lvert G \rvert
	\end{equation*}
\end{theorem*}
The invariant $I_\mathcal{C}$ is in fact much easier to calculate than the state sum,
and several examples are given.
It is also shown that in most cases,
$I_\mathcal{C}$ does not differentiate between smooth structures on the same topological manifold.
\fxwarning{Still a bit short?}
\subsection*{Triangulations and handle decompositions}
\fxwarning{Not clear completely yet how to order this ssection with the other in Section 1}
\fxwarning{Reference?}
As a general theme in many flavours of topology,
spaces are built up from simple elementary building blocks.
A large class of topological spaces can be constructed as simplicial complexes or as CW-complexes.
In the former approach,
the elementary building blocks are simpler and the possibilities of gluing them are fewer,
whereas in the latter,
spaces can often be described much more succinctly.
In both cases, each building block has an inherent dimension,
which is a natural number.
Spaces with building blocks of dimension at most $n$ are usually said to be $n$-dimensional.

For smooth manifolds, the situation is similar.
They admit triangulations,
which yield simplicial complexes,
but also handle decompositions,
which are analogous to CW-complexes
.
Triangulations decompose an $n$-manifold into $k$-simplices,
and handle decompositions consist of $k$-handles,
respectively for $0 \leq k \leq n$.

We favour handle decompositions over triangulations here because they are much more succinct.\footnote{%
In higher dimensions, there is another reason:
Triangulations do not completely specify a smooth structure,
but only a piecewise linear (PL) structure,
whereas handle decompositions exist both in the smooth and in the PL category.
In four dimensions, the PL and smooth categories are still equivalent,
which allows to formulate smooth TQFTs as state sum models over triangulations.
}
One can specify the 4-dimensional sphere $S^4$ with as little as a 0-handle and a 4-handle,
while the standard triangulation coming from the boundary of the 5-simplex contains 20 triangles.
It is no surprise that state sum models are often defined for triangulations,
but no values are calculated even for simple manifolds
(as in \cite{Cui2016TQFTs, CraneYetterKauffman:1997177}).
Evaluating TQFTs on bordisms via handle decompositions is both conceptually more direct \cite{JuhaszTQFTsSurgery14},
as well as computationally more efficient.
For example, the Crane-Yetter model was famously shown to be invertible for modular categories using handle decompositions \cite[Theorem 4.5]{Roberts:1995SkeinTheoryTV},
and most concrete values for premodular categories have only been calculated more than 20 years after its state sum definition,
again using handle decompositions \cite{BaerenzBarrett2016Dichromatic}.

The price to pay is the additional complexity when attaching handles to each other.
While the gluing data of simplices is completely combinatorial,
handle attachment data is quite topological in nature,
and is best described in diagrams.
In contrast to simplices,
handles of all dimensions $0 \leq k \leq n$
must be thickened to $n$ dimensions before they are glued along their $(n-1)$-dimensional boundaries.
Consequently, all of the attachment data can be described diagrammatically in $(n-1)$-dimensions.
We thus describe 4-manifolds with 3-dimensional diagrams.

To evaluate the TQFT on them,
these diagrams are decorated with data from a (non-strict) 3-category.
In particular, a $G$-crossed braided category is a special case of a monoidal 2-category \cite[Section 6]{Cui2016TQFTs},
and thus a special case of a 3-category.
And indeed, there is a beautiful way to label diagrams of 4-dimensional handle decompositions with $G$-crossed braided spherical fusion categories ($G\times$-BSFCs),
by depicting 3-handles explicitly in the calculus.
\fxwarning{Maybe say why we need 3-handles? Because we need to label 1-simplices with group elements?}

\subsection*{Outline}

We begin by briefly introducing $G$-crossed braided spherical fusion categories in Section \ref{sec:intro Gx-BSCFs}.
Additionally to defining the invariant,
this article gives an introduction to Kirby diagrams with 3-handles,
which does not seem to exist in the literature.
Section \ref{sec:Kirby calculus} contains this material,
and it should be accessible to TQFT researchers with a basic background in Kirby calculus of 4-manifolds.
The aim of Section \ref{sec:Graphical calculus} is to show that the language of Kirby diagrams with 3-handles is very natural for graphical calculus in $G$-crossed braided spherical fusion categories.
This is the central ingredient to Section \ref{sec:invariant},
where the invariant $I_\mathcal{C}$ is defined.
Explicit example calculations and general properties of the invariant are given in Section \ref{sec:calculations}.
The connection to the state sum model from \cite{Cui2016TQFTs} and the corresponding TQFT is made in Section \ref{sec:ssm}.
Further ideas for generalisations,
in the direction of defining spherical fusion 2-categories,
are discussed in Section \ref{sec:3-cats}.

\medskip

The figures in this article often make use of colours to illustrate calculations and may be harder to interpret when printed in black and white.

\section{Introduction to $G$-crossed braided spherical fusion categories}
\label{sec:intro Gx-BSCFs}
\subsection{Spherical fusion categories}
\label{sec:spherical intro}

Spherical fusion categories over $\Co$ are well-known,
and basic knowledge is assumed.
We will mainly use notation, conventions and graphical calculus from \cite[Section 2.1]{BaerenzBarrett2016Dichromatic}.
For convenience, an overview over the relevant notation is given in Table \ref{table:notation}.
\begin{table}
	\begin{tabular}{p{0.45\textwidth}p{0.45\textwidth}}
		\toprule
		Spherical fusion category
			& $(\mathcal{C}, \otimes, \mathcal{I})$
		\\\midrule
		Objects
			& $A, B, C, \dots$
		\\\midrule
		Morphisms
			& $\iota \in \langle A \otimes B \otimes \cdots \rangle \coloneqq \mathcal{C}(\mathcal{I}, A \otimes B \otimes \cdots)$
		\\\midrule
		Simple objects
			& $X, Y, Z$
		\\\midrule
		Set of representants of isomorphism classes of simple objects
			& $\mathcal{O}(\mathcal{C})$
		\\\midrule
		Dual object
			& $A^*$
		\\\midrule
		Nondegenerate morphism pairing
			& $(-, -)\colon \langle A_1 \otimes \cdots A_n \rangle \otimes \langle A_n^* \otimes \cdots A_1^* \rangle \to \Co$
		\\\midrule
		Trace
			& $\tr\colon \mathcal{C}(A,A) \to \Co$
		\\\midrule
		Categorical dimension
			& $\qdim{X} \coloneqq \tr(1_x)$
		\\\midrule
		Fusion algebra (complexified Grothendieck ring)
			& $\Co[\mathcal{C}]$
		\\\midrule
		``Kirby colour''
			& $\Omega_\mathcal{C} \coloneqq \bigoplus_{X \in \mathcal{O}(\mathcal{C})} \qdim{X} X \in \Co[\mathcal{C}]$
		\\\bottomrule
	\end{tabular}
	\caption{Notation and conventions in spherical fusion categories.}
	\label{table:notation}
\end{table}
We will augment the graphical calculus of fusion categories with ``round coupons'' for morphisms
(this is justified because a pivotal structure is assumed to exist)
as introduced in \cite[Section 5.3]{BakalovKirillov2001} and explained for example in \cite[Section 1]{BalsamKirillov1}.
Essentially, morphisms do not differentiate between source and target,
as these can be interchanged coherently by means of duals.

We briefly revisit the cyclically symmetric definition of morphism spaces:
\begin{align}
	\iota \in \langle A_1 \otimes A_2 \otimes \cdots \otimes A_n \rangle
		& \coloneqq \mathcal{C}(\mathcal{I}, A_1 \otimes A_2 \otimes \cdots)
	\\	& \stackrel{\mathllap{\text{pivotal structure}}}{\cong} \mathcal{C}(\mathcal{I}, A_{j+1} \otimes \cdots \otimes A_n \otimes A_1 \otimes \cdots \otimes A_j)
	\\	& \stackrel{\mathllap{\text{dualisation}}}{\cong} \mathcal{C}(A_j^* \otimes A_{j-1}^* \otimes \cdots \otimes A_1^* \otimes A_n^* \otimes \cdots \otimes A_k^*, A_j \otimes \cdots \otimes A_{k-1})
\end{align}
The morphism pairing
$(-, -)\colon \langle A_1 \otimes \cdots A_n \rangle \otimes \langle A_n^* \otimes \cdots A_1^* \rangle \to \Co$
is defined by dualisation, composition, and the isomorphism $\mathcal{C}(\mathcal{I},\mathcal{I}) \cong \Co$.
It is \emph{nondegenerate},
which implies the existence of dual bases $\{\alpha_i\} \subset \langle A_1 \otimes \cdots A_n \rangle$ and $\{\tilde{\alpha}_i\} \subset \langle A_n^* \otimes \cdots A_1^* \rangle$ satisfying $(\tilde{\alpha}_i,\alpha_j) = \delta_{i,j}$.
The choice of $\{\alpha_i\}$ is arbitrary and determines $\{\tilde{\alpha}_i\}$ completely.
\begin{definition}
	\label{def:dual bases}
	Together, the dual bases define a unique entangled vector:
	\begin{equation}
		\sum_i \alpha_i \otimes \tilde{\alpha}_i \in \langle A_1 \otimes \cdots A_n \rangle \otimes \langle A_n^* \otimes \cdots A_1^* \rangle
	\end{equation}
	Especially in the graphical calculus (Figure \ref{fig:dual bases}),
	this vector will often be abbreviated as $\alpha \otimes \tilde{\alpha}$,
	suppressing the indices.
\end{definition}
\begin{figure}
	\centering
	\begin{equation*}
		\begin{tikzpicture} [ baseline ]
			\node [ plaquette ] (alpha)                  {$\alpha$};
			\node [ plaquette ] (alphatilde) at (3, 0) {$\tilde\alpha$};
			\draw [ thick ]
				foreach \i in {1, 2, ..., 6}
				{(alpha)      -- ++(\i*60:0.9) ++(\i*60:0.25) node {$A_\i$}};
			\draw [ thick ]
				foreach \i in {1, 2, ..., 6}
				{(alphatilde) -- ++(180-\i*60:0.9) ++(180-\i*60:0.25) node {$A_\i^*$}};
		\end{tikzpicture}
		=
		\sum_i
		\begin{tikzpicture} [ baseline ]
			\node [ plaquette ] (alpha)                  {$\alpha_i$};
			\node [ plaquette ] (alphatilde) at (3, 0) {${\tilde\alpha}_i$};
			\draw [ thick ]
				foreach \i in {1, 2, ..., 6}
				{(alpha)      -- ++(\i*60:0.9) ++(\i*60:0.25) node {$A_\i$}};
			\draw [ thick ]
				foreach \i in {1, 2, ..., 6}
				{(alphatilde) -- ++(180-\i*60:0.9) ++(180-\i*60:0.25) node {$A_\i^*$}};
		\end{tikzpicture}
	\end{equation*}
	\caption{
		The dual bases convention with round coupons.
		Compare, e.g. \cite[(1.8)]{BalsamKirillov1}.
	}
	\label{fig:dual bases}
\end{figure}

\subsection{$G$-crossed braided spherical fusion categories}
\label{sec:Gx intro}

$G$-crossed braided (spherical) fusion categories
(short: $G\times$-BSFCs)
were introduced in \cite{TuraevHFTAndGCats2000} as ``crossed group categories'' and have since received a diversity of names such as $G$-crossed braided SFCs,
$G$-equivariant braided SFCs or braided $G$-crossed SFCs.
We will adopt the nomenclature $G\times$-BSFC,
but warn that $G\times$-BSFCs are usually \emph{not} braided.
Rather, they carry a new structure, the \emph{crossed braiding}.

Just as fusion categories categorify and vastly generalise finite groups,
$G\times$-BSFCs categorify finite crossed modules.
(This viewpoint is implemented explicitly in \cite[Section 4.2]{Cui2016TQFTs}.)
A crossed module consists of two groups $H$ and $G$,
a homomorphism $\deg\colon H \to G$ and a group action of $G$ on $H$.
The group $H$ is ``crossed commutative'',
that is, commutative up to an action by $G$.
This axiom is usually called the ``Peiffer rule''.
$G\times$-BSFCs now categorify finite crossed modules in the following way:
The group $H$ is generalised to a spherical fusion category $\mathcal{C}$,
while $G$ stays a group.
The homomorphism $\deg$ gives way to a $G$-grading of $\mathcal{C}$.
Just as abelian groups are usually categorified to braided fusion categories
and the commutativity axiom is replaced by the braiding,
the Peiffer rule is now replaced by a crossed braiding.

Writing out the axioms explicitly
and following the notation from \cite[Section 2.1]{Cui2016TQFTs},
we have:
\begin{definitions}
	Let $G$ be a finite group and $\mathcal{C}$ a spherical fusion category.
	Denote by $\underline{G}$ the discrete monoidal category whose objects are elements of $G$
	and the monoidal structure matches the group structure.
	\begin{itemize}
		\item A $G$-grading on $\mathcal{C}$ is a decomposition $\mathcal{C} \cong \bigoplus_{g \in G} \mathcal{C}_g$ into semisimple linear categories
			such that $\mathcal{C}_{g_1} \otimes \mathcal{C}_{g_2} \subset \mathcal{C}_{g_1g_2}$.
			This defines a function $\deg\colon \mathcal{O}(\mathcal{C}) \to G$.
		\item A $G$-action on $\mathcal{C}$ is a monoidal functor $(F, \eta, \epsilon)\colon \underline{G} \to \Aut_{\otimes,\text{piv}} (\mathcal{C})$.
			Explicitly, there is a monoidal, pivotal automorphism of $\mathcal{C}$ for each group element $g$.
			We will usually abbreviate $\acts{g}{X} \coloneqq F(g)(X)$,
			and thus the coherence isomorphisms are $\eta(g_1,g_2)_X \colon \acts{g_1}{\left(\acts{g_2}{X}\right)} \xrightarrow{\cong} \acts{g_1g_2}{X}$ and $\epsilon_X\colon \acts{e}{X} \xrightarrow{\cong} X$.
		\item A \textbf{$G$-crossed braided spherical fusion category} consists of the following structure and axioms:
			\begin{itemize}
				\item A spherical fusion category $\mathcal{C}$,
				\item a $G$-grading on $\mathcal{C}$,
				\item a $G$-action on $\mathcal{C}$ such that $\acts{g}{ \left(\mathcal{C}_{g'}\right)} \subset \mathcal{C}_{gg'g^{-1}}$,
				\item for every $g \in G, X \in \mathcal{C}_g, Y \in \mathcal{C}$,
					a natural isomorphism $c_{X,Y} \colon X \otimes Y \to \acts{g}{Y} \otimes X$,
					which is called the \textbf{crossed braiding},
				\item subject to certain axioms which are detailed e.g. in \cite[Definition 2.2]{Cui2016TQFTs} or \cite[Definition 4.41]{OnBraidedFusionCats}.
			\end{itemize}
	\end{itemize}
\end{definitions}
\begin{remark}
	A $G\times$-BSFC is usually not braided.
	$c$ is only a braiding if the $G$-action is trivial.
\end{remark}
\begin{remark}
	\label{rem:Gx grading not faithful}
	The $G$-grading is not required to be faithful,
	i.e. $\mathcal{C}_g$ may be 0 for some $g \neq e$.
	In contrast to mere graded fusion categories,
	such a $G\times$-BSFC is not always equivalent to a faithfully graded one for a subgroup of $G$,
	since the $G$-action can contain information about the whole group.
\end{remark}
\begin{examples}
	\begin{itemize}
		\item Any ribbon fusion (``premodular'') category is a spherical braided fusion category,
			and thus a $G\times$-BSFC for $G$ the universal grading group and the trivial action.
		\item \cite[Section 4.4]{OnBraidedFusionCats} explains how a braided inclusion $\Rep(G) \subset \mathcal{D}$ of the representations of a finite group into a braided fusion category yields a $G\times$-BSFC.
	\end{itemize}
\end{examples}

\section{Kirby calculus with 3-handles}
\label{sec:Kirby calculus}

\subsection{Handle decompositions}

In this section, handle decompositions of 4-manifolds and their 3-dimensional diagrams are described.
(All manifolds will be assumed to be smooth, oriented and compact.)
While handle decompositions and Kirby diagrams are described extensively in the literature \cite{GompfStipsicz, Akbulut:4manifolds, Kirby:Topology4Manifolds},
3-handles are rarely described pictorially,
the only case known to the author being \cite{HarerKasKirby86HandlebodyComplexSurfaces}.
This will turn out to be an important conceptual clarification and computational simplification in the graphical calculus of $G\times$-BSFCs,
and is thus introduced here at length,
interspersed with a recapitulation of well-known material.

\fxwarning{References for these definitions}
\begin{definitions}
	\label{def:handles}
	An \textbf{$n$-dimensional $k$-handle} is the manifold with corners\footnote{%
		Details and references about manifolds with corners and their boundaries can be found in Appendix \ref{sec:corners}.
	}
	$h_k \coloneqq D^k \times D^{n-k}$.
	\begin{itemize}
		\item From its corner structure,
			the boundary of a $k$-handle $\partial h_k = S^{k-1} \times D^{n-k} \cup D^k \times S^{n-k-1}$
			is split into the
			\textbf{attaching region} $\attachingboundary h_k = S^{k-1} \times D^{n-k}$
			and the \textbf{remaining region} $\remainingboundary h_k = D^k \times S^{n-k-1}$.
		\item $S^{k-1} \times \{0\} \subset \attachingboundary h_k$ is called the \textbf{attaching sphere}.
		\item $\{0\} \times S^{n-k-1} \subset \remainingboundary h_k$ is called the \textbf{remaining sphere},
		or ``belt sphere''.
		\fxerror{Decide whether we want to have ``belt sphere'' globally.}
	\end{itemize}
\end{definitions}

We wish to decompose $n$-manifolds $M^n$ into a filtration $\emptyset = M_{-1} \subset M_0 \subset M_1 \subset \cdots \subset M_n \cong M$,
where each $M_k$ is produced from $M_{k-1}$ by attaching $k$-handles.
This decomposition is then called a \emph{handlebody},
and it is often used interchangeably with $M$,
or regarded as extra structure on $M$.

\begin{definitions}
	\begin{itemize}
		\item A \textbf{$k$-handle attachment map} $\phi$ on an $n$-manifold $M$ is an embedding of the attaching region $\attachingboundary h_k$ of an $n$-dimensional $k$-handle into $\partial M$.
		\item The result of a handle attachment $\phi\colon \attachingboundary h_k \hookrightarrow \partial M$ is the manifold $M \cup_\phi h_k$,
			where the handle is glued along the embedded attaching region.
		\item A \textbf{$(-1)$-handlebody} is the empty set.
		\item A \textbf{$k$-handlebody} is a $(k-1)$-handlebody and successive $k$-handle attachments on it.
		\item A handle decomposition of a manifold $M^n$ is a diffeomorphism to an $n$-handlebody.
	\end{itemize}
\end{definitions}

\begin{remark}
	\label{rem:surgery}
	When attaching a $k$-handle to a manifold $M$ along $\phi$,
	the attaching boundary is removed and the remaining boundary added.
	The boundary of the result of a handle attachment is thus:
	\[ \partial (M \cup_\phi h_k) = \partial M \backslash \phi(\attachingboundary h_k) \cup \remainingboundary h_k \]
	This operation is known as \textbf{performing surgery} on $\partial M$ along $\phi$.
\end{remark}
\begin{remark}
	By definition, the attaching regions of the attached handles are in the interior of the resulting manifold,
	therefore further handles are always attached to the remaining regions of previous handles,
	explaining the naming choice ``remaining''.
\end{remark}
\fxwarning{Should we replace ``remaining boundary'' with ``belt region''?}

\fxwarning{Either at least reference to illustration or put one in myself}

\begin{theorem}[{Well-known, e.g. \cite[Section 4.2]{GompfStipsicz}}]
	Smooth manifolds have handle decompositions.
	In particular, smooth compact manifolds have decompositions with finitely many handles.
\end{theorem}
The proof of this theorem is usually via Morse theory.
Given a self-indexing Morse function,
every critical point of index $k$ corresponds to a $k$-handle.

Since arbitrarily many different handle decompositions may describe the same manifold up to diffeomorphism,
it is important to relate them.
Luckily, there is a simple complete set of moves to translate from any two diffeomorphic handle decompositions.
\begin{definition}
	\label{def:cancellable}
	A $k$-handle $h_k$ and a $(k+1)$-handle $h_{k+1}$ are \textbf{cancellable} if the attaching sphere of $h_{k+1}$ intersects the remaining sphere (the belt sphere) transversely in one point.
\end{definition}
The terminology stems from the fact that $M \cup h_k \cup h_{k+1} \cong M$ if $h_k$ and $h_{k+1}$ are cancellable.
This diffeomorphism is called \emph{cancelling} the handle pair.
\begin{theorem}[{E.g. \cite[Theorem 4.2.12]{GompfStipsicz}}]
	\label{thm:handle moves}
	Two handle decompositions of the same manifold are related by a finite sequence of handle attachment map isotopies,
	attachment sequence reorderings,
	and handle pair cancellations.
\end{theorem}
\begin{remark}
	A $k$-handle can be isotoped over another $k$-handle in a canonical way,
	this is called a \textbf{handle slide} \cite[Definition 4.2.10]{GompfStipsicz}.
	Since the attaching order for handles of the same level can be changed arbitrarily,
	any $k$-handle can be slid over any other $k$-handle.
\end{remark}

\begin{theorem}[{E.g. \cite[Proposition 4.2.13]{GompfStipsicz}}]
	\label{thm:One 0-handle and one n-handle}
	A connected, closed, smooth $n$-manifold has a handle decomposition with exactly one 0-handle and exactly one $n$-handle.
\end{theorem}
The idea of the proof is to cancel all 0-1-pairs until a single 0-handle remains,
and dually for the $n$-handle.

In general, it is not clear how to visualise intricate handle attachment maps.
In four dimensions though, all attachments happen inside three-dimensional spaces,
which allows us to depict them diagrammatically.

\subsection{Kirby diagrams}

For the remainder of the article,
we will assume handle decompositions to be in dimension $n = 4$.
The relevant special cases for 4-dimensional handles and their boundaries are listed in Table \ref{table:handles and their boundaries}.
\begin{table}
	\begin{center}
		\begin{tabular}{llll}
			\toprule
			$k$
				& Handle $h_k$
					& Attaching region $\attachingboundary h_k$
						& Remaining region $\remainingboundary h_k$
			\\\midrule
			0
				& $D^0 \times D^4$
					& $\emptyset$
						& $S^3 \cong \R^3 \cup \{\infty\}$
			\\
			1
				& $D^1 \times D^3$
					& $S^0 \times D^3 \cong \{-1, 1\} \times D^3$
						& $D^1 \times S^2 \cong [-1, 1] \times S^2$
			\\
			2
				& $D^2 \times D^2$
					& $S^1 \times D^2$
						& $D^2 \times S^1$
			\\
			3
				& $D^3 \times D^1$
					& $S^2 \times D^1$
						& $D^3 \times S^0$
			\\\bottomrule
		\end{tabular}
	\end{center}
	\caption{4-dimensional $k$-handles and their boundary,
		for $k \leq 3$.}
	\label{table:handles and their boundaries}
\end{table}

\fxwarning{Naming convention: attaching/remaining regions vs. boundaries}
\subsubsection{Remaining regions as canvases}
\label{sec:canvases}

It is possible to visualise handle decompositions of 4-manifolds as Kirby diagrams
by thinking of the remaining regions of the handles as \emph{drawing canvases}.
In fact, the remaining region of a 0-handle can be visualised as $\R^3$ with an additional point at infinity
(Figure \ref{fig:remaining region 0-handle}).
\begin{figure}
	\centering
	\begin{tikzpicture}
		\pic  at (-2, 0) { justaxes };
		\pic             { axesboundary = {$\remainingboundary h_0$} };
		\node at ( 3, 1) {$\cup \{\infty\}$};
	\end{tikzpicture}
	\caption{The remaining region of a 0-handle as a drawing canvas.}
	\label{fig:remaining region 0-handle}
	\fxwarning[layout={inline}]{It's a little bit pointless if we don't draw into it?
		Or is it ok because we draw into such a thing later?}
	\fxwarning[layout={inline}]{
		Is it confusing that the axes never show up again later?
		There is a slight inconsistency when we use them.
		We could say at some point that from now on,
		we won't use them anymore.}
	\fxwarning[layout={inline}]{The axes can come out much too light in print}
\end{figure}
Since $\attachingboundary h_k \cup \remainingboundary h_k \cong S^3$,
the remaining regions of higher handles can be visualised in the same way,
although with the attaching region removed from the canvas.

Initially, it is always possible to attach a 0-handle to the empty manifold,
which instantiates a new, empty drawing canvas.
Higher handles for $1 \leq k \leq 3$ are attached to the existing handlebody
by drawing their attaching region (Figure \ref{fig:illustrations attaching regions}) on the existing canvases,
which corresponds to embedding it onto the boundary of the handlebody.
\begin{figure}
	\centering
	\begin{tabular}{lc}
		\toprule
		Handle index $k$
			& Attaching region $\attachingboundary h_k$
		\\\midrule
		0
			& $\emptyset$
		\\\midrule
		1
			& \begin{tikzpicture}
				\node [ 1-handle ] at (-1, 0) {+};
				\node [ 1-handle ] at ( 1, 0) {$-$};
			\end{tikzpicture}
		\\\midrule
		2
			& \begin{tikzpicture}
				\pic {2-handle with crosssection};
			\end{tikzpicture}
		\\\midrule
		3
			& \begin{tikzpicture}
				\draw [ 3-handle curved ]
					circle [ radius = 1 ];
			\end{tikzpicture}
		\\\bottomrule
	\end{tabular}
	\caption{Attaching regions of 1-, 2- and 3-handles.}
	\label{fig:illustrations attaching regions}
\end{figure}
\fxwarning{Might we want to illustrate the remaining regions,
	except for the 0-handle (which can be done separately)?
	Or is it enough to say that it's S³ without the attaching region?
	One advantage is if we can illustrate the canonical attachment in the remaining region.
	Another is showing the belt sphere.
	Figure \ref{fig:pushing 1-handle} is already enough for the 1-handle.
	Maybe add 2-handle belt sphere (e.g. in 2-3 cancellation) and done?
}
Here, it will be enough to simply draw the attaching \emph{sphere},
as will be justified shortly.

\fxwarning{
	Here figure demonstrating black board convention for 2-handles?
	Or at least just a thickly drawn 2-handle (so that nobody expects the full torus)
}

As in Theorem \ref{thm:One 0-handle and one n-handle},
a 1-handle which is attached to two different 0-handles can cancel one of them
(Figure \ref{fig:0-1-cancellation}),
merging the two canvases.
If further handles are attached to the cancelling 0-handle,
their attachment maps can be isotoped across the (remaining boundary of the) 1-handle into the (remaining boundary of the) other 0-handle.
We will therefore usually assume handle decompositions to have exactly one 0-handle,
and omit the coordinate axes.
(By the same theorem, we will assume exactly one 4-handle to be present,
and never explicitly specify it.)
Examples of higher cancellations and slides will be reviewed in Section \ref{sec:handle moves}.
\begin{figure}
	\centering
	\begin{tikzpicture}
		\pic             { axes={$\remainingboundary h_{0, A}$} };
		\pic [ withRP2labels ]
			at (2,  2.2) {RP2};
		\node [ 1-handle, shading = 1handleshadingcyan ] (h1+) {+};
		\node [ below = 0.4cm ] at (h1+) {$\attachingboundary^+ h_{1,\beta}$};

		\begin{scope} [ xshift = 7cm ]
		\pic             { axes={$\remainingboundary h_{0, B}$} };
		\pic [withS2S2labels]
			at (2,  2.2) {S2S2twisted};
		\node [ 1-handle, shading = 1handleshadingcyan ] (h1-) {$-$};
		\node [ below = 0.4cm ] at (h1-) {$\attachingboundary^- h_{1,\beta}$};
		\end{scope}
	\end{tikzpicture}

	\bigskip

	\begin{tikzpicture}
		\draw [ -> ] (-7, 1.5) -- node [ above ] {0-1-cancellation} +(3, 0);

		\pic                 { axes = {$\remainingboundary h_{0, A}$} };
		\pic  at (2,  2) {RP2};
		\pic  at (2,  0) {S2S2twisted};
	\end{tikzpicture}
	\caption{The handles $h_{0,B}$ and $h_{1,\beta}$ cancel each other.
		The 0-1-cancellation move merges two drawing canvases.
	}
	\label{fig:0-1-cancellation}
\end{figure}

A handle can be attached onto several other handles of smaller index,
in which case its attaching sphere is spread over several remaining regions,
as for example in Figure \ref{fig:RP3}.

\begin{figure}[!htp]
	\centering
	\begin{tikzpicture}
		\newdimen\threehandleradius
		\tikzmath{
			\myangle           = 10;
			\threehandleradius = 2cm;
		}
		\begin{scope} [ xshift = 3.5cm ]
		\pic { axes = {$\remainingboundary h_0$}};
		\coordinate (centre) at (1, 1);
		\path [ 3-handle curved ] (centre) circle [ radius = \threehandleradius ];
		\pic at (centre) { RP2 };
		\end{scope}

		\begin{scope} [ yshift = -6cm ]
		\pic { axes = {$\remainingboundary h_1$}};
		\begin{scope} [ shift = {(1, 1)} ]
		\draw [ thick, looseness = 0.8 ]
			(180-\myangle:\threehandleradius)
				to [ out = 90 - \myangle, in = 180            ]
			(0.5, 1)
				to [ out =  0           , in =  90 + \myangle ]
			(\myangle:2cm) 
		;
		\path [ 3-handle curved ] circle [ radius = 2cm ];
		\node (hd1) [ 1-handle-shape, forbiddenregion ] at (-1, 0) {};
		\node (hd2) [ 1-handle-shape, forbiddenregion ] at ( 1, 0) {};
		\draw [ thick ]
			(hd1) -- (hd2)
			(hd1)
				to [ out = 180, in = 270 - \myangle ]
			(180-\myangle:2cm)
			(hd2)
				to [ out =   0, in = 270 + \myangle ]
			(    \myangle:2cm)
		;
		\end{scope}

		\begin{scope} [ xshift = 7cm ]
		\pic { axes = {$\remainingboundary h_2$}};
		\begin{scope} [ shift = {(1, 1)} ]
		\draw [ forbidden 2-handle, looseness = 0.8 ]
			(180-\myangle:\threehandleradius)
				to [ out = 90 - \myangle, in  = 180            ]
			(0.5, 1)
				to [ out =  0           , in  =  90 + \myangle ]
			(\myangle:2cm)
				to [ in  =  0           , out = 270 + \myangle ]
			(0, -0.5)
				to [ out = 180          , in  = 270 - \myangle ]
			(180-\myangle:2cm)
		;
		\path [ 3-handle curved ] circle [ radius = 2cm ];
		\clip [ looseness = 0.8 ]
			(-3, 1) --
			(180-4*\myangle:2cm)
				arc
					[ start angle = 180 - 4 * \myangle
					, end angle   = 180 -     \myangle
					, radius      = \threehandleradius
					]
			(180-\myangle:2cm)
				to  [ in  = 180, out = 270 - \myangle ]
			(0, -0.5)
				to  [ out =   0, in  = 270 + \myangle ]
			(    \myangle:2cm)
				arc
					[ start angle =     \myangle
					, end angle   = 4 * \myangle
					, radius      = \threehandleradius
					]
				-- (3, 1) -- (3, -2) -- (-3, -2) -- (-3, 1)
		;
		\draw [ forbidden 2-handle, looseness = 0.8 ]
			(180-\myangle:\threehandleradius)
				to [ out = 90 - \myangle, in  = 180            ]
			(0.5, 1)
				to [ out =  0           , in  =  90 + \myangle ]
			(\myangle:2cm)
				to [ in  =  0           , out = 270 + \myangle ]
			(0, -0.5)
				to [ out = 180          , in  = 270 - \myangle ]
			cycle
		;
		\end{scope}
		\end{scope}
		\end{scope}
	\end{tikzpicture}
	\caption{
		A handle decomposition of $[-1, 1] \times \RP 3$ where each $k$-handle is attached to every $j$-handle for $j < k$.
		In particular, the single 3-handle is attached to $h_1$ and $h_2$:
		Its attaching sphere touches the attaching spheres of $h_1$ and $h_2$ in $\partial h_0$
		(and $h_2$'s attaching sphere further in $\remainingboundary h_1$),
		and enters their respective remaining regions,
		partitioning it into six pieces.
	}
	\fxwarning[layout={inline}]{
		There is a subtlety with orientations here. Should we add that?
		Only if they are unintuitive
		(i.e. if the 3-handle orientation jumps on the different parts).
		Also, should we somehow show the way how the different parts are partitioned?
	}
	\label{fig:RP3}
	\fxnote[layout={inline}]{
		Isn't this confusing since we said that all mfds are closed from here on?
		No we didn't say that
	}
	\fxnote[layout={inline}]{
		Is this maybe $\RP 4$?
		No, that's not orientable but the 1-handle is attached in an orientable way.
	}
\end{figure}

\subsubsection{Attaching spheres and framings}

Up to isotopy, 1-handle and 3-handle attachments are determined by the embedding of the attaching sphere:
\fxwarning{Decide whether in general its = or $\cong$ in other situations.
	I guess it's actually often = because we really mean the standard handle.
}
An embedding of $\attachingboundary h_1 = S^0 \times D^3$ is,
up to isotopy,
specified by the embedding of the two points of $S^0$.

Similarly, a 3-handle attachment is essentially specified by the embedding of an $S^2$ \cite[Example 4.1.4 c]{GompfStipsicz}.
Since usually further handles are attached to 1-handles,
but not to 3-handles,
we will draw the complete attaching regions of 1-handles,
but only the attaching spheres of 3-handles,
as in Figure \ref{fig:illustrations attaching regions}.

Furthermore, it is sometimes helpful to mark the orientation of $S^0$ in $\attachingboundary h_1 = S^0 \times D^3$ by denoting one point as $+$ and the other as $-$,
and consequently denoting $\attachingboundary h_1 \cong \attachingboundary^+ h_1 \sqcup \attachingboundary^- h_1 = D^3_+ \sqcup D^3_-$.\footnote{%
	This is not to be confused with the independent notion of relative handle decompositions.
}
Similarly, the attaching $S^2$ of a 3-handle may carry an orientation.
The result of the attachment does not depend on the choice of orientation.

A 2-handle is attached along $\attachingboundary h_2 = S^1 \times D^2$,
and an embedding $S^1 \hookrightarrow \R^3$ can be knotted,
or even linked to other 2-handles.
The attachment carries more information than the mere embedding of its attaching $S^1$, though:
It may be \emph{framed}, that is, the $D^2$ component can be twisted by any integer multiple of a full turn.
It is possible to denote framings by labelling 2-handle attachments with integers,
but for our purposes it is more pragmatic to stipulate the \textbf{blackboard conventions}:
Our diagrams are not truly drawn in $\R^3$, but in a projection onto $\R^2$,
and this projection specifies a canonical framing on every curve.
For further details, we refer to Appendix \ref{sec:blackboard framing}.

\subsubsection{Kirby conventions}

Drawing the attaching spheres inside all remaining regions is uneconomical,
but fortunately not necessary.
Assume we have already attached a $k$-handle $h_k$ to a manifold ($k \in \{1, 2\})$,
and then attach a $j$-handle $h_j$ for $j > k$ gradually,
starting to draw the attaching sphere in $\remainingboundary h_0$ and aiming to continue into $\remainingboundary h_k$.
Then the attaching sphere of $h_j$ will eventually intersect the boundary of the attaching region of $h_k$, $\partial \attachingboundary h_k \cong S^{k-1} \times S^{3-k}$.
This intersection already canonically determines an attachment inside $\attachingboundary h_k$,
as we will see shortly.
By always assuming this convention,
the part of the diagram inside $\remainingboundary h_0$ completely determines the handle attachment.

\fxnote{Isn't there a high brow way of distinguishing convention-following attachments from all attachments?}
\fxwarning{This stuff about convention-following is maybe easier if we squash the remaining region so it is arbitrarily small?}

\begin{definition}
	\label{def:regular handle decomposition}
	A handle decomposition is \textbf{regular} iff:
	\begin{enumerate}
		\item The attachment of a $k$-handle is outside the remaining regions of other $k$-handles.
		\item No attachment intersects with $\infty \in \remainingboundary h_0 \cong \R^3 \cup \{\infty\}$.
	\end{enumerate}
\end{definition}

\begin{remarks}
	\begin{enumerate}
		\item This condition is equivalent to the requirement $k$-handles are attached only onto the remaining regions of $j$-handles for $j < k$ (strictly).
		\item This condition ensures that the diagram can be drawn in a bounded region of $\R^3$.
	\end{enumerate}
\end{remarks}

\begin{definition}
	\label{def:single picture conventions}
	A handle decomposition satisfies the \textbf{single-picture conventions} iff:
	\begin{description}
		\item [2-1]
			Inside the remaining region of any 1-handle $\remainingboundary h_1 = D^1 \times S^2 = [-1,1] \times S^2$,
			images of attaching maps of 2-handles
			are of the form $[-1,1] \times \{p_1, \dots, p_N\}$ with $p_i \in S^2$.
		\item [3-1]
			Inside the remaining region of any 1-handle $\remainingboundary h_1 = D^1 \times S^2 = [-1,1] \times S^2$,
			images of attaching maps of 3-handles
			are of the form $[-1,1] \times A$ with $A \subset S^2$
			a compact 1-dimensional submanifold.
		\item [3-2]
			Inside the remaining region of any 2-handle $\remainingboundary h_2 = D^2 \times S^1$,
			images of attaching maps of 3-handles are of the form $D^2 \times \{p_1, p_2, \dots p_N\}$,
			where $p_i \in S^1$.
	\end{description}
\end{definition}

\begin{remark}
	\label{rem:single picture}
	These two conditions have intuitive geometrical interpretations,
	on which we will expand:
	\begin{description}
		\item [2-1]
			A 2-handle entering one attaching ball $D^3_+$ of a 1-handle at a set of points
			must leave the corresponding ball $D^3_-$ on the mirror positions.
			This is a standard assumption which is usually made in text books.
		\item [3-1]
			Analogously, a 3-handle entering one attaching ball $D^3_+$ of a 1-handle at a submanifold $A$
			must leave the corresponding ball $D^3_-$ on the reflection of $A$.
		\item [3-2]
			A 3-handle enters a 2-handle along $S^1 \times \{p_1, p_2, \dots p_N\} \subset S^1 \times S^1$,
			where each $S^1 \times \{p_i\}$ follows the framing of the 2-handle attachment.
			To the author's knowledge, this convention is unmentioned in the literature,
			although it is straightforward.
	\end{description}
\end{remark}
\fxwarning{It would be good to have an example of a 3-handle attached to a 1-framed 2-handle.}

\begin{definition}
	\label{def:kirby diagram}
	A regular handle attachment satisfying the single-picture conventions is called a \textbf{Kirby diagram}.
\end{definition}

\begin{figure}
	\centering
	\begin{tikzpicture}
		\pic           { axes             = {$\remainingboundary h_0$} };
		\pic [ showisotopyreverse ]
		     at (0, 2) { straight12handle = {1-handle-shading} };

		\begin{scope} [ xshift =  7cm ]
		\pic           { axes             = {$\remainingboundary h_1$} };
		\pic [ showisotopy ]
		               { knotted12handle  = {draw, dashed, fill = gray!30!white} };
		\end{scope}

		\begin{scope} [ yshift = -6cm ]
		\draw [->] (-4.5, 2) -- node [ above ] {Isotopy} +(1.5, 0);

		\pic           { axes             = {$\remainingboundary h_0$} };
		\pic [ xscale = -1, transform shape ]
		     at (0, 2) { knotted12handle  = {1-handle-shading} };

		\begin{scope} [ xshift =  7cm ]
		\pic           { axes             = {$\remainingboundary h_1$} };
		\pic [ drawadditionalconventionalintervals ]
		               { straight12handle = {forbiddenregion} };
		\end{scope}
		\end{scope}
	\end{tikzpicture}
	\caption{Pushing a 2-handle attachment outside the 1-handle remaining region.
		After the isotopy,
		the 2-handle attaching sphere runs along a canonical interval inside of $\remainingboundary h_2$.
		(Other canonical intervals are drawn dashed.)
		It is then easy to see that the 2-handle cancels the 1-handle,
		since it intersects the belt sphere
		(the $y-z$-plane compactified at $\infty$)
		transversely in one point.
	}
	\label{fig:pushing 1-handle}
	\fxwarning[layout={inline}]{The coordinate system $z$-angle is different in different figures. Unify?}
	\fxwarning[layout={inline}]{Put ``Isotopy'' label centered? Label $D^3$s?}
\end{figure}
It is natural to ask whether essentially every handle decomposition can be drawn as a Kirby diagram.
This is ensured by the following lemma.
\begin{lemma}
	\label{lem:establish Kirby conventions}
	Given any handle decomposition,
	a succession of isotopies of the individual handle attachments can be applied
	to arrive at a Kirby diagram.
\end{lemma}
\begin{proof}
	When disregarding 3-handles,
	this is a standard fact.
	See Figure \ref{fig:pushing 1-handle} for an illustration.
	The full proof, covering the 3-handle case,
	is given in Appendix \ref{sec:details on diagrams}.
\end{proof}

\subsubsection{Kirby diagrams and 3-handles}

Usually, 3-handles are not depicted in Kirby diagrams, due to the following theorem:
\begin{theorem}[{Well-known, e.g. \cite[Section 4.4]{GompfStipsicz}}]
	\label{thm:need only 2-handlebody}
	Let $M$ and $N$ be smooth, closed, oriented 4-manifolds,
	with handle decompositions containing precisely one 0-handle and one 4-handle.
	Assume that their 2-handlebodies agree.
	Then $M \cong N$.
\end{theorem}
In essence, the 3-handle attachments contain no information in a handle decomposition of a closed manifold.
4-dimensional differential topology usually concerns itself with closed manifolds,
and 3-handles are typically excluded from most discussions.
Standard references for Kirby diagrams such as \cite{GompfStipsicz} and \cite{Akbulut:4manifolds} emphasize 1-handle and 2-handle attachments on a single 0-handle.
\fxwarning{This point needs to be broadened by a definition,
	a picture or maybe both.}

While it is true that the 2-handles contain the main complexity of a handle decomposition,
3-handles will still turn out to contain useful information for the computation of the invariant defined in Section \ref{sec:invariant}.

As a further reason whose ramifications are beyond the scope of this article,
Theorem \ref{thm:need only 2-handlebody} does not hold for manifolds with nonempty boundary.
When our aim is to describe the whole TQFT directly with handle decompositions
(without taking the detour over the state sum model),
we will eventually have to regard bordisms,
and there are indeed nondiffeomorphic manifolds with boundary that have handle decompositions with the same 2-handlebody\footnote{%
	The author is grateful to Marco Golla and Andrew Lobb to point this fact out here:
	\href{https://mathoverflow.net/questions/288246/are-there-kirby-diagrams-with-3-handles}{https://mathoverflow.net/questions/288246/are-there-kirby-diagrams-with-3-handles}}.
Thus, 3-handles may indeed contain relevant information in this situation.
This viewpoint is discussed briefly in Section \ref{sec:TQFT directly}.

\fxnote{
	Don't need to review their role in CY because there we really only attach data to the 2-handlebody,
	even in the boundary case.
}

\subsection{Handle moves}
\label{sec:handle moves}

Handle decompositions are by far not unique.
Fortunately, Theorem \ref{thm:handle moves} ensures that every two finite handle decompositions of the same closed manifold are related by a finite sequence of handle moves.
There are three kinds of moves to consider:
Handle cancellations,
which are described below in Section \ref{sec:cancellations},
and (ambient) isotopies of handle attachments
as well as reorderings of attachments of the same index,
both of which are described subsequently in Section \ref{sec:slides}.

From here on, we will not draw the coordinate grids of the remaining regions anymore,
and assume that all diagrams take place in $\remainingboundary h_0$
(unless specified otherwise).

\fxnote{Can't cite Juhasz easily because he doesn't do it self-indexed.}

\subsubsection{Cancellations}
\label{sec:cancellations}

In our Kirby diagrams,
cancellable pairs (Definition \ref{def:cancellable})
of a $k$- and a $(k+1)$-handle are relevant for $k \in \{1, 2\}$.
(We have described 0-1-cancellation already
in Figure \ref{fig:0-1-cancellation} in Section \ref{sec:canvases},
and assume that there is a single 4-handle.)

A 2-handle cancels a 1-handle if the attaching circle attaches in exactly one point for each 3-ball of the 1-handle.
This can be seen in Figure \ref{fig:pushing 1-handle},
where it was argued that any such attachment can be isotoped to one that intersects the belt sphere in one point.
If further handles are attached to the handle pair,
they can be slid off first before cancelling \cite[Figure 5.13]{GompfStipsicz}.

A 3-handle cancels a 2-handle if the 2-handle is attached along an unframed unknot,
and (the visible part of the) 3-handle attachment forms a a disk which is bounded by said unknot.
(The second half of the attaching $S^2$ is inside the remaining region of the 2-handle,
 where it intersects the belt sphere in one point.)
The situation can be seen in Figure \ref{fig:2-3-cancellation}.
\begin{figure}
	\centering
	\begin{tikzpicture}[ baseline ]
		\draw [ thick, 3-handle fill ] circle [ radius = 1cm ];
	\end{tikzpicture}
	$\qquad \xmapsto{\text{2-3-cancellation}} \qquad \emptyset$
	\caption{The visible part of the 2-3-handle cancellation.
		The 3-handle attaching sphere is split into two disks,
		one of which vanishes inside the 2-handle remaining region.
	}
	\label{fig:2-3-cancellation}
	\fxwarning[layout={inline}]{Sthg. about orientations?
		Is the picture clear enough?
		Explanation?}
\end{figure}
\fxwarning{Reference to paragraph or so when that part is better sorted}
As for the 1-2-cancellation,
it is sometimes necessary to first slide off other handles before cancelling.

\fxwarning{Possibly illustrate all this and other points in a big example with all relevant moves?}
\fxnote{What about 3-4?
	Is it allowed to attach cancelling 3-handles?
	No. Single 4-handle.
	How to see whether a 3-handle is cancellable?
	Doesn't matter.
	(I think it's a contractible $S^2$).
}

\subsubsection{Slides}
\label{sec:slides}

Applying an isotopy to a constituent of a Kirby diagram (tautologically) applies an isotopy to the corresponding handle attachment,
but there may be isotopies of handle attachments that do not conform to the Kirby conventions in the intermediate stages of the isotopy.
Indeed, isotopies of handle attachment between two non-isotopic Kirby diagrams exist.
They are commonly called ``handle slides'',
and they can be decomposed into elementary $j$-$k$-slides where $j \geq k$.
A $j$-$k$-slide is an isotopy of a $j$-handle over the remaining region of a $k$-handle that changes the Kirby diagram.

If an $l$-handle is attached to a $j$-handle,
the $j$-$k$-slide can still be performed simultaneously with an $l$-$k$-slide.

In order to allow for arbitrary $k$-$k$-slides,
the attachment order of $k$-handle attachments has to be changed sometimes.
Subtly, Kirby diagrams do not specify the order
since the attachments always commute due to the first regularity condition
(Definition \ref{def:regular handle decomposition}),
and thus resulting manifolds for different orders are diffeomorphic.
When isotoping a $k$-handle $h_k$ over another,
the latter has to be attached first,
thus a reordering may be necessary.

\fxnote{
	Isn't there a high-brow way doing this involving relative homology or so?
	No, those are knots.
}

$j$-$k$-slides for $j \leq 2$ are treated at length in the literature,
see e.g. \cite[Section 5.1]{GompfStipsicz}.
For completeness, they are briefly described in the following,
enumerated as $j$-$k$:
\begin{figure}
	\centering
			\tikzmath{
				\width  = 2.7;
				\height = 2;
				\slidex = 0.5 * \width;
				\slidey = 1.5;
				\dangle = 20;
				\alphashift = 0.3;
			}
			\begin{tikzpicture} [ baseline ]
				\pgfdeclarelayer{threehandles}
				\pgfsetlayers{threehandles,main}
				\node [ 1-handle ] (alpha-) at (-\slidex, -\slidey) {$D^{3-}_\alpha$};
				\node [ 1-handle ] (beta-)  at ( \slidex, -\slidey) {$D^{3-}_\beta$};
				\begin{scope}[horizontally]
					\path [ 3-handle fill = roch, 3-handle extends = roch ]
						(-\width, -\height)
						rectangle
						( \width,  \height);
				\end{scope}
				\node [ 1-handle ] (alpha+) at (-\slidex,  \slidey) {$D^{3+}_\alpha$};
				\node [ 1-handle ] (beta+)  at ( \slidex,  \slidey) {$D^{3+}_\beta$};
				\draw [ thin, ->, shorten > = 1mm, shorten < = 2mm ]
					(alpha+) to [ out =  20, in = 100 ] (beta+)
				;
				\begin{pgfonlayer}{threehandles}
				\draw [ 3-handle strip ]
					  (alpha-.270)
						--
					++(0, -\slidey)
				;
				\draw [ 3-handle strip ]
					  (alpha+.90)
						--
					++(0,  \slidey)
				;
				\end{pgfonlayer}
			\end{tikzpicture}
			$\xmapsto{j-1}$
			\begin{tikzpicture} [ baseline ]
				\pgfsetlayers{threehandles,main}
				\node [ 1-handle ] (alpha+) at (-\slidex,  0.8 - \slidey) {$D^{3+}_\alpha$};
				\node [ 1-handle ] (alpha-) at (-\slidex, -\alphashift - \slidey) {$D^{3-}_\alpha$};
				\node [ 1-handle ] (beta-)  at ( \slidex, -\slidey) {$D^{3-}_\beta$};
				\begin{pgfonlayer}{threehandles}
				\draw [ 3-handle strip ]
					  (alpha-.270)
						--
					++(0, -\slidey + \alphashift)
				;
				\draw [ 3-handle strip ]
					(alpha+)
						to [ out =  90, in = 270 ]
					(beta-)
				;
				\end{pgfonlayer}
				\begin{scope}[horizontally]
					\path [ 3-handle fill = roch, 3-handle extends = roch ]
						(-\width, -\height)
						rectangle
						( \width,  \height);
				\end{scope}
				\node [ 1-handle ] (beta+)  at ( \slidex,  \slidey) {$D^{3+}_\beta$};
				\begin{pgfonlayer}{threehandles}
				\draw [ 3-handle strip ]
					($(beta+.90) +(-2*\slidex, \slidey)$)
						to [ out = 270, in =  90 ]
					(beta+)
				;
				\end{pgfonlayer}
			\end{tikzpicture}
			\caption{1-1-handle slide with simultaneous 2-1-slide and 3-1-slide.
				The red 3-handle attachment separates the two components $\attachingboundary h_1 \cong D^{3+} \sqcup D^{3-}$
				of the 1-handle attachments,
				but $D^{3+}_\alpha$ can still ``tunnel'' through the remaining region of $h_{1\beta}$.
			}
			\label{fig:k-1-slides}
			\fxwarning[layout={inline}]{Do I like the dashed boundary for "extends?"
				If no, try something else. (sparsely dotted?)
				If yes, use everywhere.
			}
\end{figure}
\begin{figure}
			\centering
			\newdimen\radius
			\newdimen\sradius
			\newdimen\redmargin
			\newdimen\greywidth
			\newdimen\twotwox
			\tikzmath{
				\twotwox         = 3cm;
				\endangle  = 90; 
				\radius    = 1.5cm;
				\sradius   = 0.5cm;
				\width     = 5;
				\height    = 2;
				\redmargin = 1cm;
				\greywidth = 0.2cm;
				\diagramwhitesep = 0.4pt; 
			}
			\begin{tikzpicture} [ baseline ]
				\fill [ gray ]
					(\twotwox-\greywidth, 0)
					rectangle
					(\twotwox+\greywidth, -\height);
				\begin{scope}[horizontally]
					\newcommand{\slidinghandlepath}{
						( canvas polar cs:
						  radius = \radius
						, angle  = -\endangle
						)
						arc
						[ start angle = -\endangle
						, end angle   =  \endangle
						, radius      =  \radius
						]
					}
					\draw [ thindiagram ] \slidinghandlepath;
					\newcommand{\slidoverhandlepath}{ (\twotwox, 0) circle [ radius = \sradius ] }
					\draw [ thindiagram ] \slidoverhandlepath;
					\path [ 3-handle fill ] \slidinghandlepath;
					\path [ 3-handle fill = roch ]
						\slidinghandlepath -- ++(0, \redmargin) -- ++(\width, 0) -- ++(0, -2*\redmargin-2*\radius) -- ++(-\width, 0) -- cycle
						\slidoverhandlepath;
					\node at (0.6, 0) {$\attachingboundary h_{3A}$};
					\node at (0.8*\width, \radius) {$\attachingboundary h_{3B}$};
					\node [ below right = -1mm ] at (-45:\radius) {$\attachingboundary h_{2a}$};
					\node [ below right = -1mm ] at (\twotwox + \sradius, 0) {$\attachingboundary h_{2b}$};
				\end{scope}
				\fill [ white ]
					(\twotwox-\greywidth-\diagramwhitesep, 0)
					rectangle
					(\twotwox+\greywidth+\diagramwhitesep, \height);
				\fill [ gray ]
					(\twotwox-\greywidth, 0)
					rectangle
					(\twotwox+\greywidth, \height);
			\end{tikzpicture}
			$\xmapsto{j\text{-2-slide}}$
			\begin{tikzpicture} [ baseline ]
				\fill [ gray ]
					(\twotwox-\greywidth, 0)
					rectangle
					(\twotwox+\greywidth, -\height);
				\begin{scope}[horizontally]
					\newcommand{\slidinghandlepath}{
						( canvas polar cs:
						  x radius = \radius + \twotwox
						, y radius = \radius
						, angle    = -\endangle
						)
						arc
						[ start angle = -\endangle
						, end angle   =  \endangle
						, x radius = \radius + \twotwox
						, y radius = \radius
						]
					}
					\draw [ thindiagram ] \slidinghandlepath;
					\newcommand{\slidoverhandlepath}{ (\twotwox, 0) circle [ radius = 0.5cm ] }
					\draw [ thindiagram ] \slidoverhandlepath;
					\path [ 3-handle fill, even odd rule ] \slidinghandlepath \slidoverhandlepath;
					\path [ 3-handle fill = roch ]
						\slidinghandlepath -- ++(0, \redmargin) -- ++(\width, 0) -- ++(0, -2*\redmargin-2*\radius) -- ++(-\width, 0) -- cycle;
					\node at (0.6, 0) {$\attachingboundary h_{3A}$};
					\node at (0.8*\width, \radius) {$\attachingboundary h_{3B}$};
					\node [ below right = -1mm ] at
						( canvas polar cs:
						  angle = -35
						, x radius = \radius + \twotwox
						, y radius = \radius
						) {$\attachingboundary h_{2a}$};
					\node [ left ] at (\twotwox - \sradius, 0) {$\attachingboundary h_{2b}$};
				\end{scope}
				\fill [ white ]
					(\twotwox-\greywidth-\diagramwhitesep, 0)
					rectangle
					(\twotwox+\greywidth+\diagramwhitesep, \height);
				\fill [ gray ]
					(\twotwox-\greywidth, 0)
					rectangle
					(\twotwox+\greywidth, \height);
			\end{tikzpicture}
			\\
			\caption{2-2-handle slide of $h_{2a}$ over $h_{2b}$,
				simultaneously with a 3-2-slide of $h_{3A}$ over $h_{2b}$.
				The grey area can contain arbitrary handles,
				and may share 3-handles with $h_{2b}$.
				Furthermore, $h_{2b}$ may be arbitrarily knotted,
				and linked to other 2-handles.
				$h_{2a}$ then follows those knots and links.
			}
	\label{fig:k-2-slides}
\end{figure}

\begin{description}
	\item[j-1] To slide any handle $h_j$ over a 1-handle $h_1$,
		choose a path from the attached $\attachingboundary h_j$ to $\attachingboundary h_1$ that does not intersect with any other attachments
		(in particular no 3-handle attachments)
		and move $h_j$'s attaching sphere along into $\remainingboundary h_1$,
		entering at one of the $S^2$s bounding an attaching ball of $h_1$.

		Since the two attaching balls $D^3_+$ and $D^3_-$ can be situated in two regions of the drawing canvas that are separated by a 3-handle attachment,
		it is possible to move attachments between separated regions by means of this move.
	\item[1-1] Here, regularity requires the attachment to move all the way through the remaining boundary,
			leaving the opposite attaching ball again.
	\item[2-1] To establish the Kirby conventions,
		the 2-handle attachment needs not move completely through $\remainingboundary h_1$,
		instead its ends can protrude from the attaching balls of $h_1$ at mirror positions.
		The situation is displayed in Figure \ref{fig:RP3}.
	\item[2-2]
		The pair of pants surface is defined as a two-dimensional disk $D^2_c$ with two open disks $D^{2,\circ}_a$ and $D^{2,\circ}_b$ removed,
		its boundary is thus $S^1_a \sqcup S^1_b \sqcup S^1_c$.
		Given an embedding of the pair of pants in a Kirby diagram $K$,
		\fxwarning{It's actually in the remaining region of the handle decomposition,
		i.e. the 2-handles can be attached to other stuff as well}
		its boundary embedding defines three 2-handle attachments $h_{2,a}$, $h_{2,b}$ and $h_{2,c}$.
		The 2-2-slide transforms $K$ together with the attachments $h_{2,a}$ and $h_{2,b}$ into $K$ with $h_{2,c}$ and $h_{2,b}$,
		and one says that $h_{2,a}$ has been slid over $h_{2,b}$.

		For any two 2-handles, it is always possible to slide one over the other.
		(If they are in different regions separated by 3-handles,
		 it is still possible to perform the 2-2-slide after a series of 2-1-slides,
		 assuming the manifold is closed.)
		The resulting handle $h_{2,c}$ may depend on the choice of the pair of pants,
		but the resulting manifold does not
		(up to diffeomorphism).
\end{description}
Figures \ref{fig:k-1-slides} and \ref{fig:k-2-slides} display these slides,
but also include analogous 3-$k$-slides,
which are described in the following:
\begin{figure}
	\centering
		\newdimen\radius
		\newdimen\distance
		\tikzmath{
			\myangle  = 55;
			\radius   = 0.25cm;
			\distance = 3cm;
			\y        = 1;
			\slidy    = 5.5;
		}
		\begin{tikzpicture}
			\fill [ 3-handle fill ]
				(0, 0)
				-- (\myangle:\distance)
				-- ++(\y, 0)
				arc
					[ radius = \radius
					, start angle = 270
					, end angle   = 270 + \myangle
					]
				-- ++(180+\myangle:\distance)
				arc
					[ radius      = \radius
					, start angle = 270 + \myangle
					, end angle   = 270
					]
				-- (0, 0)
			;
			\fill [ 3-handle fill ]
				(0, 2 * \radius)
				-- ++(\myangle:\distance)
				-- ++(\y, 0)
				arc
					[ radius      = \radius
					, start angle = 90
					, end angle   = -\myangle
					]
				-- ++(180+\myangle:\distance)
				arc
					[ radius      = \radius
					, start angle = -\myangle
					, end angle   = 90
					]
				-- (0, 2 * \radius)
			;
			\draw [ 3-handle extends ]
				(0, 0)
				-- (\myangle:\distance)
				-- ++(\y, 0)
				arc
					[ radius      = \radius
					, start angle = -90
					, end angle   =  90
					]
				-- ++(-\y, 0)
				-- ++(180+\myangle:\distance)
				-- ++(\y, 0)
				arc
					[ radius      = \radius
					, start angle =  90
					, end angle   = -90
					]
				-- (0, 0)
			;
			\node at ($0.5*(\myangle:\distance)+0.5*(\y,0)+(\radius,\radius)$) {$\attachingboundary h_3$};

			\pic  at (5  , 1  ) {2-handle};
			\node at (5.6, 1.6) {$\attachingboundary h_2$};
		\end{tikzpicture}

		\vspace{1cm}
		$\xmapsto{\text{3-2-slide}}$
		\begin{tikzpicture} [ baseline = {(0, 2)} ]
			\coordinate (2handle) at (3.7, 1.2);
			\coordinate (lowerintersection) at ($(2handle)+(200:2cm and 1cm)$);
			\fill [ 3-handle fill, even odd rule ]
				(0, 0)
				-- (\myangle:\distance)
				-- ++(\slidy, 0)
				arc
					[ radius = \radius
					, start angle = 270
					, end angle   = 270 + \myangle
					]
				-- ++(180+\myangle:\distance)
				arc
					[ radius      = \radius
					, start angle = 270 + \myangle
					, end angle   = 270
					]
				-- (0, 0)
				(lowerintersection) arc
					[ x radius = 2cm
					, y radius = 1cm
					, start angle = 200
					, end angle   = -20
					]
				to [ out = 210, in = -30 ]
				(lowerintersection)
			;
			\fill [ 3-handle fill ]
				(0, 2 * \radius)
				-- ++(\myangle:\distance)
				-- ++(\slidy, 0)
				arc
					[ radius      = \radius
					, start angle = 90
					, end angle   = -\myangle
					]
				-- ++(180+\myangle:\distance)
				arc
					[ radius      = \radius
					, start angle = -\myangle
					, end angle   = 90
					]
				-- (0, 2 * \radius)
				($(2handle)+(0,0.1)$) ellipse
					[ x radius = 1.6
					, y radius = 0.65
					]
			;
			\draw [ 3-handle extends ]
				(0, 0)
				-- (\myangle:\distance)
				-- ++(\slidy, 0)
				arc
					[ radius      = \radius
					, start angle = -90
					, end angle   =  90
					]
				-- ++(-\slidy, 0)
				-- ++(180+\myangle:\distance)
				-- ++(\slidy, 0)
				arc
					[ radius      = \radius
					, start angle =  90
					, end angle   = -90
					]
				-- (0, 0)
			;

			\pic at (2handle) {2-handle};
		\end{tikzpicture}
		\\
		\fxwarning[layout={inline}]{
			Explain dashed convention as ``this may extend further''
			and use consistently elsewhere?
			The pancake analogy would want the $S^1$s where the 3-handle attaches to be marked (dotted?) as well.
			Labels on the second picture as well? But gets crammed.
			Update labels so they conform with the explicit model?
		}
	\caption{
		The 3-2-slide.
		Graphically, we can imagine the fold arc of the 3-handle attachment
		slides over the 2-handle.
		The 3-handle is understood to be extended past the dashed lines,
		to form a whole sphere.
	}
	\label{fig:3-2-slide}
\end{figure}
\begin{figure}
	\centering
		\begin{tikzpicture} [ baseline ]
			\path [ 3-handle curved ]             circle [ radius = 1 cm ];
			\node {$h_{3A}$};
			\path [ 3-handle curved roch ] (3, 0) circle [ radius = 1 cm ];
			\node at (3, 0) {$h_{3B}$};
		\end{tikzpicture}
		\quad
		$\xmapsto{\text{3-3-slide}}$
		\quad
		\begin{tikzpicture} [ baseline ]
			\path [ 3-handle curved roch ] (1,0) circle [ radius = 1 cm ];
			\node at ( 1, 0) {$h_{3B}$};
			\path [ 3-handle curved ] ellipse [ x radius = 2.5 cm, y radius = 1.5 cm ];
			\node at (-1, 0) {$h_{3C}$};
		\end{tikzpicture}
	\caption{In the 3-3-slide, one 3-handle attachment can slide over the other.}
	\label{fig:3-3-slide}
	\fxwarning{It's a bit confusing that here the slid handle is called $h_{3C}$ and not $h_{3A}$ anymore, but in the 2-2-slide they're called the same}
\end{figure}
\begin{description}
	\item[3-1]
		As in any $j$-1-slide, a 3-handle attachment can be isotoped through the remaining region of a 1-handle.
		And as in the 2-1-slide, the attachment can protrude from the attaching balls of the 1-handle at mirror positions.

		For an explicit model of this situation,
		realise that the attaching sphere of the 3-handle can be split in three parts,
		consisting of two caps and a cylinder:
		\[S^2 \cong D^2_+ \cup_{S^1_+} \left( S^1 \times [+1,-1] \right) \cup_{S^1_-} D_{2,-}\]
		Choose an embedding $S^1 \subset \partial D^3_+$ on the boundary of one ball of the 1-handle attachment.
		This induces a mirror embedding on the other ball.
		The 3-handle can attached as follows:
		The caps $D^2_\pm$ of the attaching sphere are in the drawing canvas of the 0-handle,
		bounding $S^1$ on $D^3_+$ and its mirror circle on $D^3_-$, respectively.
		The cylinder $S^1 \times [+1, -1]$ vanishes inside the remaining region of the 1-handle,
		conforming to the Kirby conventions.
		\fxnote{Below more material of the longer explanation.}
	\item[3-2]
		It is helpful to visualise the 3-2-slide by realising the attaching sphere of $h_3$ as the ``double pancake'' $S^2 \cong D^2 \cup_{S^1} D^2$,
		which consists of two disks glued along their boundary.
		The gluing $S^1$ is then slid over the attaching $S^1$ of the 2-handle,
		as in the 2-2-move,
		dragging the disks along.
		This is illustrated in Figure \ref{fig:3-2-slide}.

		For an explicit model,
		assume that the 2-handle $h_2$ is attached to the handlebody $H$ along $\attachingboundary h_2 = S^1 \times D^2$.
		Choose an embedded interval $[-1, +1] \subset \partial D^2$.
		This defines a strip $S^1 \times [-1, +1]$ on the boundary of $\attachingboundary h_2$,
		and a thickened disk $D^2 \times [-1, +1]$ inside the remaining region $\remainingboundary h_2$.
		The 3-handle now isotopes from the old attaching sphere $S^2_\mathrm{o}$ to the new attaching sphere $S^2_\mathrm{n}$,
		and we seek the image of this isotopy in the drawing canvas,
		which is a ball with a thickened disk and a ball removed:
		\[M \coloneqq (D^3_\mathrm{n} \backslash D^2 \times [-1, +1]) \backslash D^3_\mathrm{o}\]
		The $D^2 \times [-1,+1]$ is understood to be embedded such that the end disks $D^2 \times \{-1, +1\}$ are embedded in the boundary of $D^3_n$,
		one on each ``pancake''.
		$M$ has as boundary two components,
		$(S^2_\mathrm{n} \backslash (D^2_+ \sqcup D^2_-)) \cup S^1 \times [-1, +1]$,
		and $S^2_\mathrm{o}$.
		Assuming that the 3-handle is initially attached along $S^2_o$,
		it can be isotoped such that it is partially in the remaining region of the 2-handle,
		and the part staying visible in the main drawing canvas is $S^2_n \backslash (D^2_+ \sqcup D^2_-)$.
		It vanishes into $\remainingboundary h_2$ at $S^1$ times the endpoints of $[-1, +1]$.
	\item[3-3] One attaching $S^2$ of a 3-handle can ``engulf'' another one.
		In analogy to the 2-2-slide,
		the three-dimensional pair of pants is defined as a ball $D^3_C$ with two open balls $D^{2,\circ}_A$ and $D^{2,\circ}_B$ removed.
		Its boundary is $S^2_A \sqcup S^2_B \sqcup S^2_C$.
		Again, an embedding of the three-dimensional pair of pants in the boundary of a handle body defines three 3-handle attachments $h_{3,A}$, $h_{3,B}$ and $h_{3,C}$.
		The handle $h_{3,A}$ can be slid over $h_{3,B}$,
		which replaces it by $h_{3,C}$.
		This is illustrated in Figure \ref{fig:3-3-slide}.
	\item[3-$\infty$]
		Regularity requires handle attachments not to intersect with the point at infinity of $\remainingboundary h_0 = S^3 \cong \R^3 \cup \{\infty\}$.
		While any isotopy of 1-handles and 2-handles can be deformed away from this point such that the attachment never intersects with it,
		this is not always possible for 3-handles.
		The fundamental move that isotopes the attaching $S^2$ over this point is called the 3-$\infty$-move.

		When a 3-handle is only attached to $h_0$,
		its attaching $S^2$ separates the drawing canvas into an inside region and an outside region.
		The 3-$\infty$-move interchanges the roles of the two regions,
		as illustrated in Figure \ref{fig:S1S3}.
\end{description}
\begin{remark}
	\label{rem:3-infty-move}
	The 3-$\infty$-move is technically not a handle slide,
	but shares enough similarity in order to list it with the other slides.
	It is helpful to visualise the move applied to a 3-handle $h_3$ as the following procedure:
	\begin{enumerate}
		\item Attach a new 3-handle $h_3^\infty$ such that its attaching $S^2$ surrounds the whole diagram,
			or equivalently the point at $\infty$.
		\item Slide $h_3$ over $h_3^\infty$.
		\item Remove $h_3^\infty$.
	\end{enumerate}
\end{remark}
\fxwarning{Want remark that there is an $n-1$-move in any dimension?}
\begin{remark}
	Often it is possible to perform slides along existing handles.
	For example, sliding a 1-handle $h_{1,\alpha}$ over $h_{1,\beta}$ requires a path from $\attachingboundary h_{1,\alpha}$ to $\attachingboundary h_{1,\beta}$.
	If there is a 2-handle $h_2$ attached to both these 1-handles,
	a segment of its attaching circle provides such a path,
	and the 1-1-slide can be performed along it,
	while simultaneously sliding $h_2$ over $h_{1,\beta}$ as well.

	Similarly, a 2-2-slide can sometimes be performed along a 3-handle attachment,
	while simultaneously sliding the 3-handle.
	This is illustrated in Figure \ref{fig:k-2-slides}.
\end{remark}
\subsection{Examples}

\fxwarning{As further examples,
	do we want illustrations of $\CP{2}$, $\opCP{2}$, or at least references?
	Although, they are mentioned in the calculations, that'll suffice.
	}

\subsubsection{$S^1 \times S^3$}

\begin{figure}
	\centering
	\begin{tikzpicture} [ baseline ]
		\node [ 1-handle ] at (-2,0) {+};
		\node [ 1-handle ]           {-};
		\path [ 3-handle curved ]
			circle [ radius = 1cm ];
	\end{tikzpicture}
	$\xleftrightarrow{3\text{-}\infty}$
	\begin{tikzpicture} [ baseline ]
		\node [ 1-handle ] at (2,0) {-};
		\node [ 1-handle ]          {+};
		\path [ 3-handle curved ]
			circle [ radius = 1cm ];
	\end{tikzpicture}
	\caption{
		Two Kirby diagrams arising from the same handle decomposition of $S^1 \times S^3$,
		related by the 3-$\infty$-move.
	}
	\fxwarning[layout={inline}]{Spacing}
	\fxnote[layout={inline}]{Do we want the intermediate step at infinity?}
	\label{fig:S1S3}
\end{figure}

There is a standard handle decomposition of $S^n$ as a single 0-handle and a single $n$-handle.
Since handle decompositions have products,
and $h_{k_1} \times h_{k_2} \cong h_{k_1 + k_2}$,
there is a decomposition of $S^1 \times S^3$ into $h_0 \cup h_1 \cup h_3 \cup h_4$.
Two possible Kirby diagrams of this handle body are shown in Figure \ref{fig:S1S3}.
They are related via the 3-$\infty$-move.
See also \cite[Figure 4.15]{GompfStipsicz}.

\subsubsection{$S^1 \times S^1 \times S^2$}

For 3-dimensional manifolds,
Heegard diagrams are a standard depiction of handle decompositions.
\cite[Example 4.6.8]{GompfStipsicz} explains how to construct a handle decomposition of $S^1 \times M^3$ from a Heegard diagram of $M^3$.
In \cite[Section 6.2.1]{BaerenzBarrett2016Dichromatic},
a Kirby diagram of $S^1 \times S^1 \times S^2$ is derived,
but 3-handles are suppressed.
Figure \ref{fig:S1S1S2} shows the same Kirby diagram with 3-handles.

The reader might be surprised about the apparent lack of symmetry between $D^{3+}_\alpha$ and $D^{3-}_\alpha$
(and similarly $D^{3+}_\beta$ and $D^{3-}_\beta$),
with the 3-handle spheres covering only one of the 1-handle balls.
(After all, each 1-handle corresponds to an $S^1$ factor,
and mirroring the diagram such that e.g. $D^{3\pm}_\alpha$ are preserved and the $D^{3\pm}_\beta$ are interchanged
results in orientation reversal of the $S^1$ corresponding to the 1-handle $h_{1\beta}$,
which is an isomorphism.)
This asymmetry stems from the second regularity condition in Definition \ref{def:regular handle decomposition},
which stipulates that no handle attachment intersect with $\infty$.
If we remove that restriction, we can slide a point on $\attachingboundary h_{2b}$ to $\infty$ and end up with a symmetric diagram like the one given in \cite[Figure 4.1a]{HarerKasKirby86HandlebodyComplexSurfaces},
where $h_{2b}$ is attached along the $y$-axis,
and $h_{3A}$ and $h_{3B}$ are attached along the $y-z$-plane and the $x-y$-plane, respectively.

\begin{anfxwarning}{An explanation in words how to construct it.
	Is it helpful at all?
	Maybe as caption in the figure?
}
\small
The standard decomposition of $S^1 \times S^2$ has a single 3-handle with its attaching sphere filling out all of the diagram
(north and south poles are both inside the 2-handle,
leaving the 2-handle at the polar circles
and entering the 1-handle at the equator).
When multiplying with the interval and creating $I \times M$,
this 3-handle becomes a 3-handle in the Kirby diagram.
When realising $S^1 = h_0 \cup h_1$,
the 2-handle of $M$ gives rise to a 3-handle which starts inside the remaining sphere of the 2-handle,
goes from the 2-handle to the additional 1-handle coming from the single 0-handle of $M$,
through the 1-handle,
and enters the 2-handle on the other side again,
closing off in the remaining sphere.
\end{anfxwarning}
\begin{figure}
	\centering
	\newdimen\littleradius
	\newdimen\encircradius
	\newdimen\intodist
	\tikzmath{
		\onehandlex   = 3;
		\onehandley   = 2;
		\littleradius = 1mm;
		\encircradius = 6mm;
		\gapangle     = 180 * \diagramwhitesep / (\encircradius * 3.14);
		\slantangle   = atan(\onehandley/\onehandlex);
		\intodist     = 2mm;
		\awayangle    = 30;
	}
	\begin{tikzpicture}
		[ into 1-handle/.style = {shorten > = -\intodist} ]
		\pgfdeclarelayer{debug}
		\pgfdeclarelayer{background}
		\pgfdeclarelayer{behindtwohandle}
		\pgfdeclarelayer{twohandle}
		\pgfdeclarelayer{beforetwohandle}
		\pgfdeclarelayer{middle}
		\pgfdeclarelayer{front}
		\pgfsetlayers{debug,background,behindtwohandle,twohandle,beforetwohandle,middle,main,front}

		\coordinate (h1-apos) at ( \onehandlex,            0);
		\coordinate (h1+apos) at (-\onehandlex,            0);
		\coordinate (h1-bpos) at (           0,  \onehandley);
		\coordinate (h1+bpos) at (           0, -\onehandley);
		\coordinate (twohandle) at ($0.5*(h1+apos) + 0.5*(h1-bpos)$);

		\begin{pgfonlayer}{debug}
		\end{pgfonlayer}
		\begin{pgfonlayer}{background}
			\node [ 1-handle ] (h1-b) at (h1-bpos) {$D^{3-}_\beta$};
		\end{pgfonlayer}
		\begin{pgfonlayer}{middle}
			\node [ 1-handle ] (h1-a) at (h1-apos) {$D^{3-}_\alpha$};
			\node [ 1-handle ] (h1+a) at (h1+apos) {$D^{3+}_\alpha$};
		\end{pgfonlayer}
		\begin{pgfonlayer}{front}
			\node [ 1-handle ] (h1+b) at (h1+bpos) {$D^{3+}_\beta$};
		\end{pgfonlayer}

		\draw [ thick, into 1-handle ] (h1-a) -- (h1-b);
		\begin{pgfonlayer}{beforetwohandle}
			\path [ 3-handle = roch ]
				($(twohandle)+(\slantangle+90:\encircradius)$)
				to [ out = \slantangle + 180, in = \slantangle ]
				($(h1+a.\slantangle)+(180+\slantangle:\intodist)+(90+\slantangle:\littleradius)$)
				arc
					[ radius      = \littleradius
					, start angle = \slantangle +  90
					, end angle   = \slantangle + 270
					]
				to [ out = \slantangle, in = \slantangle + 180 ]
				($(twohandle)+(\slantangle+270:\encircradius)$)
				arc
					[ radius      = \encircradius
					, start angle = \slantangle - 90
					, end angle   = \slantangle + 90
					]
			;
			\draw [ thick ]
				(h1+a) --
				($(twohandle) + (\slantangle:\encircradius-\diagramwhitesep)$);
		\end{pgfonlayer}
		\begin{pgfonlayer}{behindtwohandle}
			\path [ 3-handle ]
				($(twohandle)+(\slantangle+90:\encircradius)$)
				to [ out = \slantangle, in = \slantangle + 180 ]
				($(h1-b.\slantangle+180)+(\slantangle:\intodist)+(90+\slantangle:\littleradius)$)
				arc
					[ radius      = \littleradius
					, start angle = \slantangle +  90
					, end angle   = \slantangle + 270
					]
				to [ out = \slantangle + 180, in = \slantangle ]
				($(twohandle)+(\slantangle + 270:\encircradius)$)
				arc
					[ radius      = \encircradius
					, start angle = \slantangle + 270
					, end angle   = \slantangle +  90
					]
			;
			\draw [ thick, into 1-handle ]
				($(twohandle) + (\slantangle:\encircradius+\diagramwhitesep)$) --
				(h1-b);
		\end{pgfonlayer}

		\begin{pgfonlayer}{twohandle}
			\draw [ thick ]
				($(twohandle) + (-180+\slantangle+\gapangle:\encircradius)$)
				arc
					[ start angle = -180 + \slantangle + \gapangle
					, end angle   =  180 + \slantangle - \gapangle
					, radius      = \encircradius
					]
			;
		\end{pgfonlayer}

		\draw [ thick, into 1-handle ] (h1+b) --
			node [ below right ] {$\attachingboundary h_{2a}$} (h1-a);
		\draw [ thick, into 1-handle ] (h1+b) -- (h1+a);

		\begin{pgfonlayer}{behindtwohandle}
			\path [ 3-handle ]
				($(twohandle)+(\slantangle+90:\encircradius)$)
				to [ out = \slantangle, in = \slantangle + 180 ]
				($(h1-b.\slantangle+180)+(\slantangle:\intodist)+(90+\slantangle:\littleradius)$)
				arc
					[ radius      = \littleradius
					, start angle = \slantangle + 90
					, end angle   = \slantangle - 90
					]
				to [ out = \slantangle + 180, in = \slantangle ]
				($(twohandle)+(\slantangle + 270:\encircradius)$)
				arc
					[ radius      = \encircradius
					, start angle = \slantangle - 90
					, end angle   = \slantangle + 90
					]
			;
			\draw [ thin, <- ]
				(-0.35*\onehandlex, 0.88*\onehandley)
				--
				+(0, 1.5)
				node [ above ] {$\attachingboundary h_{3A}$}
			;
			\path [ 3-handle = roch ]
				($(twohandle)+(\slantangle+\awayangle+90:\encircradius)$)
				to [ out = \slantangle + \awayangle, in =  180 ]
				($(h1-b)+(0,1)$)
				to [ out =                              0, in = 180 - \slantangle ]
				($(h1-a.-\slantangle+180)+(-\slantangle:\intodist)+(90-\slantangle:\littleradius)$)
				arc
					[ radius      = \littleradius
					, start angle =  90 - \slantangle
					, end angle   = 270 - \slantangle
					]
				to [ out = 180 - \slantangle, in = \slantangle - \awayangle ]
				($(twohandle)+(\slantangle-\awayangle+270:\encircradius)$)
				arc
					[ radius      = \encircradius
					, start angle = \slantangle - \awayangle - 90
					, end angle   = \slantangle + \awayangle + 90
					]
		;
		\end{pgfonlayer}
		\begin{pgfonlayer}{beforetwohandle}
			\path [ 3-handle = roch ]
				($(twohandle)+(\slantangle+90:\encircradius)$)
				to [ out = \slantangle + 180, in = \slantangle ]
				($(h1+a.\slantangle)+(180+\slantangle:\intodist)+(90+\slantangle:\littleradius)$)
				arc
					[ radius      = \littleradius
					, start angle = \slantangle + 90
					, end angle   = \slantangle + 270
					]
				to [ out = \slantangle, in = \slantangle + 180 ]
				($(twohandle)+(\slantangle + 270:\encircradius)$)
				arc
					[ radius      = \encircradius
					, start angle = \slantangle - 90
					, end angle   = \slantangle + 90
					]
			;
			\draw [ thin, <- ]
				(-0.8*\onehandlex, 0.4*\onehandley)
				--
				+(135:1.5cm)
				node [ above left = -1mm ] {$\attachingboundary h_{3B}$}
			;
		\end{pgfonlayer}
		\path [ 3-handle ]
			($(twohandle)+(\slantangle-\awayangle+90:\encircradius)$)
			to [ out = \slantangle - \awayangle + 180, in =  90 ]
			($(h1+a)+(-1,0)$)
			to [ out =                            270, in = 180 - \slantangle ]
			($(h1+b.-\slantangle+180)+(-\slantangle:\intodist)+(270-\slantangle:\littleradius)$)
			arc
				[ radius      = \littleradius
				, start angle = 270 - \slantangle
				, end angle   =  90 - \slantangle
				]
			to [ out = 180 - \slantangle, in = \slantangle + \awayangle + 180 ]
			($(twohandle)+(\slantangle+\awayangle+270:\encircradius)$)
			arc
				[ radius      = \encircradius
				, start angle = \slantangle + \awayangle - 90
				, end angle   = \slantangle - \awayangle + 90
				]
		;

		\draw [ thin, <- ]
			($(twohandle)+(\slantangle+90:1.05*\encircradius)$)
			--
			+(\slantangle+90:1.5cm)
			node [ above left = -1mm ] {$\attachingboundary h_{2b}$};
		\node [ right ] at (h1+a.east)  {$\attachingboundary h_{3A}$};
		\node [ below ] at (h1-b.south) {$\attachingboundary h_{3B}$};
	\end{tikzpicture}
	\caption{Handle decomposition of $S^1 \times S^1 \times S^2$ with two 3-handles,
		two 2-handles and two 1-handles.
		Choosing an arbitrary orientation and start point,
		the attaching $S^1$ of $h_{2a}$ attaches to $h_{1\alpha}$ at $D^{3-}_\alpha$,
		runs through the remaining region $\remainingboundary h_{1\alpha}$,
		leaves at $D^{3+}_\alpha$,
		attaches to $h_{1\beta}$ at $D^{3+}_\beta$,
		runs through $\remainingboundary h_{1\beta}$,
		leaves at $D^{3-}_\beta$,
		attaching a second time to $h_{1\alpha}$ and $h_{1\beta}$ each,
		and finally closing the loop.
		The attachment of $h_{3A}$ starts (invisibly) inside the remaining region of $h_{2b}$,
		extends to a $D^2$ and leaves it at $\attachingboundary h_{2b}$,
		proceeding as a cylinder around $\attachingboundary h_{2a}$ and $D^{3+}_\alpha$,
		attaching to $h_{1\beta}$ at $D^{3+}_\beta$,
		leaving again at $D^{3-}_\beta$ and attaching a second time to $h_{2b}$,
		closing off inside its remaining region.
	}
	\label{fig:S1S1S2}
	\begin{anfxwarning}[layout={inline}]{Finish illustration}
		\begin{enumerate}
			\item Wonder whether moving one 3-handle,
				or the 2-handle together with the other 3-handle to $\infty$
				improves it.
			\item Tweak opacity (lower it)
		\end{enumerate}
	\end{anfxwarning}
\end{figure}

\subsubsection{Fundamental group}
\label{sec:fundamental group}

It is well-known that the 2-handlebody of a handle decomposition gives a presentation of the fundamental group
\cite[Solution of Exercise 4.6.4(b)]{GompfStipsicz}, \cite[Section 2.3.3]{BaerenzBarrett2016Dichromatic},
where each 1-handle is a generator and each 2-handle a relation.

For example, it is a good exercise to verify that in Figure \ref{fig:S1S1S2},
the two 1-handles $h_{1\alpha}$ and $h_{1\beta}$ constitute two generators $\alpha$ and $\beta$ each,
and the 2-handle $h_{2a}$ results in the relation $\alpha\beta\alpha^{-1}\beta^{-1}$
(after having oriented the attaching spheres arbitrarily).
$h_{2b}$ is not attached to any 1-handles and thus yields the trivial relation.
We have shown that $\pi_1 (S^1 \times S^1 \times S^2) \cong \Z \oplus \Z$,
as expected.

Turning a handle decomposition upside down \cite[Section 4.2]{GompfStipsicz}
shows that it is possible to present the fundamental group with the 3-handles as generators and again the 2-handles as relations.
Each 2-handle yields a relation where the generators corresponding to all attached 3-handles are multiplied in cyclical order.\footnote{%
To see this, visualise that for each 3-hande,
a noncontractible $S^1$ starts in the interior of the single 4-handle,
passes through the first of the two $D^3$s constituting the remaining region of the 3-handle,
and becomes visible as $\text{pt.} \times [-1, +1] \subset S^2 \times [-1, +1]$ in the thickened attaching sphere of the 3-handle,
and enters the second $D^3$ of the remaining region,
back into the 4-handle.
It can be contracted in the remaining region of a 2-handle if the 3-handle is attached to it.
}
A generator is inverted if the boundary orientation deriving from the (arbitrarily chosen) orientation of the attaching sphere $S^2$ of $h_3$
does not match the (again arbitrarily chosen) orientation of the attaching $S^1$ of $h_2$.

\begin{anfxnote}{Some more description of how to contract the $S^1$. Useful?}
\small
if we push the visible interval into the part of the attaching sphere onto the boundary 2-handle.
It can be visualised as $\R^3 \cup \infty$,
with the attaching region as an embedded $S^1 \times D^2$,
and the remaining region as a lot of $S^1$s forming Hopf links with the attaching regions.
The relation is then seen by pushing the rest of the formerly noncontractible $S^1$ into the remaining region of the 2-handle and making one of the Hopf-linking $S^1$s out of it.
\end{anfxnote}

Again, Figure \ref{fig:S1S1S2} offers a good exercise to compute the fundamental group in this presentation.
The two 3-handles yield generators $A$ and $B$.
None of the 3-handles is attached to $h_{2a}$,
so its relation is trivial.
Both 3-handles are attached twice to $h_{2b}$, though.
Choosing arbitrary orientations and taking care that the orientations on both parts of the attaching spheres of the 3-handles match,
we read off the relation $ABA^{-1}B^{-1}$ and arrive at the same group.

\section{Graphical calculus in $G$-crossed braided spherical fusion categories}
\label{sec:Graphical calculus}

Our strategy to define an invariant of 4-manifolds from a $G$-crossed braided spherical fusion category $\mathcal{C}$ (short: $G\times$-BSFC)
is to choose a Kirby diagram of the manifold,
and to interpret and evaluate this diagram in a graphical calculus of $\mathcal{C}$.
As will be shown in this section,
this calculus is conveniently similar to Kirby calculus with 3-handles.

\subsection{From manifolds to morphisms}
\fxerror{Consider moving after the diagrams section}
\begin{figure}
	\centering
	\newdimen\arrowdistance
	\setlength{\arrowdistance}{2mm}
	\begin{tikzpicture}
		[ down/.style = {->, transform canvas={xshift=-\arrowdistance}}
		, up/.style   = {<-, transform canvas={xshift= \arrowdistance}}
		]
		\node (manifold) {4-manifold};
		\node (handle) [below=of manifold] {Handle decomposition};
		\draw [down]
			(manifold)
			-- node [left] {Existence of Morse functions}
			(handle);
		\draw [up]
			(manifold)
			-- node [right, align=left] {Handle cancellations,\\regular isotopies of attaching maps}
			(handle);
		\node (kirby) [below=of handle] {Kirby diagram};
		\draw [down]
			(handle)
			-- node [left] {Lemma \ref{lem:establish Kirby conventions}}
			(kirby);
		\draw [up]
			(handle)
			-- node [right] {$n$-$k$ handle slides, $3$-$\infty$-move}
			(kirby);
		\node (diagram) [below=of kirby] {Unlabelled (regular) planar diagram};
		\draw [down]
			(kirby)
			-- node [left] {Definition \ref{def:kirby diagram to planar diagram}}
			(diagram);
		\draw [up]
			(kirby)
			-- node [right] {Remark \ref{rem:isotopies of kirby diagrams and planar diagrams}}
			(diagram);
		\node (labelled) [below=of diagram] {Labelled (regular) planar diagram};
		\draw [down]
			(diagram)
			-- node [left] {Definition \ref{def:labelled diagram}}
			(labelled);
		\draw [up]
			(diagram)
			-- node [right] {Proposition \ref{prop:diagram evaluation invariant under isotopy}}
			(labelled);
		\node (morphism) [below=of labelled] {Invariant morphism};
		\draw [down]
			(labelled)
			-- node [left] {Definition \ref{def:invariant}}
			(morphism);
		\draw [up]
			(labelled)
			-- node [right] {Theorem \ref{thm:Invariance}}
			(morphism);
	\end{tikzpicture}
	\caption{
		In order to define the invariant, a 4-manifold is translated to a morphism in a $G$-BSFC through various intermediate steps.
		For each of these steps, additional data has to be added, or additional properties assumed.
		The arrows pointing down illustrate this additional structure,
		while the arrows pointing up justify it through a proof of independence on the specific structure chosen.
	}
	\label{fig:manifolds to morphisms}
\end{figure}
\fxerror{Manifold -> Handle decomposition -> Kirby diagram -> Kirby diagram with infinity normalised -> Unlabelled diagram -> Unlabelled regular planar diagram -> Labelling and interpretation.
Missing some lemmas in the graphics.}

\fxerror{Consider making this its own section, move Def 4.1 and details on regularity here.
Expand on when two diagrams are equivalent (represent the same manifold/evaluate the same)}

Our general guiding principle to arrive at a graphical calculus is that $G\times$-BSFCs are special degenerate (weak) 3-categories,
and thus monoidal 2-categories \cite{Cui2016TQFTs}.
Viewing them as 3-categories,
there is a single object,
the 1-morphisms correspond to group elements,
the 2-morphisms correspond to objects of the fusion category,
and the 3-morphisms correspond to morphisms of the fusion category.
The diagrams should thus be 3-dimensional,
with group elements labelling 2-dimensional \textbf{sheets},
\fxwarning{Alternative convention: Call them ``surfaces'' or 2-strata.}
objects of the fusion category labelling 1-dimensional \textbf{ribbons},
and morphisms of the fusion category labelling \textbf{points}.

Here, a similarity can already be seen to the elements of Kirby diagrams,
which are the $S^2$s corresponding to 3-handles,
1-dimension curves for 2-handles,
and $D^3$s (thick points) for 1-handles.
Kirby diagrams themselves are not flexible enough for arbitrary graphical calculus in a $G\times$-BSFC.
This leads us to the definition of planar diagrams,
which are an intermediate step between Kirby diagrams and the desired graphical calculus.
Figure \ref{fig:manifolds to morphisms} gives an overview over all steps necessary to compute the invariant from a manifold.

\subsection{Unlabelled planar diagrams}
\fxerror{Clarify their role. Every Kirby diagram gives rise to one, and in principle all moves generalise to planar diagrams (no actually not. e.g. 2-1).
But it's more convenient to have this definition, especially also for later with the simplices.
Basically, these diagrams are the vessel for graphical calculus.}
\fxerror{planar? closed? what now?}
\begin{definition}
	\label{def:unlabelled diagram}
	An unlabelled, closed \textbf{diagram} $\mathcal{D}$ consists of the following data:
	\begin{enumerate}
		\setcounter{enumi}{-1}
		\item A finite set $\mathcal{D}_0$ of embeddings $p\colon D^3 \to \R^3$,
			called (thickened) \textbf{points},
		\item a finite set $\mathcal{D}_1$ of framed embeddings of \textbf{lines} $[-1, 1]$ and \textbf{circles} $S^1$
			(collectively called \textbf{ribbons})
			into $\R^3$ such that the boundaries (endpoints) of each line are embedded in the boundaries of the thickened points,
		\item a finite set $\mathcal{D}_2$ of smooth embeddings of $n$-gons (see Definition \ref{def:n-gons}),
			called \textbf{sheets},
			or sometimes \textbf{disks} (when $n > 0$) and \textbf{spheres}.
			into $\R^3$ such that each boundary component of each disk is either a line or embedded on a boundary of a thickened point.
	\end{enumerate}
\end{definition}

\begin{figure}
	\centering
	\begin{tikzpicture}[ looseness = 3 ]
		\pgfdeclarelayer{threehandle}
		\pgfsetlayers{threehandle,main}
		\node [ 1-handle ] (p) {$p$};
		\draw [ thick ] (p.45) to [ out = 45, in = -45 ] (p.-45);
		\begin{pgfonlayer}{threehandle}
			\path [ 3-handle ] (p.45) to [ out = 45, in = -45 ] (p.-45);
		\end{pgfonlayer}
	\end{tikzpicture}
	\caption{An unlabelled closed diagram with a single point, a single line, and a single disk.}
\end{figure}

\begin{remark}
	$n$-gons are manifolds with corners,
	likewise their embedded image will have corners.\footnote{%
		Instead of smooth embeddings of manifolds with corners,
		one may imagine a \emph{piecewise smooth} embedding of a disk or sphere instead.
		The formulation with corners appears cleaner though,
		particularly in the light of Definition \ref{def:handles},
		where handles are defined as manifolds with corners.
	}
	The $n$-gons will always be $2k$-gons, in fact,
	where $k$ is the number of times the sheet attaches to a thickened point
	(or equivalently the number of lines it attaches to).
	The corners of the sheet will always be attached to an endpoint of a line embedding.
\end{remark}

The reader may be concerned at this point that neither sheets nor ribbons were required to be thickened, or framed.
The reason is that disks and spheres contain no framing information,
and the framing of a ribbon can be canonically chosen to be the \emph{blackboard framing} by choosing a projection into the plane.
(For further details, also see Appendix \ref{sec:blackboard framing}.)
\fxwarning{Possibly absorb everything important into def. below and move this bracket}
Since we will ultimately want to evaluate diagrams in monoidal categories
(which have a diagrammatic calculus in the plane),
we shall make sure that the projection of a diagram contains all essential information.
This is encoded in the following definition:
\begin{definition}[Blackboard framing]
	We single out a projection into the plane onto the first two components of a vector:
	\begin{align}
		\pi \colon \R^3 &\xrightarrow{\cong} \R^2 \oplus \R \xrightarrow{\pi_{\R^2}} \R^2 \\\nonumber
		(x,y,z) &\mapsto (x,y)
	\end{align}
	A diagram $\mathcal{D}$ is compatible with the \textbf{blackboard framing} if the following conditions are satisfied:
	\begin{itemize}
		\item Each thickened point $D^3$ can be split in half along an \textbf{equator} $S^1$ such that its boundary is split into ``upper'' and ``lower'' disks,
			$S^2 = D^2_u \cup_{S^1} D^2_l$
			and each disc $D^2_{\{u,l\}}$ is \emph{embedded} into $\R^2$ by $\pi$.
		\item Lines attach to thickened points transversely in the equator.
		\item On sheets, $\pi$ is an immersion except at a finite set of 1-dimensional compact embedded piecewise smooth manifolds, the \textbf{fold graphic}
			\cite[Section 1.3]{Schommer-Pries:PhD}.
			The non-smooth points of the fold graphic are called \textbf{cusp points},
			and the smooth parts \textbf{fold arcs}.
			\fxnote{Noone says that has to be a manifold.
				In fact I'm not sure this will always work.
				(E.g. for the full twist it seems we need a cusp.)
				We can have cusps as well,
				and then I need to prove one relation that $1_{\acts{g}{X}} = (\eta \otimes 1) \circ (1 \otimes \eta)$.
			}
		\item On the union of fold arcs and ribbons, $\pi$ is an embedding except at a finite set of points, the \textbf{crossings}.
	\end{itemize}
\end{definition}
\begin{remark}
	In the projection, a thickened point is given by simply an embedded $D^2$.
\end{remark}
\begin{remark}
	\fxerror{But do we ever use these orientations anywhere?}
	The embedded diagram data are understood to be oriented.
	The resulting data is sometimes depicted graphically, which manifests depending on the dimension on the datum:
	\begin{enumerate}
		\setcounter{enumi}{-1}
		\item The orientation of points can either agree with the orientation of $\R^3$,
			or be opposite, specifying a sign.
		\item The orientation on a ribbon manifests as a direction along the ribbon,
			often depicted with a little arrow.
		\item The orientation of a sheet can coincide locally with the orientation of the plane $\R^2$,
			or be opposite, specifying a sign per noncritically embedded stratum.
	\end{enumerate}
\end{remark}
\begin{remark}
	Ribbons (and borders of thickened points) locally partition the plane into two ``sides''.
	$\pi$ being an immersion on sheets outside of fold arcs implies that sheets incident to ribbons never change the side locally.

	This may be confusing at first when recalling that we emphasized in Remark \ref{rem:single picture} that the single picture conventions
	(Definition \ref{def:single picture conventions})
	imply that in a 3-2-handle attachment,
	the 3-handle must follow the framing of the 2-handle.
	But since any ribbon also follows the blackboard framing,
	it does not ``turn'' in the plane,
	so these two conditions conform.
	\fxerror{Ok, so this could actually work!
		Read probably Turaev's book and hopefully it has the graphical calculus of braided cats \emph{with blackboard framing}.
		Also look in my thesis whether I used any good reference on that.
		Another good reference may be Turaev, V., Quantum invariants of knots and 3-manifolds
	}
\end{remark}

The parts where $\pi$ is critical on the diagram,
the fold graphic and the crossings,
will later correspond to structure morphisms of the $G$-crossed fusion category.
Certain regularity conditions need to be imposed in order to make the diagrams suitable for graphical calculus.

\begin{definition}[Regular diagrams]
	\label{def:regular diagram}
	A diagram $\mathcal{D}$ is \textbf{regular} if the following conditions are satisfied:
	\fxerror{If I do $n$-gons I need to say that the corners of the $n$-gons don't intersect anything}
	\begin{itemize}
		\item It is compatible with the blackboard framing.
		\item The projections of ribbons and fold arcs do not intersect projections of thickened points.
		\item The projections of exactly two ribbons, two fold arcs or a ribbon and a fold arc,
			intersect transversely at a crossing.
		\item The projections of different crossings and cusp points never coincide.
	\end{itemize}
\end{definition}
\begin{remark}
	One can convince oneself that one can perturb any planar diagram with an isotopy to a regular diagram.
	They are thus also often said to be in \textbf{general position},
	and we will usually assume that a diagram is regular.
\end{remark}
\fxwarning{Take a best-of of my diagrams and label them with all that goodie stuff}

\begin{definition}
	\label{def:regular isotopy}
	An \textbf{regular planar isotopy} is an isotopy of planar diagrams,
	such that they are regular, and each intermediate diagram is regular.
\end{definition}

\begin{figure}
	\tikzmath{
		\height   = 1.0;
		\width    = 0.5;
		\widthree = 0.9;
		\thirdx   = 0.3;
		\twistheight = 1.5;
		\twistwidth  = 1;
		\cancellationsize = 2;
	}
	\centering
	\begin{tabular}{p{4cm}cc}
		\toprule
		Name
			& Left hand side
				& Right hand side
		\\\midrule
		Point isotopy
			&
			\begin{tikzpicture}[ looseness = 4, baseline ]
				\pic { point isotopy = 0.6 };
			\end{tikzpicture}
				&
				\begin{tikzpicture}[ looseness = 4, baseline ]
					\pic { point isotopy = {-0.8} };
				\end{tikzpicture}
			\\\midrule
		Reidemeister I'
		(twist move)
			&
			\tikzmath{
				\intersectionslant = 25;
				\crossingBL = 270 - \intersectionslant;
				\crossingBR = 270 + \intersectionslant;
				\crossingTR =  90 - \intersectionslant;
				\crossingTL =  90 + \intersectionslant;
				\foldcrossingslant = 40;
				\foldcrossingBL = 270 - \foldcrossingslant;
				\foldcrossingBR = 270 + \foldcrossingslant;
				\foldcrossingTR =  90 - \foldcrossingslant;
				\foldcrossingTL =  90 + \foldcrossingslant;
				}
			\begin{tikzpicture}[ looseness = 1.5, baseline ]
				\coordinate (lowerend) at (0, -\twistheight);
				\coordinate (lowerfoldcrossing) at (0, -1);
				\coordinate (lower) at (0.2, -0.6);
				\coordinate (lowerturn) at ($(lower)+(0.6, 0)$);
				\coordinate (middle) at (0, 0);
				\coordinate (upperturn) at ($2*(lowerturn |- middle) - (lowerturn)$);
				\coordinate (upper) at ($2*(lower |- middle) - (lower)$);
				\coordinate (upperfoldcrossing) at ($2*(lowerfoldcrossing |- middle) - (lowerfoldcrossing)$);
				\coordinate (upperend) at ($2*(lowerend |- middle) - (lowerend)$);
				\coordinate (left) at (-\twistwidth, 0);
				\coordinate (foldleft) at (-0.4, 0);
				\coordinate (foldright) at (1.5, 0);

				\draw [ thick ]
					(lowerend) -- (lowerfoldcrossing)
					to [ out =  90, in = \crossingBL ]
					(lower)
					to [ out = \crossingTR, in =  90 ]
					(lowerturn);
				\draw [ thick ]
					(upperturn)
					to [ out = 270, in = \crossingBR ]
					(upper)
					to [ out = \crossingTL, in = -90 ]
					(upperfoldcrossing) -- (upperend);
				\draw [ diagramwhite ]
					(lowerturn)
					to [ out = 270, in = \crossingBR ]
					(lower)
					to [ out = \crossingTL, in = 270 ]
					(middle)
					to [ out =  90, in = \crossingBL ]
					(upper)
					to [ out = \crossingTR, in =  90 ]
					(upperturn);

				\draw [ thick ]
					(lowerturn)
					to [ out = 270, in = \crossingBR ]
					(lower)
					to [ out = \crossingTL, in = 270 ]
					(middle)
					to [ out =  90, in = \crossingBL ]
					(upper)
					to [ out = \crossingTR, in =  90 ]
					(upperturn);

				\fill [ 3-handle fill ]
					(left |- lowerend)
					--
					(lowerend) -- (lowerfoldcrossing)
					to [ out =  90, in = \crossingBL ]
					(lower)
					to [ out = \crossingTR, in =  90 ]
					(lowerturn)
					to [ out = 270, in = \crossingBR ]
					(lower)
					to [ out = \crossingBL, in =  90 ]
					(lowerfoldcrossing)
					to [ out = \foldcrossingBR, in = 270 ]
					(foldright)
					to [ out =  90, in = \foldcrossingTR ]
					(upperfoldcrossing)
					to [ out = 270, in = \crossingTL ]
					(upper)
					to [ out = \crossingTR, in = 90 ]
					(upperturn)
					to [ out = 270, in = \crossingBR ]
					(upper)
					to [ out = \crossingTL, in = -90 ]
					(upperfoldcrossing) -- (upperend)
					--
					(left |- upperend)
					-- cycle;

				\fill [ 3-handle fill ]
					(upperfoldcrossing)
					to [ out = 270, in = \crossingTL]
					(upper)
					to [ out = \crossingBL, in =  90 ]
					(middle)
					to [ out = 270, in = \crossingTL ]
					(lower)
					to [ out = \crossingBL, in =  90 ]
					(lowerfoldcrossing)
					to [ out = \foldcrossingTL, in = 270 ]
					(foldleft)
					to [ out =  90, in = \foldcrossingBL ]
					(upperfoldcrossing);

				\draw [ 3-handle ]
					(foldleft)
					to [ out =  90, in = \foldcrossingBL ]
					(upperfoldcrossing)
					to [ out = \foldcrossingTR, in =  90 ]
					(foldright)
					to [ out = 270, in = \foldcrossingBR ]
					(lowerfoldcrossing)
					to [ out = \foldcrossingTL, in = 270 ]
					(foldleft);
			\end{tikzpicture}
				&
				\begin{tikzpicture} [ baseline ]
					\draw [ thick ] (0, -\twistheight) -- (0, \twistheight);
					\fill [ 3-handle fill ] (0, \twistheight) rectangle (-\twistwidth, -\twistheight);
				\end{tikzpicture}
		\\\midrule
		Reidemeister II (cancellation of inverse crossings)
			&
			\begin{tikzpicture} [ baseline ]
				\draw [ diagram ]
					( \width   , -\height)
					to [ out = 90, in = -90 ]
					(-\width   ,        0)
					to [ out = 90, in = -90 ]
					( \width   ,  \height);
				\draw [ white, line width = \diagramunderwhite ]
					(-\width   , -\height)
					to [ out = 90, in = -90 ]
					( \width   ,        0)
					to [ out = 90, in = -90 ]
					(-\width   ,  \height);
				\path [ 3-handle fill ]
					( \width   , -\height)
					to [ out = 90, in = -90 ]
					(-\width   ,        0)
					to [ out = 90, in = -90 ]
					( \width   ,  \height)
					--
					( \widthree,  \height)
					--
					( \widthree, -\height)
					-- cycle;
				\draw [ thick ]
					(-\width   , -\height)
					to [ out = 90, in = -90 ]
					( \width   ,        0)
					to [ out = 90, in = -90 ]
					(-\width   ,  \height);
				\path [ 3-handle fill = roch ]
					(-\width   , -\height)
					to [ out = 90, in = -90 ]
					( \width   ,        0)
					to [ out = 90, in = -90 ]
					(-\width   ,  \height)
					--
					(-\widthree,  \height)
					--
					(-\widthree, -\height)
					-- cycle;
			\end{tikzpicture}
				&
				\begin{tikzpicture} [ baseline ]
					\draw [ diagram ]
						( \width   , -\height)
						--
						( \width   ,  \height);
					\path [ 3-handle fill ]
						( \width   , -\height)
						--
						( \width   ,  \height)
						--
						( \widthree,  \height)
						--
						( \widthree, -\height)
						-- cycle;
					\draw [ thick ]
						(-\width   , -\height)
						--
						(-\width   ,  \height);
					\path [ 3-handle fill = roch ]
						(-\width   , -\height)
						--
						(-\width   ,  \height)
						--
						(-\widthree,  \height)
						--
						(-\widthree, -\height)
						-- cycle;
				\end{tikzpicture}
		\\\midrule
		Reidemeister III
		(crossing isotopy)
			&
			\begin{tikzpicture} [ baseline ]
				\draw [ thick ] (-\thirdx, -\height) -- (-\thirdx, \height);
				\draw [ diagram ]
					( \width   , -\height)
					to [ out = 90, in = -90 ]
					(-\width   ,  \height);
				\draw [ white, line width = \diagramunderwhite ]
					(-\width   , -\height)
					to [ out = 90, in = -90 ]
					( \width   ,  \height);
				\path [ 3-handle fill ]
					( \width   , -\height)
					to [ out = 90, in = -90 ]
					(-\width   ,  \height)
					--
					( \widthree,  \height)
					--
					( \widthree, -\height)
					-- cycle;
				\draw [ thick ]
					(-\width   , -\height)
					to [ out = 90, in = -90 ]
					( \width   ,  \height);
				\path [ 3-handle fill = roch ]
					(-\width   , -\height)
					to [ out = 90, in = -90 ]
					(0\width   ,  \height)
					--
					(-\widthree,  \height)
					--
					(-\widthree, -\height)
					-- cycle;
			\end{tikzpicture}
				&
				\begin{tikzpicture} [ baseline ]
					\draw [ thick ] (\thirdx, -\height) -- (\thirdx, \height);
					\draw [ diagram ]
						( \width   , -\height)
						to [ out = 90, in = -90 ]
						(-\width   ,  \height);
					\draw [ white, line width = \diagramunderwhite ]
						(-\width   , -\height)
						to [ out = 90, in = -90 ]
						( \width   ,  \height);
					\path [ 3-handle fill ]
						( \width   , -\height)
						to [ out = 90, in = -90 ]
						(-\width   ,  \height)
						--
						( \widthree,  \height)
						--
						( \widthree, -\height)
						-- cycle;
					\draw [ thick ]
						(-\width   , -\height)
						to [ out = 90, in = -90 ]
						( \width   ,  \height);
					\path [ 3-handle fill = roch ]
						(-\width   , -\height)
						to [ out = 90, in = -90 ]
						(0\width   ,  \height)
						--
						(-\widthree,  \height)
						--
						(-\widthree, -\height)
						-- cycle;
				\end{tikzpicture}
		\\\midrule
		Cusp isotopy
			&
			\begin{tikzpicture} [ baseline ]
				\pic {cusp strike={1}};
			\end{tikzpicture}
				&
				\begin{tikzpicture} [ baseline ]
					\pic {cusp strike={-1}};
				\end{tikzpicture}
		\\\midrule
		Cusp moves.
		Shown here: cancellation/inversion.
		Also relevant: Other inversion side, flip, swallowtail.
			&
			\begin{tikzpicture} [ baseline ]
				\coordinate (top)        at (0  ,  \cancellationsize);
				\coordinate (right)      at (\cancellationsize,  0  );
				\coordinate (turnmargin) at (0  ,  0.1);
				\coordinate (turn)       at (0.4,  0  );
				\coordinate (upperturn)  at ($        (turnmargin) + (turn)$);
				\coordinate (lowerturn)  at ($(0,0) - (turnmargin) - (turn)$);
				\coordinate (cusp)       at (0  , -1.2);

				\draw [ 3-handle ] ($(top)+(right)$) rectangle ($(0,0)-(top)-(right)$);
				\draw [ 3-handle, preaction = { 3-handle fill } ]
					($(0,0)-(cusp)$)
					to [ out = 300, in =  90 ]
					(upperturn)
					to [ out = 270, in =  60 ]
					(cusp)
					to [ out = 120, in = 270 ]
					(lowerturn)
					to [ out =  90, in = 240 ]
					cycle;
				\draw [ 3-handle draw, dashed, thick ]
					($(0,0)-(right)$)
					to [ out =   0, in =  90 ]
					(upperturn)
					to [ out = 270, in =  90 ]
					(lowerturn)
					to [ out = 270, in = 180 ]
					($(0,0)+(right)$);
			\end{tikzpicture}
				&
				\begin{tikzpicture} [ baseline ]
					\draw [ 3-handle ]
						( \cancellationsize,  \cancellationsize)

						rectangle
						(-\cancellationsize, -\cancellationsize);
					\end{tikzpicture}
		\\\bottomrule
	\end{tabular}
	\label{fig:planar moves}
	\caption{
		Different kinds of planar moves.
		Note that on each ribbon,
		there are infinitely many possibilities of attaching further 3-handles,
		for which only one example per move is shown.
		Furthermore, each move that occurs for a ribbon can also occur for a fold line.
	}
\end{figure}

\begin{figure}
	\centering
	\begin{tabular}{ccc}
		\toprule
		\begin{tikzpicture}
			\pic {3-handle-over-2-handle={2}{0.3}};
		\end{tikzpicture}
			&
			\begin{tikzpicture}
				\pic {3-handle-over-2-handle={2}{1.0}};
			\end{tikzpicture}
				&
				\begin{tikzpicture}
					\pic {3-handle-over-2-handle={2}{1.3}};
				\end{tikzpicture}
		\\\midrule
		Regular: sheet above ribbon
			& Critical point
				& Regular: sheet beside ribbon
		\\\bottomrule
	\end{tabular}
	\label{fig:regular}
	\caption{
		By changing the projection,
		a regular diagram can change into one that does not satisfy the regularity condition.
		In this case, the intersection of the projection of a fold arc and a ribbon is not transverse.
		A further change leads again to a regular diagram,
		but one that is not regularly isotopic to the original one.
		They are related by a planar move though,
		in this case the Reidemeister II move.
	}
\end{figure}

\fxerror{Baselines}
\fxerror{The other cusp cancellation is easier to draw}
\fxerror{Dashed line stronger contrast. And align squares}
\fxerror{Some moves like Reidemeister 2 can occur for ribbons as well as for folds,
so invariance under them needs 2 proofs.}

\begin{lemma}
	\label{lem:planar moves}
	Two regular planar diagram are framed isotopic iff they are related by a finite sequence of regular planar isotopies and \textbf{planar moves},
	illustrated in Figures \ref{fig:planar moves} and \ref{fig:regular}.
\end{lemma}
\begin{proof}
	For points and ribbons, this is standard,
	e.g. \cite[Theorem 3.7]{joyalstreet1991braidedI} \cite[Theorem 6.1]{Shum1994Tortiletensorcategories}.
\fxnote{Cite Schommer-Pries appropriately.
	There are moves for arcs/circles/ribbons, and they're a tiny tad harder now because they need to pull the sheets around.
	And there are moves for the sheets: E.g. move an arc under a fold maximum/minimum or bifurcation/cusp, but also cancel two cusps.
	In that example, I guess I can first move all arcs away from cusps and then cancel them while they're not acting on anything? (The cusp cancellation area is contractible, so contract it and cancel it there.)
}
	For sheets, this is an exercise in Cerf theory,
	carried out in detail e.g. in \cite[Chapter 1.5]{Schommer-Pries:PhD}.
\end{proof}

\fxnote{Arrgh does that mean that I've ruled out Reidemeister I' because there a fold arc meets 2 pieces of ribbon, and one of them not tranversely?
Wait no. Reidemeister I' occurs \emph{precisely} because of the blackboard framing!}

\begin{remark}
	The Reidemeister I move,
	which equates a single twist by $2\pi$ of one ribbon to an untwisted ribbon,
	does not appear among the list of moves.
	This is intentional.
	While it is important for unframed knot and link invariants,
	it is not admissible in framed invariants since it changes the framing.
	The manifold invariant developed here is sensitive to the framing
	(otherwise it would not be able to distinguish $\CP{2}$ and $S^4$),
	and thus cannot validate the Reidemeister I move.

	In ribbon categories, this is mirrored algebraically by the fact that in the non-symmetric case,
	the twist is not trivial.
	In $G\times$-BSFCs, the twist is not even an endomorphism canonically.
	For a $g$-graded object $X \in \mathcal{C}_g$,
	it types as $\theta_X\colon X \to \acts{g}{X}$,
	and there is in general no coherence isomorphism to compare it to.
	(While $X$ and $\acts{g}{X}$ are isomorphic,
	they are not canonically isomorphic if $g$ is nontrivial.)
\end{remark}

We have drawn diagrams of Kirby diagrams before in Section \ref{sec:Kirby calculus}.
They are indeed diagrams in the sense of Definition \ref{def:unlabelled diagram}.
\begin{definition}
	\label{def:kirby diagram to planar diagram}
	A Kirby diagram $K$
	(in the sense of Definition \ref{def:kirby diagram})
	of $M$ straightforwardly gives rise to an unlabelled planar diagram $\mathcal{D}$,
	the \textbf{derived} planar diagram.
	\fxwarning{Should we call it D here if we're gonna call it K later all the time? Or make a point of the distinction?}
	It is computed by translating the attaching spheres to elements of the diagram:
	\begin{enumerate}
		\item Every 1-handle $h_1$ gives rise to two thickened points $D^3_\pm$,
			for each ball of the attaching region of $h_1$.
		\item A 2-handle gives rise to an embedded circle (a knot)
			if it is not attached to any 1-handles,
			or else to one or several lines incident on the thickened points corresponding to the 1-handles it is attached to.
			Possibly, an isotopy has to be applied such that the blackboard framing of the embedding matches the given framing of the attachment.
		\item A 3-handle gives rise to an embedded sphere
			if it is not attached to any handles,
			or else to one or several disks incident on the thickened points and ribbons corresponding to the 1-handles and 2-handles it is attached to.
			Possibly, an isotopy has to be applied such that in the projection,
			no disk ``changes the side'' of a 2-handle it is attached to
			(see Appendix \ref{sec:details on diagrams} for details).
	\end{enumerate}
\end{definition}
\begin{remark}
	\label{rem:isotopies of kirby diagrams and planar diagrams}
	Isotopies of Kirby diagrams correspond to isotopies of planar diagrams,
	so Lemma \ref{lem:planar moves} again applies:
	Two Kirby diagrams are isotopic if they are related by regular isotopies and planar moves.
\end{remark}

\fxerror{Draw the twist. Reference \cite[Remark 2.2.7]{DouglasReutter2018fusion} and \cite[Figure 46]{Schaumann:GrayCats}.
(Only do one sheet and say that the other goes analogously.)
Possibly do Reidemeister I' to show that everything cancels.}

\subsection{Labelled diagrams}
\label{sec:labelled diagrams}

If a regular diagram is labelled appropriately with data from a $G\times$-BSFC $\mathcal{C}$,
a morphism in $\mathcal{C}$ can be extracted.
This situation is set up in the following two definitions.
\begin{definition}
	A sheet with boundary on a ribbon $r$ is called \textbf{incident} to $r$.
	\fxwarning{Why a new word? Why not just ``attached''?}
	All sheets $\{(s_i,\pm)\}$ incident to $r$ are written as $\incident r$.
	They form an ordered multiset,
	starting at the top (viewed from the projection) right-hand (viewed from the ribbon orientation)
	sheet and going completely around the ribbon with the right-hand rule
	(starting, at first, into the drawing plane).
	The sign is $+$ if the boundary orientation of the sheet matches the orientation of the ribbon,
	and $-$ otherwise.

	Analogously, the ribbons $\{(r_i,\pm)\}$ incident to a thickened point $p$ are denoted as $\incident p$.
	This is a cyclically ordered set,
	starting anywhere on the boundary $S^1$ of the projected disk of $p$ and proceeding counter-clockwise.
	The sign encodes with which endpoint the ribbon attaches to the point.
	(It may attach with both ends.)
\end{definition}
\begin{definition}
	\label{def:labelling}
	A \textbf{labelling} of a diagram $\mathcal{D}$ with a $G$-crossed fusion category $\mathcal{C}$ consists of three functions with the following signatures:
	\begin{align}
		g &\colon \mathcal{D}_3 \to G\\
		X &\colon \mathcal{D}_2 \to \mathcal{O}(\mathcal{C})\\
		\iota &\colon \mathcal{D}_1 \to \mor \mathcal{C}
		\intertext{They need to satisfy the following \textbf{typing relations}:}
		\deg (X(r)) & = \prod_{(s,\pm) \in \incident r} g(s)^\pm\\
		\iota(p) &\in \left< \bigotimes_{(r, \pm) \in \incident p} X(r)^\pm \right>
	\end{align}
	\fxwarning{Notation for vector/morphism has only been introduced in the table.}
	Here, $g(s)^\pm$ denotes either $g(s)$ or $g(s)^{-1}$,
	depending on the sign of $(s, \pm)$.
	Similarly, by $X(r)^\pm$ we mean $X(r)$ for $+$ and $X(r)^*$ for $-$.
	The product is performed in the order specified in the previous definition.
\end{definition}
Informally, a labelling attaches group elements to sheets,
simple objects to ribbons,
and morphisms to the points.
The degree of the objects on a ribbon is given by the sheets incident to the ribbon,
and the morphism of a point must be in the morphism space of the objects labelling the incident ribbons.

To extract a morphism from a labelled diagram,
we project the diagram into the plane and mostly apply the well-known diagrammatic calculus of pivotal fusion categories,
treating the crossings and sheets as additional data.
These translate directly to $G$-crossed structures:
\begin{definition}
	\label{def:labelled diagram}
	A labelled, regular diagram $\mathcal{D}$ can be evaluated to an endomorphism of $\mathcal{I}$ (a complex number).
	To evaluate the diagram for given $g$, $X$ and $\iota$,
	follow this algorithm:
	\begin{enumerate}
		\item Starting from the back of the drawing plane,
			whenever a sheet $s$ covers ribbons or points,
			$g(s)^\pm$ acts on objects and morphisms labelling them.
			(The sign specifies whether the projection maps the sheet orientation onto the canonical orientation of the plane.)
		\item Insert a crossed braiding for an intersection point of two ribbons (Figure \ref{fig:braiding}),
			and appropriate $G$-crossed coherence isomorphism for every sheet incident on the right hand side of the overcrossing ribbon.
			(Figure \ref{fig:braiding with sheet the other way}).
		\item Insert appropriate $G$-crossed coherence isomorphisms at intersections involving at least one sheet fold arc.
			(Figure \ref{fig:G-coherences}).
			(For cusps, no morphism has to be inserted.)
		\item Interpret the resulting diagram in the graphical calculus of pivotal fusion categories.
	\end{enumerate}
	The resulting endomorphism of $\mathcal{I}$, and equivalently the corresponding number,
	is denoted as $\left<\mathcal{D}(g, X, \iota)\right>$.
	\fxwarning{Link to example calculations as soon as they're illustrated?}
\end{definition}
\begin{remark}
	Without difficulty,
	this definition can be generalised to open diagrams as in Section \ref{sec:spherical intro}.
	The result is then in a morphism space $\langle A \otimes B \otimes \cdots \rangle$,
	where $A$, $B$, \dots are the open-ended ribbons in the diagram.
\end{remark}
\begin{remark}
	Since $G\times$-BSFCs satisfy a coherence theorem
	(\cite{Mueger2010:StructureBraidedCrossedGCategories}, see also \cite[Theorem 2.3]{Cui2016TQFTs}),
	the choice of coherence isomorphism is always uniquely defined,
	given all group labellings $g$ and object labellings $X$.
\end{remark}
\begin{remark}
	A sheet incident on a point does not act on the morphism that labels the point,
	but only on the objects (resp. morphisms) labelling the ribbons (resp. points) which are \emph{covered} by the sheet in the projection.
\end{remark}
\begin{remark}
	Fold lines play a surprisingly small role in this calculus.
	It is usually helpful imagine fold lines to be labelled with the monoidal unit $I$ of the category.
	This makes it more intuitive why they do not influence the evaluation.
\end{remark}
Examples are found in Section \ref{sec:calculations}.

\tikzmath{
	\height      = 1.3;
	\width       = 0.8;
	\widthree    = 1.1;
	\labelheight = 0.6 * \height;
	\labelwidthh = 0.8 * \width;
}
\begin{figure}
	\centering
	\begin{tikzpicture} [ baseline ]
		\draw [ diagram, directed = 0.3 ]
			( \width   , -\height)
			node [ below ] {$r_2$}
			to [ out = 90, in = -90 ]
			(-\width   ,  \height);
		\draw [ white, line width = \diagramunderwhite ]
			(-\width   , -\height)
			to [ out = 90, in = -90 ]
			( \width   ,  \height);
		\path [ 3-handle fill ]
			( \width   , -\height)
			to [ out = 90, in = -90 ]
			(-\width   ,  \height)
			--
			(-\widthree,  \height)
			--
			(-\widthree, -\height)
			-- cycle;
		\draw [ thick, directed = 0.3 ]
			(-\width   , -\height)
			node [ below ] {$r_1$}
			to [ out = 90, in = -90 ]
			( \width   ,  \height);
		\path [ 3-handle fill = roch ]
			(-\width   , -\height)
			to [ out = 90, in = -90 ]
			( \width   ,  \height)
			--
			(-\widthree,  \height)
			--
			(-\widthree, -\height)
			-- cycle;
		\node at (0,  \labelheight) {$s_1$};
		\node at (0, -\labelheight) {$s_2$};
	\end{tikzpicture}
	$\xmapsto{\text{evaluation}}$
	\begin{tikzpicture} [ baseline ]
		\node [ plaquette ] (braiding) {$c$};
		\draw [ diagram, directed ]
			(-\width, -\height)
			node [ below ] {$X(r_1)$}
			-- (braiding)
		;
		\draw [ diagram, directed ]
			( \width, -\height)
			node [ below ] {$X(r_2)$}
			-- (braiding)
		;
		\draw [ diagram, opdirected ]
			(-\width,  \height)
			node [ above ] {$\!\!\!\!\!\acts{g(s_1)}{X(r_2)}$}
			-- (braiding)
		;
		\draw [ diagram, opdirected ]
			( \width,  \height)
			node [ above ] {$X(r_1)$}
			-- (braiding)
		;
	\end{tikzpicture}
	\caption{Evaluating the crossed braiding for the sheet labelling $g$ and the ribbon labelling $X$.
		The typing constraints demand that $\deg(X(r_1)) = g(s_1)$
		and $\deg(X(r_2)) = g(s_2)$.
		}
	\label{fig:braiding}
\end{figure}
\begin{figure}
	\centering
	\begin{tikzpicture} [ baseline ]
		\draw [ diagram, directed ]
			( \width   , -\height)
			node [ below ] {$r_2$}
			to [ out = 90, in = -90 ]
			(-\width   ,  \height);
		\draw [ white, line width = \diagramunderwhite ]
			(-\width   , -\height)
			to [ out = 90, in = -90 ]
			( \width   ,  \height);
		\path [ 3-handle fill ]
			( \width   , -\height)
			to [ out = 90, in = -90 ]
			(-\width   ,  \height)
			--
			(-\widthree,  \height)
			--
			(-\widthree, -\height)
			-- cycle;
		\draw [ thick, directed ]
			(-\width   , -\height)
			node [ below ] {$r_1$}
			to [ out = 90, in = -90 ]
			( \width   ,  \height);
		\path [ 3-handle fill = roch ]
			(-\width   , -\height)
			to [ out = 90, in = -90 ]
			( \width   ,  \height)
			--
			( \widthree,  \height)
			--
			( \widthree, -\height)
			-- cycle;
		\node at ( \labelwidthh, 0) {$s_1$};
		\node at (-\labelwidthh, 0) {$s_2$};
	\end{tikzpicture}
	$\xmapsto{\text{evaluation}}$
	\begin{tikzpicture} [ baseline ]
		\node [ plaquette ] (braiding)  at ( 0.2, -0.5) {$c$};
		\node [ plaquette ] (coherence) at (-0.4,  0.5) {$\epsilon \circ \eta$};
		\draw [ diagram, directed ]
			(-\width, -\height)
			node [ below ] {$X(r_1)$}
			-- (braiding)
		;
		\draw [ diagram, directed ]
			( \width, -\height)
			node [ below ] {$\acts{g(s_1)}{X(r_2)}$}
			-- (braiding)
		;
		\draw [ diagram, opdirected2 ]
			(-\width,  \height)
			node [ above ] {$X(r_2)$}
			-- (coherence)
			-- (braiding)
		;
		\draw [ diagram, opdirected ]
			( \width,  \height)
			node [ above ] {$X(r_1)$}
			-- (braiding)
		;
	\end{tikzpicture}
	\caption{For sheets on the right hand side of an overcrossing ribbon,
		an additional coherence has to be inserted.
		Note that here, $\deg(X(r_1)) = g(s_1)^{-1}$,
		as becomes apparent when considering a coordinate patch in which $s_1$ and $r_1$ are mapped on the upper half plane and the $x$-axis.
	}
	\label{fig:braiding with sheet the other way}
\end{figure}

\begin{figure}
	\centering
	\tikzmath{
		\coherenceheight = 2;
	}
	\begin{tikzpicture} [ baseline ]
		\pic {3-handle-over-2-handle={\coherenceheight}{0.3}};
	\end{tikzpicture}
	=
	\begin{tikzpicture} [ baseline ]
		\pic {3-handle-over-2-handle={\coherenceheight}{1.3}};
	\end{tikzpicture}
	\qquad
	$\xmapsto{\text{evaluate}}$
	\qquad
	\begin{tikzpicture}
		[ baseline
		, decoration =
			{ markings
			, mark=at position 0.25 with {\favouritearrow}
			, mark=at position 0.5  with {\favouritearrow}
			, mark=at position 0.75 with {\favouritearrow}
			}
		]
		\node [ plaquette ] (coherencelower) at (0,-0.9)                {$\eta^{-1}$};
		\node [ plaquette ] (coherenceupper) at ($-1*(coherencelower)$) {$\eta$};
		\draw
			[ diagram
			, postaction={decorate}
			]
			(0, -\coherenceheight)
			node [ below ] {$X(r)$}
			--
			(coherencelower)
			--
			node [ right ] {$\acts{g(s)^{-1}}{(\acts{g(s)}{X(r)})}$}
			(coherenceupper)
			--
			(0,  \coherenceheight)
			node [ above ] {$X(r)$}
		;
	\end{tikzpicture}
	=
	\begin{tikzpicture} [ baseline ]
		\draw [ diagram, directed ]
			(0, -\coherenceheight)
			node [ below ] {$X(r)$}
			--
			(0,  \coherenceheight)
			node [ above ] {$X(r)$}
		;
	\end{tikzpicture}
	\caption{
		An $s = S^2$ sheet covers the ribbon $r$.
		The covered interval of the ribbon is acted upon twice,
		with $g(s)^{-1}$ and $g(s)$.
		(Uniquely determined) $G$-coherence morphisms are inserted at the crossings of fold arc and ribbon.
		Once isotoped away, $s$ does not cover any other part of the diagram and thus acts as the identity.
	}
	\label{fig:G-coherences}
\end{figure}

\begin{figure}
	\centering
	\tikzmath{
		\twistheight = 2;
		\twistwidth  = 0.8;
		\intersectionslant = 25;
		\crossingBL = 270 - \intersectionslant;
		\crossingBR = 270 + \intersectionslant;
		\crossingTR =  90 - \intersectionslant;
		\crossingTL =  90 + \intersectionslant;
		\foldcrossingslant = 40;
		\foldcrossingBL = 270 - \foldcrossingslant;
		\foldcrossingBR = 270 + \foldcrossingslant;
		\foldcrossingTR =  90 - \foldcrossingslant;
		\foldcrossingTL =  90 + \foldcrossingslant;
		\baselineheight = -\twistheight / 2;
		}
	\begin{tikzpicture}[ looseness = 0.7, baseline = {(0, \baselineheight)} ]
		\coordinate (lowerend) at (0, -\twistheight);
		\coordinate (lowerfoldcrossing) at (0, -1);
		\coordinate (lower) at (0.1, -0.7);
		\coordinate (lowerturn) at ($(lower)+(0.4, 0)$);
		\coordinate (middle) at (0, -0.1);
		\coordinate (left) at (-\twistwidth, 0.1);
		\coordinate (foldleft) at (-0.2, -0.2);
		\coordinate (foldright) at (1, 0);

		\draw [ thick ]
			(lowerend) -- (lowerfoldcrossing)
			to [ out =  90, in = \crossingBL ]
			(lower)
			to [ out = \crossingTR, in =  90 ]
			(lowerturn);
		\draw [ diagramwhite ]
			(lowerturn)
			to [ out = 270, in = \crossingBR ]
			(lower)
			to [ out = \crossingTL, in = 270 ]
			(middle);

		\fill [ 3-handle fill ]
			(left |- lowerend)
			--
			(lowerend) -- (lowerfoldcrossing)
			to [ out =  90, in = \crossingBL ]
			(lower)
			to [ out = \crossingTR, in =  90 ]
			(lowerturn)
			--
			(foldright)
			to [ out =  90, in =   0]
			(left)
			-- cycle;

		\draw [ 3-handle extends ]
			(foldright)
			to [ out =  90, in =   0]
			(left)
			--
			(left |- lowerend)
			--
			(lowerend);

		\draw [ thick ]
			(lowerturn)
			to [ out = 270, in = \crossingBR ]
			(lower)
			to [ out = \crossingTL, in = 270 ]
			(middle);

		\fill [ 3-handle fill ]
			(foldright)
			-- (lowerturn)
			to [ out = 270, in = \crossingBR ]
			(lower)
			to [ out = \crossingBL, in =  90 ]
			(lowerfoldcrossing)
			to [ out = \foldcrossingBR, in = 270 ]
			(foldright)
			--
			(left |- foldright)
			-- cycle;

		\fill [ 3-handle fill ]
			(middle)
			to [ out = 270, in = \crossingTL ]
			(lower)
			to [ out = \crossingBL, in =  90 ]
			(lowerfoldcrossing)
			to [ out = \foldcrossingTL, in = 270 ]
			(foldleft)
			to [ out =  90, in = \foldcrossingTL ]
			(middle);

		\draw [ 3-handle extends ]
			(foldleft)
			to [ out =  90, in = \foldcrossingTL ]
			(middle);

		\draw [ 3-handle fold ]
			(foldright)
			to [ out = 270, in = \foldcrossingBR ]
			(lowerfoldcrossing)
			to [ out = \foldcrossingTL, in = 270 ]
			(foldleft);

		\draw [ 3-handle extends ]
			(foldleft)
			to [ out = 270, in = 270 ]
			(foldright);

		\fill [ 3-handle fill ]
			(foldright)
			to [ out = 270, in = \foldcrossingBR ]
			(lowerfoldcrossing)
			to [ out = \foldcrossingTL, in = 270 ]
			(foldleft)
			to [ out = 270, in = 270 ]
			(foldright);

		\node [ below ] at (lowerend) {$r$};
		\node [ above right ] at (left |- lowerend) {$s$};
	\end{tikzpicture}
	$\qquad\xrightarrow{\text{evaluate}}\qquad$
	\begin{tikzpicture}[ baseline = {(0, \baselineheight)} ]
		\node [ plaquette, inner sep = 0.1 ]
			(twist) at (0, -\twistheight/2) {$\theta_{X(r)}$};
		\draw [ diagram, directed2 ]
			(0, -\twistheight)
			node [ below ] {$X(r)$}
			--
			(twist)
			--
			(0, 0)
				node [ above ] {$\acts{g(s)}{X(r)}$};
	\end{tikzpicture}
	\caption{The twist is an isomorphism,
		but not canonically an automorphism.
		This can be seen from the different covering of sheets on the ends of the ribbon:
		When the ribbon performs a full twist,
		the sheet wraps it once fully.
		This adds a group action on the target object.
	}
	\label{fig:twist}
\end{figure}

\begin{lemma}
	The evaluation of a labelled diagram is invariant under isotopies of regular diagrams
	(Definition \ref{def:regular isotopy}).
\end{lemma}
\begin{proof}
	Regular isotopies do not change the topology of the projected diagram in the plane.
	By coherence of the diagrammatic calculus of pivotal fusion categories,
	the evaluation is invariant.
\end{proof}

\begin{proposition}
	\label{prop:diagram evaluation invariant under isotopy}
	The evaluation of a diagram is invariant under arbitrary (framed) isotopies of diagrams.
\end{proposition}
\begin{proof}
	By Lemma \ref{lem:planar moves} and the previous lemma,
	we need to prove invariance of the diagram evaluation from Definition \ref{def:labelled diagram} under each planar move.
	\begin{description}
		\item[Point isotopy]
			Naturality of the crossed braiding.
		\item[Reidemeister I']
			By definition of the twist, see e.g. \cite[Figure 4a]{Cui2016TQFTs}.
		\item[Reidemeister II]
			By the definition of the graphical representation of the inverse braiding
			(in the case of ribbons),
			or the inverse of the coherence $\eta$,
			as seen in Figure \ref{fig:G-coherences}.
		\item[Reidemeister III]
			By the $G$-crossed braid axiom,
			sometimes also called the ``heptagon axiom'',
			as in \cite[Definition 2.2.2]{Cui2016TQFTs}.
		\item[Cusp isotopy]
			The left hand side of this move evaluates to two cancelling coherence isomorphisms.
		\item[Cusp cancellation]
			When the folded sheet parts do not cover any other part of the diagram
			(as can be achieved with a combination of other moves such as cusp isotopies),
			then these parts do not change the diagram.
	\end{description}
\end{proof}

\begin{remark}
	\label{rem:generalise labellings}
	Once a sheet labelling $g$ and a ribbon labelling $X$ are fixed,
	a labelling $\iota$ of points with morphisms gives rise to an elementary tensor in the tensor product of all morphism spaces in the diagram:
	\begin{equation}
		\bigotimes_{p \in \mathcal{D}_1} \iota(p) \in \bigotimes_{p \in \mathcal{D}_1} \left< \bigotimes_{(r, \pm) \in \incident p} X(r)^\pm \right>
	\end{equation}
	It will turn out that arbitrary vectors in the right hand side vector space are often relevant,
	so we will generalise the notion of morphism labellings to arbitrary vectors,
	not just elementary tensors.
	The evaluation from the previous definition can be generalised uniquely to such vectors,
	by the universal property of the tensor product.

	Similarly, object labellings $X$ are generalised to a labelling with elements from the fusion algebra $\Co[\mathcal{C}]$.
	From now on, we will often implicitly make use of these two generalisations.
\end{remark}

\subsection{Kirby colours and sliding lemmas}
\label{sec:sliding lemmas}

Let us fix some notation and prove the essential lemmas for the invariant definition in Section \ref{sec:invariant}.
For this section, assume $\mathcal{C}$ to be a $G\times$-BSFC.

\subsubsection{$G$-graded fusion categories}

Recall that the set of (chosen representatives of equivalence classes of) simple objects in $\mathcal{C}$ is denoted by $\mathcal{O}(\mathcal{C})$,
and the same notation applies to any semisimple linear category.
Set cardinality is denoted as $\lvert \mathcal{O}(\mathcal{C}) \rvert$.
\begin{definition}[{\cite[Definition 6.6]{TuraevHFTAndGCats2000}}]
	The \textbf{Kirby colour} of degree $g$ is defined as the following element of the fusion algebra $\Co[\mathcal{C}]$:
	\begin{equation}
		\Omega_g \coloneqq \Omega_{\mathcal{C}_g} = \bigoplus_{X \in \mathcal{O}(\mathcal{C}_g)} \qdim{X} X
	\end{equation}
\end{definition}
\begin{remark}
	Note that $\Omega_\mathcal{C} = \bigoplus_{g \in G} \Omega_g$.
\end{remark}
\begin{remark}
	\label{rem:graded kirby colour dual}
	Unlike $\Omega_\mathcal{C}$, $\Omega_g$ is not self-dual.
	In fact, sphericality implies $\Omega_g^* \cong \Omega_{g^{-1}}$.
	This implies that in the graphical calculus,
	the orientation of ribbons labelled $\Omega_g$ needs to be specified.
\end{remark}
\begin{lemma}
	\label{lem:global dimension degree}
	Let $\mathcal{C}_g \not\simeq 0$.
	Then $\qdim{\Omega_g} = \qdim{\Omega_e}$.
\end{lemma}
\begin{proof}
	Jumping slightly ahead and using Lemma \ref{lem:graded sliding lemma},
	we slide a loop labelled with $\Omega_g$ over an $\Omega_e$-loop and find $\qdim{\Omega_g}^2 = \qdim{\Omega_g}\qdim{\Omega_e}$.
	Then $\qdim{\Omega_g}$ can be can cancelled since it is nonzero.
	For the analogous lemma with Frobenius-Perron dimensions,
	see \cite[Proposition 8.20]{OnFusionCategories}.
\end{proof}
\begin{remark}
	For a faithful grading
	($\mathcal{C}_g \not\simeq 0$ for every $g \in G$),
	the previous lemma is equivalent to $\qdim{\Omega_g} \cdot \lvert G \rvert = \qdim{\Omega_\mathcal{C}}$.
	But since a non-faithful $G$-grading on $\mathcal{C}$ is always given by a subgroup of $H \subset G$
	and a faithful $H$-grading,
	we can still leverage the equation $\qdim{\Omega_g} \cdot \lvert H \rvert = \qdim{\Omega_\mathcal{C}}$ if this subgroup is known.
\end{remark}

\subsubsection{The $G$-crossed braiding and encirclings}

\begin{definition}
	Let $h, g \in G$.
	The double braiding of two objects $A \in \mathcal{C}_h$ and $B \in \mathcal{C}_g$ is defined as:
	\begin{align}
		\beta_{A,B} &\colon A \otimes B \to \acts{hgh^{-1}}{A} \otimes \acts{h}{B}\\\nonumber
		\beta_{A,B} &\coloneqq (\eta_B \otimes 1_B) \circ c_{B,A} \circ c_{A,B}
	\end{align}
	\fxwarning{Weird alignment}
	$\eta_B$ is the unique coherence of the $G$-action.
	\fxwarning{maybe just work it out}
	Note that $\beta_{A,B}$ is, up to $G\times$-coherences,
	an automorphism if $h = g = e$.
\end{definition}

\begin{definition}
	Let $g \in G$.
	The \textbf{encircling} of an object $A \in \mathcal{C}_e$ by an object $B \in \mathcal{C}_g$ is defined by the following partial trace:
	\begin{align}
		\Delta_{A,B} &\colon A \to \acts{g}{A}\\\nonumber
		\Delta_{A,B} &\coloneqq \tr_B((1_A \otimes \epsilon_B) \circ \beta_{A,B}) \colon A \to \acts{g}{A}
	\end{align}
	$B$ can be generalised to ($g$-graded) elements of the fusion algebra straightforwardly,
	and we will freely make use of this generalisation.
\end{definition}

\begin{lemma}
	\label{lem:graded sliding lemma}
	Assume $A \in \mathcal{C}_e$, $h, g \in G$ and $B \in \mathcal{C}_g$.
	There is a \textbf{graded sliding lemma} for $G\times$-BSFCs,
	which changes the grade of the encircling:
	\begin{equation}
		\Delta_{A,\Omega_h} \otimes 1_B
		= (\eta_A \otimes 1_B) \circ c_{B,A} \circ \left(\epsilon_B \otimes \Delta_{A,\Omega_{g^{-1}h}} \right) \circ c_{A,B}
		\label{eq:graded sliding lemma}
	\end{equation}
	\fxwarning{Coherences inserted here?}
	The encirclement may be arbitrarily linked or knotted,
	as far as this is possible with regard to the grading.
\end{lemma}
\begin{proof}
	In analogy to usual proofs of sliding lemmas
	(e.g. \cite[Corollary 3.5]{Kirillov2011:Stringnet} and \cite[Lemmas 3.3 and 3.4]{BaerenzBarrett2016Dichromatic}),
	\fxwarning{
		I prove something slightly different,
		so maybe Bruce's Lemma 6.5 from string diagrams
		or Altschüler-Bruguières are better sources.
	}
	$B^*$ and $\Omega_h$ are fused to a single strand,
	but since $B \in \ob \mathcal{C}_{g}$,
	this strand has to be labelled $\Omega_{g^{-1}h}$.
	We have used \cite[Lemma 6.6.1]{TuraevHFTAndGCats2000}.
\end{proof}
\begin{remark}
	To see why the previous lemma is called ``sliding lemma'',
	revisit Figure \ref{fig:k-2-slides},
	which contains graphical representations of the two sides of \eqref{eq:graded sliding lemma}
	if the 3-handle $h_{3A}$ is labelled with $g^{-1}h$,
	$h_{3B}$ with $h$,
	$h_{2a}$ with $B$,
	and the gray area is replaced by a ribbon labelled with $A$.
	\fxwarning{It might be better to make a new diagram instead?}
\end{remark}

\begin{definition}[{Well-known, e.g. \cite[Definition 2.41]{BaerenzBarrett2016Dichromatic}}]
	Let $\mathcal{D}$ be a braided fusion category.
	Then $\mathcal{D}'$ is the full symmetric subcategory spanned by trivially braiding objects,
	called the \textbf{symmetric centre}.

	Let $A$ be an object in $\mathcal{D}$.
	Then $A'$ is defined (up to isomorphism)
	to be the maximal subobject of $A$ in $\mathcal{D}'$,
	and $\tau_A\colon A \to A$ the idempotent defined as projection onto $A'$ followed by inclusion into $A$
	(not depending on the choice of $A'$).
	If $\tau_A = 1_A$,
	or equivalently $A \cong A'$,
	or $A \in \ob \mathcal{D}'$,
	then is $A$ is said to be \textbf{transparent}.
\end{definition}
\begin{lemma}[{Well-known, e.g. \cite[Lemma 2.46]{BaerenzBarrett2016Dichromatic}}]
	\label{lem:killing}
	Let $\mathcal{D}$ be a ribbon fusion category and $A$ an object therein.
	Then $\Delta_{A,\Omega_\mathcal{D}} = \tau_A \cdot \qdim{\Omega_\mathcal{D}}$.
	In particular, let $X$ be a simple object in $\mathcal{D}$.
	Then $\Delta_{X,\Omega_\mathcal{D}} = 1_X \cdot \qdim{\Omega_\mathcal{D}}$ iff $X$ is transparent,
	and 0 otherwise.

	This lemma is called the \textbf{killing lemma},
	since nontransparent $X$ are ``killed off'' by an encircling with $\Omega_\mathcal{D}$.
\end{lemma}

\begin{lemma}
	\label{lem:killing G-crossed}
	To the knowledge of the author,
	this generalisation of the killing lemma has not been discussed in the literature before.

	Assume that $g \in G$ and $\mathcal{C}_g \not\simeq 0$.
	Then the \textbf{$G\times$-killing lemma} holds for any $A \in \ob \mathcal{C}_e$:
	\begin{equation}
		\Delta_{A, \Omega_{g^{-1}}} \circ \Delta_{A, \Omega_g}
		= \tau_A \cdot \qdim{\Omega_e}^2
	\end{equation}
\end{lemma}

\begin{proof}
	The left hand side is represented diagrammatically by a closed cylinder encircling an $A$-labelled ribbon,
	with the sheet labelled $g$.
	The $\Omega_{g^{-1}}$-encircling can be slid off the $\Omega_g$-encircling,
	\fxwarning{Cite the result that we can insert a complete $\Omega_e$.
		Insertion lemma?}
	yielding factor $\qdim{\Omega_{g^{-1}}} = \qdim{\Omega_g}$.
	This changes the grade of the remaining encircling to $\Omega_e$,
	thus we can apply the killing lemma \ref{lem:killing}.
\end{proof}
\begin{corollary}
	\label{cor:killing iso}
	Let $X \in \mathcal{C}_e$ be simple.
	Then encircling $\Delta_{X, \Omega_g}$ is an isomorphism $X \to \acts{g}{X}$ iff $X$ is transparent
	(the inverse being $\Delta_{X, \Omega_{g^{-1}}}\cdot \qdim{\Omega_e}^{-2}$),
	and 0 otherwise.
	This justifies generalising the name ``killing lemma''.
\end{corollary}

\fxwarning{Some extra blabla subsection missing about $G$-crossed coherences?}

\section{The invariant}
\label{sec:invariant}

In this section, we will define an invariant of closed, smooth, oriented 4-manifolds
by labelling Kirby diagrams (Section \ref{sec:Kirby calculus})
with data from $G\times$-BSFCs (Sections \ref{sec:Gx intro} and \ref{sec:labelled diagrams}),
and show its independence of the chosen Kirby diagram by means of the lemmas from Section \ref{sec:sliding lemmas}.

Fix again a $G\times$-BSFC $\mathcal{C}$,
a manifold $M$,
and a Kirby diagram $K$ for $M$.
Recall from Definition \ref{def:kirby diagram to planar diagram}
that $K$ gives rise to an unlabelled planar diagram $\mathcal{D}$.
The 1-, 2- and 3-handles of $K$ can be labelled appropriately to yield a labelling of the derived diagram $\mathcal{D}$.
In detail, we require the following data:
\begin{definition}[Labelling of Kirby diagrams]
	Denote by $K_j$ the set of $j$-handles.
	A labelling of $K$ by $\mathcal{C}$ is specified by:
	\begin{itemize}
		\item A function $g\colon K_3 \to G$,
		\item a function $X\colon K_2 \to \Co[\mathcal{C}]$ such that the derived labelling of $\mathcal{D}$ type-checks,
		\item a function $\iota\colon K_1 \to \mor \mathcal{C} \otimes \mor \mathcal{C}$ into type-checking dual morphism spaces:
			\[\iota(h_1) \in \langle \otimes_{(h_2,\pm) \incident h_1} X(h_2)^\pm \rangle \otimes \langle \otimes_{(h_2,\pm) \incident^{-1} h_1} X(h_2)^\mp \rangle\]
			Here, $\incident h_1$ denote the attached 2-handles,
			and $\incident^{-1}$ denotes the reverse cyclical ordering.
	\end{itemize}
\end{definition}
\begin{remark}
	Crucially, the handles of the Kirby diagrams receive labels,
	and not each individual element of the diagram.
	It is important that the different lines stemming from a single 2-handle are labelled with the same data,
	and the same is true for the different disks from a single 3-handle.
\end{remark}

\fxnote{Be careful where in the following process we need the coherences from $\mathcal{C}$,
	or the ones from the monoidal bicategory.
	This could be the key to finding where the cohomology class shows up.
}

The invariant is now defined as a sum of diagram evaluations over all possible 3-handle labellings,
assigning the appropriately graded Kirby colour to every 2-handle and the dual bases from Definition \ref{def:dual bases} to every 1-handle.
\fxwarning{Would be better if dual bases had a specific definition I could refer to}
\begin{definition}
	\label{def:invariant}
	Making full use of the generalisations from Remark \ref{rem:generalise labellings},
	the invariant assigned to a $G\times$-BSFC $\mathcal{C}$ is defined as:
	\begin{equation}
		I_\mathcal{C}(K)
			\coloneqq \sum_{g\colon K_3 \to G}
			\left<K\left(g, h_2 \mapsto \Omega_{\deg(h_2)}, \bigotimes_{h_1} \sum_i \phi_{h_1,i} \otimes \tilde{\phi}_{h_1,i} \right)\right>
			\qdim{\Omega_\mathcal{C}}^{\lvert K_1 \rvert - \lvert K_2 \rvert}
	\end{equation}
	By abuse of notation, we have used $\deg(h_2) \coloneqq \prod_{(g,\pm) \in \delta h_2} g^\pm$,
	which is the degree of a type-checking object labelling a 2-handle $h_2$.
	$\sum_i \phi_{h_1,i} \otimes \tilde{\phi}_{h_1,i}$ denotes the sum over the dual bases of $\langle s(h_1) \rangle$ and  $\langle s(h_1)^* \rangle$.
	Writing out the sums explicitly is possible:
	\begin{align}
		I_\mathcal{C}(K)
			&= \sum_{g\colon K_3 \to G
			  }
			  \sum_{\substack{X\colon K_2 \to \mathcal{O}(\mathcal{C})
			     \\ \deg(X(h_2)) = \deg(h_2)
			  }}
			  \sum_{\substack{\phi_{h_1, i}, \tilde{\phi}_{h_1,i}
			     \\ \forall h_1 \in K_1
			  }}
			  \left< K\left(g, X, \bigotimes_{h_1} \phi_{h_1,i} \otimes \tilde{\phi}_{h_1,i}\right) \right>
			  \nonumber\\
			&\qquad \cdot \qdim{\Omega_\mathcal{C}}^{\lvert K_1 \rvert - \lvert K_2 \rvert}
			  \cdot \prod_{h_2 \in K_2} \qdim{X(h_2)}
		\label{eq:invariant long}
	\end{align}
	\fxwarning{Does this still look ugly? Inspiration from dichro?}
\end{definition}

\begin{theorem}
	\label{thm:Invariance}
	For a given manifold $M$,
	$I_\mathcal{C}$ does not depend on the choice of Kirby diagram $K$,
	in other words, it is an invariant of smooth, closed, oriented 4-manifolds.
\end{theorem}
\begin{proof}
	We have to prove the following lemmas about $I_\mathcal{C}$:
	\begin{itemize}
		\item Independence of orientation choices
			(Lemma \ref{lem:independence orientations}).
		\item Invariance under isotopies of the diagram
			(Lemma \ref{lem:invariance isotopies}).
		\item Invariance under handle slides
			(Lemma \ref{lem:invariance slides}).
		\item Invariance under handle cancellations
			(Lemma \ref{lem:invariance cancellations}).
	\end{itemize}
\end{proof}
\begin{lemma}
	\label{lem:independence orientations}
	The invariant does not depend on the choice of orientations of the attaching spheres.
\end{lemma}
\begin{proof}
	Let us verify the statement for the attaching spheres of $k$-handles for all relevant values of $k$:
	\handleenumerationstyle
	\begin{enumerate}
		\item A change in orientation of $S^0$ results in the exchange of the two attaching disks of a 1-handle.
			In Definition \ref{def:dual bases},
			we can simply exchange the roles of $\iota$ and $\tilde{\iota}$
			without changing the diagram.
		\item Changing the orientation of an attaching $S^1$ of a 2-handle $h_2$
			dualises its labelling object,
			but also inverts the grade $\deg(h_2)$ of the incident 3-handles
			(reusing the notation from Definition \ref{def:invariant}).
			From Remark \ref{rem:graded kirby colour dual},
			we know that $\Omega_{\deg(h_2)}^* \cong \Omega_{\deg(h_2)^{-1}}$,
			thus the evaluation is independent of this choice.
		\item Since inversion is an involution on $G$,
			we can reindex the summation over 3-handle labellings $g$
			and redefine $g(h_3) \mapsto g(h_3)^{-1}$ for the 3-handle $h_3$ whose attaching sphere was reoriented.
	\end{enumerate}
\end{proof}
\begin{lemma}
	\label{lem:invariance isotopies}
	The invariant does not change if an isotopy is applied to a $k$-handle.
\end{lemma}
\begin{proof}
	Essentially, this is Proposition \ref{prop:diagram evaluation invariant under isotopy},
	but translated to the labelled planar diagram derived from a Kirby diagram.
	If the isotopy is regular, the diagram in the pivotal category does not change.

	Let us retrace the proof for planar moves occurring from a $k$-handle isotopy,
	for any value of $k$:
	\handleenumerationstyle
	\begin{enumerate}
		\item Isotopy on a 1-handle (possibly with attachments)
			may cause point isotopies,
			which are covered by naturality of the crossed braiding.
		\item The braid, or heptagon, axioms
			(e.g. \cite[(6) and (7)]{Cui2016TQFTs})
			and naturality of the crossed braiding
			ensure invariance under isotopies involving or changing the crossings.
		\item Sliding an attachment of a 3-handle $h_3$ under a part of the diagram does not change the extracted diagram.
			Sliding it over a part of the diagram acts on said part with $g(h_3)^{-1}$,
			followed by an action with $g(h_3)$.
			$G\times$-coherences and their inverses are inserted where the fold arcs of the attaching sphere crosses the remaining diagram.
			We can cancel these natural isomorphisms and recover the original diagram.
			(See e.g. Figure \ref{fig:G-coherences},
			where the two coherences $\eta$ and $\eta^{-1}$ can be cancelled.)
	\end{enumerate}
\end{proof}
It is useful to keep the following fact in mind for further proofs and calculations.
\begin{remark}
	\label{rem:3-handle closed diagram}
	Recall that $G$ acts on $\mathcal{C}$ via monoidal automorphisms.
	This implies that on a \emph{closed diagram},
	the group action from a 3-handle attachment is trivial,
	since a closed diagram corresponds to an endomorphism of the monoidal unit $\mathcal{I}$.
\end{remark}
\begin{lemma}
	\label{lem:invariance slides}
	The invariant does not change if any handle slide is applied to the Kirby diagram.
\end{lemma}
\begin{proof}
	We prove invariance under each $j$-$k$-slide:
	\begin{description}
		\item[1-1]
			Invariance under this slide is already satisfied for each 3-handle- and 2-handle-labelling individually.
			It is proven by Corollary \ref{cor:1-1-slide} in the appendix,
			with $f$ the label of the 1-handle attachment being slid,
			and $\alpha$ the 1-handle to be slid over.
		\item[2-1] Assume, without loss of generality,
			that the 2-handle $S^1$ is only slid through halfways.
			This is again Corollary \ref{cor:1-1-slide},
			where $X = X(h_2)^* \otimes X(h_2)$ and $f$ is the duality coevaluation.
		\item[2-2] This is the graded sliding lemma,
			\ref{lem:graded sliding lemma}.
		\item[3-1] This does not change the extracted diagram in $\mathcal{C}$,
			since attaching a 3-handle to a 1-handle has no effect in the evaluation.
			(If, in the process of sliding,
			 the 3-handle attachment moves above another part of the diagram,
			 Lemma \ref{lem:invariance isotopies} applies.)
		\item[3-2] Up to $G\times$-coherences,
			this doesn't change the extracted diagram.
		\item[3-3] Assume $h_3$ slides over $h_3'$.
			Up to $G\times$-coherences,
			this simply multiplies the labelling $g(h_3')$ by $g(h_3)$.
			Therefore, we can reindex the sum over 3-handle-labellings accordingly
			(since group multiplication is a set isomorphism)
			and recover the original value.
		\item[3-$\infty$] Recall Remark \ref{rem:3-infty-move} to visualise this move.
			Attaching a new 3-handle at infinity that acts on the whole diagram leaves the evaluation invariant
			(Remark \ref{rem:3-handle closed diagram}).
			Sliding over it is an invariant again,
			as we just proved.
			Removing the 3-handle at infinity again does not change the evaluation.
	\end{description}
\end{proof}
In order to prove invariance under handle cancellations,
the following lemma is useful.
\begin{lemma}
	\label{lem:multiplicativity diagrams}
	Let $K_1$ and $K_2$ be Kirby diagrams.
	$I_\mathcal{C}$ is multiplicative under disjoint union of diagrams:
	\begin{equation}
		I_\mathcal{C}(K_1 \sqcup K_2) = I_\mathcal{C}(K_1) \cdot I_\mathcal{C}(K_2)
	\end{equation}
\end{lemma}
\begin{proof}
	Evaluating the disjoint union of two diagrams results in the monoidal product of the evaluations.
	Since both diagrams are closed,
	they evaluate to endomorphisms of $\mathcal{I}$,
	so the product is simply the multiplication in $\Co$.
	This shows $\langle (K_1 \sqcup K_2) (g, X, \iota) \rangle = \langle K_1  (g, X, \iota) \rangle \cdot \langle K_2 (g, X, \iota) \rangle$,
	from which the claim follows readily.
\end{proof}
\begin{lemma}
	\label{lem:invariance cancellations}
	The invariant does not change if cancelling handle pairs are removed from, or added to the Kirby diagram.
\end{lemma}
\begin{proof}
	Perform all necessary handle slides such that the cancelling handle pair is disconnected from the remaining diagram.
	By Lemma \ref{lem:multiplicativity diagrams},
	we only need to show that the diagram of the cancelling pair evaluates to 1.
	This is done in the following for the two relevant cancellations:
	\begin{description}
		\item[1-2]
			For a simple object $X \in \ob \mathcal{C}$ labelling the 2-handle,
			the morphism space $\langle X \rangle$ assigned to a 1-handle attaching disk is $\Co$ iff $X = \mathcal{I}$,
			and 0 otherwise.
			In the case it is $\mathcal{I}$,
			the basis consists of a single vector,
			so the sum ranges over a single summand of value 1
			(after having noted that the exponent of the normalisation cancels).
		\item[2-3] After consulting Figure \ref{fig:2-3-cancellation} and recalling that the closed loop is a diagram for the categorical dimension,
			it is apparent that the diagram evaluates to $\sum_g \qdim{\Omega_g} = \qdim{\Omega_\mathcal{C}}$,
			cancelling the normalisation.
	\end{description}
	Note that we assume all 0-1-cancellations and 3-4-cancellations to have taken place already.
\end{proof}

\section{Calculations}
\label{sec:calculations}

As an advantage over state sum models and Hamiltonian formulations
(to which the connection will be made in Section \ref{sec:ssm}),
it is much easier to calculate explicit values of the invariant.
This is mainly because handle decompositions are a very succinct description for smooth manifolds,
but also because the graphical calculus of $G\times$-BSFCs matches Kirby diagrams closely.

This observation was already made in \cite{BaerenzBarrett2016Dichromatic} for Kirby diagrams without 3-handles,
where the Crane-Yetter invariant was calculated for several 4-manifolds.
Due to the following proposition, several results can be recovered for the invariant presented here:
\begin{proposition}
	\label{prop:Crane-Yetter}
	Let $M$ be a 4-manifold with a handle decomposition that does not contain any 3-handles,
	and let $K$ be a Kirby diagram for this decomposition.
	Then its invariant from Definition \ref{def:invariant} is equal to its renormalised Crane-Yetter invariant $\widehat{CY}$ \cite[Proposition 6.1.1]{BaerenzBarrett2016Dichromatic} for the trivial degree,
	up to a factor involving the Euler characteristic:
	\begin{equation}
		I_\mathcal{C}(M)
			= \widehat{CY}_{\mathcal{C}_e}(M)
			\cdot \left( \frac{\qdim{\Omega_\mathcal{C}}}{\qdim{\Omega_{\mathcal{C}_e}}}\right)^{2 - \chi(M)}
	\end{equation}
	In particular, if $\mathcal{C}$ is concentrated in the trivial degree, they coincide.
\end{proposition}
\begin{proof}
	Without 3-handles,
	all 2-handles are labelled with $\Omega_e$,
	and there is no $G$-action on any part of the diagram.
	Reminding ourselves that the number of 3-handles $\lvert K_3 \rvert$ is zero,
	furthermore $\lvert K_0 \rvert = \lvert K_4 \rvert = 1$,
	and therefore here $\chi(M) = \lvert K_0 \rvert - \lvert K_1 \rvert + \lvert K_2 \rvert - \lvert K_3 \rvert + \lvert K_4 \rvert = 2 - \lvert K_1 \rvert + \lvert K_2 \rvert$,
	we compare to Definition \ref{def:invariant} and \cite[Propositions 4.13 and 6.1.1]{BaerenzBarrett2016Dichromatic}.
\end{proof}
\fxwarning{I believe that such a manifold is always simply connected (because it can't contain non-cancelling 1-handles either?),
	but it's an open question whether all simply connected ones look like this.}

\subsection{Example manifolds}
\label{sec:example manifolds}
The following perspective was pointed out to the author by Ehud Meir:
If we fix a particular manifold $M^4$ and vary $\mathcal{C}$,
the quantity $I_\mathcal{C}(M)$ becomes an invariant of $G\times$-BSFCs.
We will show here that many known invariants of fusion categories can be recovered by choosing the correct manifold.
Table \ref{table:example values} summarises the results.
\begin{table}
	\centering
	\begin{tabular}{lllll}
		\toprule
		Manifold
			& Invariant value
				& $\chi(M)$
					& $\sigma(M)$
						& $\pi_1(M)$
		\\\midrule
		$S^4$ (including exotic candidates)
			& 1
				& 2
					& 0
						& 1
		\\\midrule
		$S^1 \times S^3$
			& $\lvert G \rvert \cdot \qdim{\Omega_\mathcal{C}}$
				& 0
					& 0
						& $\Z$
		\\\midrule
		$S^2 \times S^2$, $\CP{2} \connsum \opCP{2}$
			& $\qdim{\Omega_e'} \cdot \qdim{\Omega_e} \cdot \qdim{\Omega_\mathcal{C}}^{-2}$
				& 4
					& 0
						& 1
		\\\midrule
		$\CP{2}$
			& $\sum_{X \in \mathcal{O}(\mathcal{C}_e)}\qdim{X}^2 \theta_X \cdot \qdim{\Omega_\mathcal{C}}^{-1}$
				& 3
					& 1
						& 1
		\\\midrule
		$\opCP{2}$
			& $\sum_{X \in \mathcal{O}(\mathcal{C}_e)}\qdim{X}^2 \theta_X^{-1} \cdot \qdim{\Omega_\mathcal{C}}^{-1}$
				& 3
					& -1
						& 1
		\\\midrule
		$S^1 \times S^1 \times S^2$ (faithful grading)
			& $\lvert G \rvert^2 \cdot \lvert \mathcal{O}(\mathcal{C}_e') \rvert \cdot \qdim{\Omega_e}$
				& 0
					& 0
						& $\Z \oplus \Z$
		\\\midrule
		$S^1 \times S^3 \connsum S^1 \times S^3 \connsum S^2 \times S^2$
			& $\lvert G \rvert^2 \cdot \qdim{\Omega_e'} \cdot \qdim{\Omega_e}$
				& 0
					& 0
						& $\Z * \Z$
		\\\bottomrule
	\end{tabular}
	\caption{
		Example values for the invariant,
		in comparison with classical invariants.
		($\Z * \Z$ denotes the free product of $\Z$ with itself,
		 i.e. the free group on two generators.)
		}
	\label{table:example values}
\end{table}

\begin{anfxwarning}{The following doesn't make so much sense,
	since the input data is just different.
	Need a pair of manifolds that can be distinguished here that can't with CY.
	Maybe something with $\pi_3$,
	or with the same $\pi_2$ and $\pi_1$, but different action.}
\small
We can already deduce from it that the invariant reflects more information about the input category than the Crane-Yetter invariant \cite[Tables 3 and 4]{BaerenzBarrett2016Dichromatic}.
For example, in the extension of a modular category,
it can recover the dimension of the original category $\mathcal{C}_e$ and the size of the group.
\end{anfxwarning}

\subsubsection{$S^1 \times S^3$}

Recall the Kirby picture for this manifold from Figure \ref{fig:S1S3}.
The attaching sphere of the 3-handle, labelled with a group element $g$,
acts trivially in the diagram according to Remark \ref{rem:3-handle closed diagram}.
There is a single 1-handle and no 2-handle present.
Summing over the 1-dimensional morphism space of the monoidal identity yields:
\begin{align}
	I_\mathcal{C}(S^1 \times S^3)
	&= \sum_{\substack{g \in G \\ \iota \in \langle \mathcal{I} \rangle}}
		\tikzbo{
			\node[plaquette]             {$\iota$};
			\node[plaquette] at (0.8, 0) {$\tilde\iota$};
		}
		\cdot \qdim{\Omega_\mathcal{C}}
	\\
	&= \sum_{g \in G}
		\qdim{\Omega_\mathcal{C}}
	\\
	&= \lvert G \rvert \qdim{\Omega_{\mathcal{C}}}
\end{align}

\subsubsection{$S^2 \times S^2$}

Since $S^2$ admits a standard handle decomposition with a single 0-handle and a single 2-handle,
the product handle decomposition of $S^2 \times S^2$ only consists of a 0-handle, two 2-handles and a 4-handle,
so Proposition \ref{prop:Crane-Yetter} is applicable
and we can follow the calculation from \cite[(6.1.3)]{BaerenzBarrett2016Dichromatic}
to arrive at:
\begin{equation}
	I_\mathcal{C}(S^2 \times S^2) = \frac{\qdim{\Omega_{\mathcal{C}_e}'}\qdim{\Omega_{\mathcal{C}_e}}}{\qdim{\Omega_\mathcal{C}}^2}
\end{equation}
For the details of the calculation,
assume as Kirby diagram for $S^2 \times S^2$ the Hopf link of two 2-handle attachments \cite[Figure 4.30]{GompfStipsicz}.
Both 2-handles are labelled with $\Omega_{\mathcal{C}_e}$.
Then by the killing lemma \ref{lem:killing},
one of them only contributes with the symmetric centre $\Omega_{\mathcal{C}_e}'$.
Then the diagram unlinks (since any object in $\Omega_{\mathcal{C}_e}'$ braids trivially with any other object),
and we arrive at the desired value after including the normalisation.

\subsubsection{Complex projective plane and Gauss sums}
\label{sec:CP2 and Gauss sums}

The complex projective plane $\CP 2$ has a handle decomposition with a single 0-handle, 2-handle and 4-handle,
so again Proposition \ref{prop:Crane-Yetter} is applicable and we can reuse the results from \cite[Section 3.4]{BaerenzBarrett2016Dichromatic}.
The 2-handle is attached along the 1-framed unknot
(see the equation below),
and accordingly the evaluation of this diagram is the trace over the \emph{twist},
whose eigenvalue on a simple, trivially graded object $X$ we will denote by $\theta_X$ here:
\begin{equation}
	I\left(
	\begin{tikzpicture} [ baseline = {(0,-0.1)} ]
		\pic {CP2};
	\end{tikzpicture}
	\right)
	= \sum_{X \in \mathcal{O}(\mathcal{C}_e)}\qdim{X}^2 \theta_X \cdot \qdim{\Omega_\mathcal{C}}^{-1}
\end{equation}
Flipping the orientation famously results in a non-diffeomorphic manifold,
$\opCP 2$.
Not surprisingly, the Kirby diagram is the mirror diagram of the above,
and its invariant is the same with $\theta$ replaced by $\theta^{-1}$.
The values $\sum_{X \in \mathcal{O}(\mathcal{C}_e)}\qdim{X}^2 \theta_X^\pm$ are known as the \emph{Gauss sums} of the ribbon fusion category $\mathcal{C}_e$.

\subsubsection{$S^1 \times S^1 \times S^2$}

To study an example where the presence of 3-handles influences the manifold,
we turn to $S^1 \times S^1 \times S^2$,
borrowing several calculational steps from \cite[Section 6.2.1]{BaerenzBarrett2016Dichromatic}.
For now, we assume that the $G$-grading on $\mathcal{C}$ is faithful
(i.e. $\mathcal{C}_g \not\simeq 0$ for any $g$),
which is quite restrictive in the context of $G\times$-BSFCs:
Compare Remark \ref{rem:Gx grading not faithful} and note that Corollary \ref{cor:killing iso} implies that any faithfully graded $G\times$-BSFC has natural isomorphisms $X \to \acts{g}{X}$ for every group element $G$ on the transparent subcategory $X \in \mathcal{C}'$.
Completing this calculation for non-faithfully graded $G\times$-BSFCs requires a deeper study of the group action and would probably yield an interesting invariant of the category.

Recall the Kirby diagram from Figure \ref{fig:S1S1S2} and note that $h_{2a}$ does not have any 3-handle attached to it and will thus be labelled $\mathcal{C}_e$,
while $h_{2b}$ has both 3-handles attached twice with opposite boundary orientation and will be labelled by $\Omega_{[g_1, g_2]}$,
where $[g_1, g_2] \coloneqq g_1 g_2 g_1^{-1} g_2^{-2}$.
\fxwarning{Consider renaming $g_1 \mapsto g_A, g_2 \mapsto g_B$ according to 3-handle labels.}
\fxwarning{Consider globally changing conventions that 3-handles are lower case (group elements) and 2-handles are upper case (objects).
	But note that we label 2-handles with simples.}
The number of 1-handles equals the number of 2-handles,
cancelling the normalisation,
and we can begin to evaluate the diagram:
\newdimen\encircleradius
\newdimen\mysep
\tikzmath{
	\encircleradius = 0.4cm;
	\height         = 1.3;
	\width          = 2.3;
	\mysep          = 0.2cm;
}
\begin{align}
	I(S^1 \times S^1 \times S^2)
	&= \sum_{\mathclap{\substack{
		g_1, g_2 \in G
		\\
		X \in \mathcal{O}(\mathcal{C}_e)
		\\
		\alpha, \beta
	}}}
	\qdim{X}
	\cdot
	\begin{tikzpicture}[baseline=(O.base)]
		\node (O) {\vphantom{X}};
		\node [ plaquette ] (alpha)    at ( \width,  \height) {$\acts{g_1}{\alpha}$};
		\node [ plaquette ] (hatalpha) at (-\width, -\height) {$\tilde\alpha$};
		\node [ plaquette ] (hatbeta)  at ( \width, -\height) {$\tilde\beta$};
		\node [ plaquette ] (beta)     at (-\width,  \height) {$\acts{g_2}{\beta}$};
		\draw [ thick, directed ]
			(alpha)
				-- node [ above, very near start ] {$\acts{g_1g_2}{X}$}
				   node [ above, very near end   ] {$\acts{g_2g_1}{X}$}
			(beta)
		;
		\draw [ thick, directed ]
			(beta)
				-- node [ left ]  {$\acts{g_2}{X}$}
			(hatalpha)
		;
		\draw [ thick, directed ]
			(hatalpha)
				-- node [ below ] {$X$}
			(hatbeta)
		;
		\draw [ thick, directed ]
			(hatbeta)
				-- node [ right ] {$\acts{g_1}{X}$}
			(alpha)
		;
		\draw
			[ diagram
			, shorten <   = \diagramwhitesep
			, shorten >   = \diagramwhitesep
			]
			(\encircleradius, \height)
			arc
				[ radius      = \encircleradius
				, start angle =   0
				, end angle   = 360
				];
		\node [ above ] at (0, \height + \encircleradius) {$\Omega_{[g_1,g_2]}$};
	\end{tikzpicture}
	\label{eq:S1S1S2defeq}
	\intertext{
		We are using the dual basis convention from Figure \ref{fig:dual bases} when summing over $\alpha$ and $\beta$.
		The contribution from $h_{2b}$ is an encircling where Lemma \ref{lem:killing G-crossed} is applicable,
		so we can restrict the sum over $X$ to $\mathcal{C}_e'$,
		since $\mathcal{C}_{[g_1, g_2]} \not\simeq 0$.
		The morphism spaces like $\langle (\acts{g_2}{X})^* \otimes X \rangle \ni \tilde\alpha$ may not be familiar at first,
		but recall that Corollary \ref{cor:killing iso} defines a canonical isomorphism $X \to \acts{g_2}{X}$,
		the encircling by $\Omega_{g_2}$.
		The morphism space $\langle X^* \otimes X \rangle$ is inhabited by the coevaluation $\coev_X \colon \mathcal{I} \to X^* \otimes X$,
		and $\langle (\acts{g_2}{X})^* \otimes X \rangle$ is spanned by the coevaluation precomposed with the mentioned isomorphism.
		We thus insert the encircling with $\Omega_{g_2}$ graphically where $\acts{g_2}{X}$ enters $\tilde\alpha$,
		and insert the inverse of the isomorphism
		(which is encircling with $\Omega_{g_2}$ and prefactors detailled in Corollary \ref{cor:killing iso})
		at $\alpha$.
		The analogous computation can be done for $\beta$ and $\tilde\beta$.
		These four new encirclings can be slid off the original encircling from $h_{2b}$,
		by the sliding lemma \ref{lem:graded sliding lemma}.
		This changes its grade to $\mathcal{C}_e$,
		and unlinks the encirclings from the diagram.
		They can thus be evaluated and result in dimensional factors.
		Including all prefactors, we have:
	}
	&= \sum_{\mathclap{\substack{
		g_1, g_2 \in G
		\\
		X \in \mathcal{O}(\mathcal{C}_e')
		\\
		\alpha, \beta
	}}}
	\qdim{X}
	\cdot
	\begin{tikzpicture}[baseline=(O.base)]
		\node (O) {\vphantom{X}};
		\node [ plaquette ] (alpha)      at ( \width,  \height) {$\alpha$};
		\node [ plaquette ] (tildealpha) at (-\width, -\height) {$\tilde\alpha$};
		\node [ plaquette ] (tildebeta)  at ( \width, -\height) {$\tilde\beta$};
		\node [ plaquette ] (beta)       at (-\width,  \height) {$\beta$};
		\draw [ thick, directed ]
			(alpha)
				-- node [ above, very near start ] {$X$}
			(beta)
		;
		\draw [ thick, directed ]
			(beta)
				-- node [ left ]  {$X$}
			(tildealpha)
		;
		\draw [ thick, directed ]
			(tildealpha)
				-- node [ below ] {$X$}
			(tildebeta)
		;
		\draw [ thick, directed ]
			(tildebeta)
				-- node [ right ] {$X$}
			(alpha)
		;
		\draw
			[ diagram
			, shorten <   = \diagramwhitesep
			, shorten >   = \diagramwhitesep
			]
			(\encircleradius, \height)
			arc
				[ radius      = \encircleradius
				, start angle =   0
				, end angle   = 360
				];
		\node [ above ] at (0, \height + \encircleradius) {$\Omega_e$};
	\end{tikzpicture}
	\\\nonumber
	&\qquad \cdot \qdim{\Omega_{g_1}} \qdim{\Omega_{g_2}}  \qdim{\Omega_{g_1^{-1}}} \qdim{\Omega_{g_2^{-1}}}
	\\\nonumber
	&\qquad \cdot \qdim{X}^{-2} \qdim{\Omega_e}^{-2} \qdim{\Omega_{g_1}}^{-1} \qdim{\Omega_{g_1}}^{-1}
	\intertext{
		By Lemma \ref{lem:global dimension degree},
		all factors of the form $\qdim{\Omega_g}$ are equal if they are nonzero.
		The remaining calculation follows \cite[Section 6.2.1]{BaerenzBarrett2016Dichromatic}.
		The ordinary killing lemma \ref{lem:killing} is applied,
		and evaluations and coevaluations with the correct prefactors inserted for the dual bases.
	}
	&= \lvert G \rvert^2 \cdot \lvert \mathcal{O}(\mathcal{C}_e') \rvert \cdot \qdim{\Omega_e}
\end{align}

\begin{anfxerror}{Picture problems}
	\begin{itemize}
		\item Write out some steps in between as formulae.
			I had e.g. the dimcircles of the slid-off other encirclings,
			but not sure whether that's so clear.
	\end{itemize}
\end{anfxerror}
\begin{anfxwarning}{Picture improvement}
	\begin{itemize}
		\item Unsatisfying vertical alignment of tikzpicture (node contents)
		\item Is the white sep for the encirclement good?
		\item Unify all circle sizes in the document
	\end{itemize}
\end{anfxwarning}
\fxnote{
	This calculation seems to imply that we can always take the incident circles of 3-handles on 1-handles as actual $\Omega$-loops!
	Maybe add this as part of the calculus?
	But what to do when 3-handles are incident as lines between 2-handle points?
	And what if it's not faithfully graded?
}
\fxwarning{It would be good if the middle step with the many encirclings could be illustrated}

\paragraph{Variant: Trivial grading on $\Z_3$}
To see the effect of the group action in a non-faithfully graded $G\times$-BSFC $\mathcal{C}$,
we choose explicit examples for the category.
First, set $G = \Z_2$ and $\mathcal{C}_{\Z_3}^{\Z_2} = \Vect_{\Z_3}$,
which denotes $\Z_3$-graded finite-dimensional vector spaces,
with simple objects $\mathcal{I}, \omega, \omega^*$.
Let the grading be concentrated in the trivial degree,
and equip the category with
the trivial $\Z_2$-action,
and the trivial braiding.
Second, define $\widetilde{\mathcal{C}}_{\Z_3}^{\Z_2}$ with the same data as $\mathcal{C}_{\Z_3}^{\Z_2}$,
but alter the group action such that the nontrivial element $\sigma \in \Z_2$
acts as $\acts{\sigma}{\omega} = \omega^*$,
and $\acts{\sigma}{\omega^*} = \omega$.

For $\mathcal{C}_{\Z_3}^{\Z_2}$ we jump ahead slightly to Proposition \ref{pro:dw} and learn:
\begin{equation}
	I_{\mathcal{C}_{\Z_3}^{\Z_2}}(S^1 \times S^1 \times S^2)
		= \lvert \{ \phi\colon \Z \oplus \Z \to \Z_2 \} \rvert
		\cdot \widehat{CY}_{\Vect_{\Z_3}}(S^1 \times S^1 \times S^2) = 4 \cdot 9 = 36
\end{equation}
For $\widetilde{\mathcal{C}}_{\Z_3}^{\Z_2}$,
we have to follow the calculation from \eqref{eq:S1S1S2defeq}.
Since $\Z_3$ is abelian, $\Omega_{[g_1, g_2]} = \Omega_e$ throughout.
The encircling unlinks and contributes as a global factor of $\qdim{\Omega_e} = 3$.
We do not have natural isomorphisms $X \cong \acts{g}{X}$ and need to perform the sums over the morphisms explicitly.
Luckily, the morphism spaces $\langle \acts{g}{X^*} \otimes X \rangle$ are only $\Co$
for $g = e$ or $X = \mathcal{I}$,
and 0 otherwise.
The diagram evaluates to 1 if both morphism spaces are $\Co$,
which leaves us to merely count the admissible labels $(X, g_1, g_2)$:
\begin{equation}
	I_{\widetilde{\mathcal{C}}_{\Z_3}^{\Z_2}}
		= \lvert \{ (\mathcal{I}, g_1, g_2) | g_1, g_2 \in \Z_2  \} \sqcup \{ (\omega, e, e), (\omega^*, e, e) \} \rvert \cdot \qdim{\Omega_e}
		= 6 \cdot 3 = 18
\end{equation}
\fxwarning{Recheck. And update TQFT examples if error.}

\subsection{Connected sum, smooth structures, and simply-connectedness}

It is a natural question to ask how strong the invariant is,
and in particular whether it can detect smooth structures.
In this subsection, we demonstrate that generically,
this is not the case.

\begin{proposition}
	\label{pro:connected sum}
	The invariant is multiplicative under connected sum.
	Explicitly, let $M_1$ and $M_2$ be two manifolds,
	and $M_1 \connsum M_2$ their connected sum.
	Then:
	\begin{equation}
		I_\mathcal{C}(M_1 \connsum M_2) = I_\mathcal{C}(M_1) \cdot I_\mathcal{C}(M_2)
	\end{equation}
\end{proposition}
\begin{proof}
	Given two Kirby diagrams for $M_1$ and $M_2$, respectively,
	their disjoint union is a diagram for $M_1 \connsum M_2$.
	Then the statement follows from Lemma \ref{lem:multiplicativity diagrams}.
\end{proof}
This proposition has far-reaching consequences for generic $G\times$-BSFCs:
\begin{corollary}
	Assume furthermore that the Gauss sum of $\mathcal{C}_e'$ is invertible.
	Then $I_\mathcal{C}$ does not detect exotic smooth structures.
\end{corollary}
\begin{proof}
	By \cite[Proposition 6.11]{OnBraidedFusionCats}, the Gauss sums of $\mathcal{C}_e$ are invertible,
	and thus also $I_\mathcal{C}(\CP 2)$ and $I_\mathcal{C}(\opCP 2)$,
	by the results of Section \ref{sec:CP2 and Gauss sums}.
	We can apply the remark from \cite[Section 1.4]{TeichnerKasprowskiPowell4ManifoldsCP2} to infer that $I(M)$ for any $M$ depends only on the signature, the Euler characteristic, the fundamental group and the fundamental class of $M$.
\end{proof}

\begin{corollary}
	Assume again that the Gauss sum of $\mathcal{C}_e'$ is invertible.
	For a simply connected manifold $M$ with Euler characteristic $\chi(M)$ and signature $\sigma(M)$,
	the invariant readily computes as:
	\begin{equation}
		I_\mathcal{C}(M)
			=     I_\mathcal{C}(  \CP 2)^{\frac{\chi(M) + \sigma(M)}{2}-1}
			\cdot I_\mathcal{C}(\opCP 2)^{\frac{\chi(M) - \sigma(M)}{2}-1}
	\end{equation}
\end{corollary}
\begin{proof}
	See e.g. \cite[Lemma 3.12]{BaerenzBarrett2016Dichromatic}.
\end{proof}
\begin{example}
	Under the same assumptions as before, we have
	$I_\mathcal{C}\left(\CP{2} \connsum \opCP{2}\right) = I_\mathcal{C}\left(S^2 \times S^2\right)$.
	\fxwarning{More left rights here and elsewhere}
\end{example}

\begin{remark}
	The invariant in its generic form does not seem to be helpful in the search for exotic smooth 4-manifolds,
	nor does it depend directly on the 3-type of the manifold
	(as was hoped e.g. in \cite[Section 5]{WilliamsonWangGCrossedTopPhases2017}).

	In line with what we noted at the beginning of Section \ref{sec:example manifolds},
	the more fruitful viewpoint is to study the resulting invariants for $G\times$-BSFCs when fixing a particular manifold.
	By Table \ref{table:example values},
	we could recover the global dimensions of $\mathcal{C}$, $\mathcal{C}_e$ and $\mathcal{C}_e'$,
	as well as the rank of $\mathcal{C}_e'$,
	and the Gauss sums.
	By the next subsection we will be able to recover most information about the group $G$.
	It remains an interesting open question whether there is a finite set of manifolds such that the corresponding set of invariants is complete on $G\times$-BSFCs,
	i.e. can distinguish any two $G\times$-BSFCs up to equivalence.
\end{remark}

\subsection{Untwisted Dijkgraaf-Witten theory}
\label{sec:dw}

\begin{definition}
	Let $G$ be a finite group.
	The \textbf{untwisted Dijkgraaf-Witten invariant} $DW_G(M)$ is defined as the number of $G$-connections on the manifold $M$:
	\begin{equation}
		DW_G(M) \coloneqq \lvert \{ \phi\colon \pi_1(M) \to G \} \rvert
	\end{equation}
	This normalisation is not the most common in the literature,
	but it is simpler for our purposes.
\end{definition}
\begin{proposition}[Compare {\cite[Proposition 4.5]{Cui2016TQFTs}}]
	\label{pro:dw}
	Let $\mathcal{C}$ be concentrated in the trivial degree,
	i.e. $g \neq e \implies \mathcal{C}_g \simeq 0$
	(or equivalently $\mathcal{C} \simeq \mathcal{C}_e$ as fusion categories),
	and the $G$-action be trivial.
	Then the invariant is a product of the Dijkgraaf-Witten invariant and the renormalised Crane-Yetter invariant \cite[Proposition 6.1.1]{BaerenzBarrett2016Dichromatic}:
	\begin{equation}
		I_\mathcal{C}(M) = DW_G(M) \cdot \widehat{CY}_{\mathcal{C}_e}(M)
	\end{equation}
	\fxnote{Note to self:
		This is right because here $\mathcal{C} \simeq \mathcal{C}_e$}
\end{proposition}
\begin{proof}
	Any $G$-labelling on 3-handles for which any 2-handle is labelled with $\Omega_g$ such that $g \neq e$ does not contribute to the invariant since $\Omega_g = 0$ in that case.
	Recall the presentation of the fundamental group from Section \ref{sec:fundamental group}.
	The 3-handles are generators of $\pi_1(M)$, while the 2-handles give relations.
	A homomorphism from $\pi_1(M)$ to $G$ is presented by an assignment of a $G$-element for every 3-handle,
	such that the relations of $\pi_1(M)$ are satisfied.
	The former is given by a $G$-labelling,
	the latter is implemented by the fact that only those $G$-labellings contribute to the invariant where all 2-handles are labelled by $\Omega_e$.
	Thus the contributing $G$-labellings indeed run over all possible homomorphisms $\phi\colon \pi_1(M) \to G$.
	As in the proof of Proposition \ref{prop:Crane-Yetter},
	it is easy to see that each such contribution is just $\widehat{CY}_{\mathcal{C}_e}(M)$.
\end{proof}

\section{Recovering the state sum model}
\label{sec:ssm}

The main motivation for this work was the second open question in \cite[Section 7]{Cui2016TQFTs}.
There, a state sum model is defined for $G\times$-BSFCs in terms of triangulations,
but it is noted that the model is impractical to compute.
We will show now that the invariant defined here is equal to the state sum up to a classical factor,
giving an economical way to calculate the state sum.

The state sum takes a $G\times$-BSFC $\mathcal{C}$ and a manifold $M$ with compatible triangulation $\mathcal{T}$,
labels 1-simplices with group elements, 2-simplices with simple objects and 3-simplices with morphisms.
For every 4-simplex, a local partition function is defined.

There is a well-known procedure called \textbf{chain mail} to turn an invariant defined on Kirby diagrams into a state sum model.
A smooth triangulation gives rise to a handle decomposition,
and we have to pull back the definition along this procedure.
Chain mail has been described already in \cite[Section 4.3]{Roberts:1995SkeinTheoryTV}.
As we go through the steps, we refer to \cite[Section 5]{BaerenzBarrett2016Dichromatic} for details,
which matches our conventions and notation to a large degree.
\begin{definition}[Well-known]
	\label{def:dual handle decomposition}
	Every triangulation $\mathcal{T}$ of a smooth manifold $M$ gives rise to a handle decomposition,
	where the $k$-simplices are thickened to $(n-k)$-handles.
	This is called the \textbf{dual handle decomposition}.
	The set of $k$-simplices is denoted as $\mathcal{T}_k$.
\end{definition}
\begin{theorem}
	\label{thm:equal to SSM}
	Up to a factor involving the Euler characteristic,
	the invariant from Definition \ref{def:invariant}
	is equal to the state sum $Z_\mathcal{C}$ from \cite[(23)]{Cui2016TQFTs}.
	Explicitly, let $M$ be a connected manifold and $\mathcal{T}$ an arbitrary triangulation,
	then:
	\begin{equation}
		I_\mathcal{C}(M) = Z_\mathcal{C}(M; \mathcal{T}) \cdot \qdim{\Omega_\mathcal{C}}^{1 - \chi(M)} \cdot \lvert G \rvert
	\end{equation}
\end{theorem}
\begin{proof}
	We begin with the dual handle decomposition of $\mathcal{T}$.
	It does not have a single 0-handle and 4-handle,
	thus we cannot directly calculate $I$ of this decomposition.
	The strategy is to slightly modify the manifold until we can calculate the invariant,
	and recover the original value from there.
	Each 0-handle can be regarded as a drawing canvas
	which contains the Kirby diagram of the corresponding 4-simplex $\sigma$.
	Proposition \ref{prop:4-simplex evaluation} will show that the evaluation of such a diagram
	(denoted here as $\langle \sigma \rangle$)
	is just the quantity $Z^\pm_F(\sigma)$ from \cite[Section 3.1]{Cui2016TQFTs}.

	In order to join all drawing canvases by cancelling the excessive 0-handles and,
	we attach $\lvert \mathcal{T}_4 \rvert - 1$ 1-handles
	(recall that the handle decomposition had $\lvert \mathcal{T}_4 \rvert$ 0-handles to begin with).
	Since the resulting manifold is not closed,
	we attach $\lvert \mathcal{T}_4 \rvert - 1$ 3-handles to cancel the boundary.
	The resulting manifold is $N \coloneqq M \connsum^{\lvert \mathcal{T}_4 \rvert - 1} S^1 \times S^3$,
	for which we know from Proposition \ref{pro:connected sum}:
	\begin{equation}
		I_\mathcal{C}(M) = I_\mathcal{C}(N) \cdot \lvert G \rvert^{1 - \lvert \mathcal{T}_4 \rvert} \qdim{\Omega_\mathcal{C}}^{1 - \lvert \mathcal{T}_4 \rvert}
	\end{equation}
	We know that still $\lvert \mathcal{T}_0 \rvert - 1$ 3-handles will be cancelled by excessive 4-handles,
	then we can calculate $I_\mathcal{C}(N)$ diagrammatically.
	Since the diagram for $N$ is a disjoint union of diagrams for each 4-simplex $\sigma \in \mathcal{T}_4$
	(with $\lvert \mathcal{T}_4 \rvert - 1$ 3-handles added and $\lvert \mathcal{T}_0 \rvert - 1$ removed,
	 each incurring a factor of $\lvert G \rvert$),
	it factors as a product like in the proof of Lemma \ref{lem:multiplicativity diagrams},
	and we can infer from \eqref{eq:invariant long}:
	\begin{equation}
		I_\mathcal{C}(N)
			= \sum_{g,X} \qdim{\Omega_\mathcal{C}}^{\lvert \mathcal{T}_3 \rvert - \lvert \mathcal{T}_2 \rvert}
			  \lvert G \rvert^{\lvert \mathcal{T}_4 \rvert - \lvert \mathcal{T}_0 \rvert}
			  \prod_{h_2} \qdim{X_{h_2}}
			  \prod_{\sigma \in \mathcal{T}_4}
			  \langle \sigma \rangle
	\end{equation}
	We compare with \cite[(23)]{Cui2016TQFTs},
	which uses the notation $D^2 = \qdim{\Omega_\mathcal{C}}$ and labels objects with $f$,
	and defines (in our notation):
	\begin{equation}
		Z_\mathcal{C}(M; \mathcal{T})
			= \sum_{g,X} \qdim{\Omega_\mathcal{C}}^{\lvert \mathcal{T}_0 \rvert - \lvert \mathcal{T}_1 \rvert}
			  \lvert G \rvert^{-\lvert \mathcal{T}_0 \rvert}
			  \prod_{h_2} \qdim{X_{h_2}}
			  \prod_{\sigma \in \mathcal{T}_4}
			  \langle \sigma \rangle
	\end{equation}
	We combine all three equations and recall that the Euler characteristic is $\chi(M) = \lvert \mathcal{T}_0 \rvert - \lvert \mathcal{T}_1 \rvert + \lvert \mathcal{T}_2 \rvert - \lvert \mathcal{T}_3 \rvert + \lvert \mathcal{T}_4 \rvert$
	to verify the claim.
\end{proof}
\begin{proposition}
	\label{prop:4-simplex evaluation}
	The Kirby diagram of a 4-simplex $\sigma$ evaluates to the quantity $Z^\pm_F(\sigma)$ from \cite[Section 3.1]{Cui2016TQFTs}.
\end{proposition}
\begin{proof}
	The only novelty over \cite[Section 4.3]{Roberts:1995SkeinTheoryTV}
	and \cite[Section 5]{BaerenzBarrett2016Dichromatic}
	is the appearance of 3-handles in the Kirby diagram of a 4-simplex.
	Each 3-handle is given by a 1-simplex, or edge, in the triangulation,
	and is thus specified by its two end vertices.
	For any third vertex (of which there exist three),
	a 2-simplex, or triangle, exists such that the edge is part of the boundary of the triangle,
	and thus the 3-handle attaches to the corresponding 2-handle.
	There are five 3-simplices (each opposite a vertex)
	which are thickened to five 1-handles.
	These connect the 4-simplex to any neighbouring 4-simplex,
	and thus only one of the attaching $D^3$ is visible in the diagram.
	Since a 3-simplex has four boundary triangles,
	four 2-handles attach to each 1-handle.
	The resulting diagram is shown in Figure \ref{fig:4-simplex},
	but not all 3-handles are shown, for clarity.
	The diagram is not unique since any 3-dimensional isotopy can be applied to it,
	but it is favourable to minimise the number of crossings to keep calculations simple.
	Similarly, it is desirable to minimise the number of 3-handles covering a 1-handle attaching disk.
	The minimum number is one each,
	and it is achieved in the diagram,
	following the convention of \cite[Figure 13 (left)]{Cui2016TQFTs}.
	\fxwarning{this proof is still a bit terse}
\end{proof}
\begin{figure}
	\centering
	\pgfdeclarelayer{nodes}
	\pgfsetlayers{main,nodes}
	\begin{tikzpicture} [ looseness = 1.5 ]
		\tikzmath{
			\angleo = 70;
			\anglei = 25;
			\anglex = 70;
			\angley = 10;
		}
		\begin{pgfonlayer}{nodes}
		\node [ 1-handle ] (h0134) at (-1,  2) {};
		\node [ 1-handle ] (h0124) at (-1,  0) {};
		\node [ 1-handle ] (h1234) at ($(h0134)+(-\anglei:2.5cm)$) {};
		\node [ 1-handle ] (h0234) at ($(h1234)+(180+\anglex:2.5cm)$) {};
		\end{pgfonlayer}
		\node [ 1-handle ] (h0123) at ($(h0134)+ (\anglei:2.5cm)$) {};
		\coordinate (inbetween) at (3.6, 2);

		\draw [ diagram ]
			(h0123) to [ out = -20, in = -20 ] (h0124)
			(h0123) to [ out =  10, in =  30 ] (h0234)
			(h0134) to [ out =  \angleo, in =  90 ] (inbetween)
			        to [ out = -90, in = -\angley ] (h0234)
		;
		\draw [ diagram ]
			(h0124) -- (h0234)
			(h0123) -- (h1234)
			(h0124) -- (h1234)
			(h0134) -- (h0123)
			(h0134) -- (h0124)
			(h0134) -- (h1234)
			(h0234) -- (h1234)
		;
		\path [ 3-handle ]
			(h0134.\angleo)
				to [ out =  \angleo, in =  90 ]
				(inbetween)
				to [ out = -90, in = -\angley ]
			(h0234.-\angley)
			--
			(h0234.\anglex)
			--
			(h1234.180+\anglex)
			--
			(h1234.180-\anglei)
			--
			(h0134.-\anglei)
			--
			(h0134.\angleo)
		;
	\end{tikzpicture}
	\caption{The Kirby diagram of a 4-simplex,
		showing only the nontrivially acting 3-handle.
	}
	\label{fig:4-simplex}
\end{figure}

\subsection{From state sum models to TQFTs}
\label{sec:SSMs to TQFTs}
There is a well-known recipe to define a TQFT $\mathcal{Z}_\mathcal{C}$ from a topological state sum model \cite{TuraevViro:1992865} such that $\mathcal{Z}_\mathcal{C}(M^4) = Z_\mathcal{C}(M^4)$ for a 4-manifold,
where $M^4$ is regarded as a cobordism from the empty manifold to itself.
The dimensions of the state spaces $\mathcal{Z}_\mathcal{C}(N^3)$ assigned to boundary 3-manifolds are of interest and can be calculated as $\dim \mathcal{Z}_\mathcal{C}(N^3) = \tr 1_{\mathcal{Z}_\mathcal{C}(N^3)} = \mathcal{Z}_\mathcal{C}(S^1 \times N^3) = Z_\mathcal{C}(S^1 \times N^3)$.
Using Theorem \ref{thm:equal to SSM},
this calculation is now much easier than directly from the state sum:
\begin{corollary}
	The dimension of the state space assigned to a 3-manifold $N$ is:
	\begin{equation}
		\dim(\mathcal{Z}_\mathcal{C}(N)) = I_\mathcal{C}(S^1 \times N) \cdot \qdim{\Omega_\mathcal{C}}^{-1} \cdot \lvert G \rvert^{-1}
	\end{equation}
\end{corollary}
\begin{proof}
	Note that the Euler characteristic is multiplicative and $\chi(S^1 \times M) = \chi(S^1) = 0$.
	We continue the calculation of $Z_\mathcal{C}(S^1 \times M)$ to arrive at the result.
\end{proof}
\begin{examples}
	\begin{itemize}
		\item
			One readily verifies that $\dim(\mathcal{Z}_\mathcal{C}(S^3)) = 1$.
			It is well-known that such a TQFT is,
			up to a factor of $\mathcal{Z}_\mathcal{C}(S^4)^{-1}$,
			multiplicative under connected sum,
			in agreement with Proposition \ref{pro:connected sum}.
		\item
			For a faithfully graded $G\times$-BSFC,
			we calculate $\dim(\mathcal{Z}_\mathcal{C}(S^1 \times S^2)) = \lvert G \rvert \cdot \lvert \mathcal{O}(\mathcal{C}_e') \rvert \cdot \qdim{\Omega_e} \cdot \qdim{\Omega_\mathcal{C}}^{-1} = \lvert \mathcal{O}(\mathcal{C}_e') \rvert$,
			showing that the theory is non-invertible when $\mathcal{C}_e$ is not modular.
		\item
			For the two examples from $\Z_2$ acting on $\Vect_{\Z_3}$,
			we get
			$\dim\left(\mathcal{Z}_{\mathcal{C}_{\Z_3}^{\Z_2}}(S^1 \times S^2)\right) = 6$
			and $\dim\left(\mathcal{Z}_{\widetilde{\mathcal{C}}_{\Z_3}^{\Z_2}}(S^1 \times S^2)\right) = 3$.
	\end{itemize}
\end{examples}

\subsection{Discussion: Defining a TQFT directly from handle decompositions}
\label{sec:TQFT directly}
Manifolds with nonempty boundary are described by handle decompositions as well \cite[Section 5.5]{GompfStipsicz}.
In fact, handle attachments can be seen as the fundamental generators of bordisms,
and TQFTs can be defined naturally in terms of them \cite{JuhaszTQFTsSurgery14}.

An arbitrary Kirby diagram
(one that does not necessarily correspond to a closed manifold)
specifies both a boundary 3-manifold $N$
and a bordism $M\colon \emptyset \to N$.
The boundary $N$ is constructed via \emph{surgery},
as per Remark \ref{rem:surgery}.
It is in principle possible to define a TQFT directly from handle decompositions of bordisms,
but doing so rigorously is beyond the scope of this article and is left as future work.

In Atiyah's axiomatisation of TQFTs \cite{Atiyah1988TQFT},
a vector space $\mathcal{Z}(N)$ is assigned to every boundary manifold $N$,
and a vector $\mathcal{Z}(M) \in \mathcal{Z}(\partial M)$
is assigned to every top-dimensional manifold.
We informally propose these two constructions for a $G\times$-BSFC $\mathcal{C}$.
(It should also be possible to generalise the approach to functorial TQFTs,
 and repeat the constructions for bordisms $M\colon N_1 \to N_2$ with a nonempty domain $N_1$ by using relative Kirby calculus.)

In analogy to the Turaev-Viro model \cite{Kirillov2011:Stringnet},
one would define a \emph{string net space},
or ``skein space'',
for $\mathcal{Z}(N)$.
Essentially, it would be defined as the vector space over $\mathcal{C}$-labelled diagrams
(as in Definition \ref{def:labelling})
\emph{embedded in} $N$,
modulo local relations in $\mathcal{C}$
(the evaluation from Definition \ref{def:labelled diagram}).
For $S^3$, this space is then tautologically 1-dimensional since the labelled diagrams are defined to be evaluated on it.
But for a more complicated manifold $N$,
the dimension of its state space should be higher,
and should in particular coincide with the values $I_\mathcal{C}(S^1 \times N)$ derived in the previous subsection.

For a given Kirby diagram $K$,
the manifold $S^3(K)$ is defined as surgery on $S^3$ along the spheres in $K$.
The string net space $\mathcal{Z}(S^3(K))$ has then an easier description as $\mathcal{C}$-labelled diagrams embedded in $S^3 \backslash K$,
modulo local relations and Kirby moves.
(The latter are a complete set of moves that relate any two surgery diagrams of diffeomorphic 3-manifolds.)

The vacuum state assigned to $S^3(K)$
(i.e. the vector corresponding to $M\colon \emptyset \to N$,
 where $K$ is a Kirby diagram for $M$)
should then simply be the empty diagram.
This could confuse at first since different $M$ potentially give rise to different vacua,
but note that the vacuum state is only the empty diagram in the particular surgery diagram $K$,
and would usually not be mapped to the vacuum state of $S^3(K')$ under a diffeomorphism $S^3(K) \cong S^3(K')$,
for a different diagram $K'$ corresponding to a non-diffeomorphic bordism $M'\colon \emptyset \to N$.

\section{Spherical fusion 2-categories and related work}
\label{sec:3-cats}
This section is kept in an informal discussion style.

\medskip

One reason to adopt the graphical calculus presented in Section \ref{sec:Graphical calculus} was to stay close to the language of Kirby diagrams with 3-handles,
but another reason was to imitate the graphical calculus of Gray categories (semistrict 3-categories) with duals \cite{Schaumann:GrayCats}.
The main reason not to translate Kirby diagrams into Gray category diagrams is to take advantage of the graphical calculus of pivotal categories,
which is a considerable shortcut.

$G\times$-BSFCs can be seen as monoidal 2-categories \cite{Cui2016TQFTs, DouglasReutter2018fusion},
which in turn can be seen as one-object 3-categories,
this approach is thus natural.
The translation of $G\times$-BSFCs into monoidal 2-categories is summarised in Table \ref{table:translation Gx-SFCs to monoidal 2-categories}.
\begin{table}
	\centering
	\begin{tabular}{lll}
		\toprule
		Datum in a $G\times$-BSFCs
			& Notation
				& Datum in the monoidal 2-category
		\\\midrule
		Group element
			& $g \in G$
				& Object
		\\\midrule
		Grade
			& $\mathcal{C}_g$
				& Endocategory
		\\\midrule
		Object
			& $A \in \mathcal{C}$
				& 1-morphism
		\\\midrule
		Morphism
			& $f\colon A \to B$
				& 2-morphism
		\\\midrule
		Group multiplication
			& $g_1 g_2$
				& Monoidal product
		\\\midrule
		Group inverse
			& $g^{-1}$
				& Duality
		\\\midrule
		Group action
			& $\acts{g}{A}$
				& 1-morphism composition
		\\\midrule
		Monoidal product
			& $A \otimes B$
				& 1-morphism composition
		\\\midrule
		Crossed braiding
			& $c_{X,Y}\colon X \otimes Y \to \acts{g}Y \otimes X$
				& Interchanger
		\\\bottomrule
	\end{tabular}
	\caption{The translation of $G\times$-BSFCs to monoidal 2-categories.
		Summarized from \cite[Section 6]{Cui2016TQFTs} and \cite[Constructions 2.1.23 and 2.3.6]{DouglasReutter2018fusion}.}
	\label{table:translation Gx-SFCs to monoidal 2-categories}
\end{table}

Since \cite{Mackaay:Sph2cats}, there is a search for a good categorification of spherical fusion categories to the world of monoidal 2-categories.
The definition proposed there was soon found to be too restrictive,
but it was long an open problem to find a more general and still well-behaved definition.
It is shown in \cite{Cui2016TQFTs} explicitly that $G\times$-BSFCs already yield more general monoidal 2-categories,
but still a good definition of spherical fusion 2-category is expected to be much more general.

\cite{DouglasReutter2018fusion} offers a detailed and promising definition,
relates it to \cite{Cui2016TQFTs},
and defines a state sum model.
It is to be expected that the invariant based on Kirby diagrams with 3-handles presented here can be generalised to that notion of spherical fusion 2-categories,
and that it will coincide with that state sum.
Some elements of this article that played no role in the invariant,
such as the fold lines of sheets,
are then less trivial to handle.
The details of this correspondence are left for future work,
but one can speculate that the handle picture will once again allow for much more efficient computations.

In the following discussion,
we want to informally discuss such a generalisation,
and digress on a good ``higher spherical axiom''.

\paragraph{4-cocycles and pentagonators}
The Dijkgraaf-Witten model is usually defined for the datum of a finite group $G$
\emph{and} a 4-cocycle $\omega \in C^4(G, U(1))$
(with cohomologous cocycles giving rise to equivalent theories).
If $[\omega] \neq 0$, the theory is called ``twisted'',
but in Section \ref{sec:dw},
only the untwisted Dijkgraaf-Witten model occurred.
In \cite[Section 4.4]{Cui2016TQFTs},
the state sum model is generalised to include a 4-cocycle,
and the twisted Dijkgraaf-Witten model is recovered.

The reason seems to be that the monoidal 2-category defined from a $G\times$-BSFC is very strict.
In a general monoidal bicategory
(see e.g. \cite[Appendix B.4]{Schommer-Pries:PhD}),
there exist associators for the monoidal product,
and even these do not satisfy the pentagon axiom on the nose,
but rather up to an invertible modification,
the \textbf{pentagonator}.
In a monoidal 2-category from a $G\times$-BSFC,
the monoidal product is associative though,
and the pentagonator is trivial.
\fxnote{I'd like a remark that a Gx-BSFC together with a pentagonator should give a monoidal 2-cat, and its DR-model should be Cui's \cite[Section 4.4]{Cui2016TQFTs}.
	But unfortunately DR don't have pentagonators as explicit structure.
	So all their structure goes into the braiding?
	And they can recover pentagonator symbols in a good skeletalisation.
}
This seems quite restrictive,
and we would want to allow for a weaker structure.

Still with an associative monoidal product,
the pentagonator can contain nontrivial data.
It consists of an invertible automorphism of the identity 1-morphism for the tensor product of every four objects,
satisfying the ``associahedron equation''.
For the case of monoidal 2-categories from $G\times$-BSFCs,
one can verify that pentagonators in fact correspond to 4-cocycles of $G$.
It would be interesting to study to what data all further coherences of monoidal 2-categories correspond,
and to define the invariant, or the state sum model,
in terms of it.

Note though that in \cite[Definition 3.3.8]{DouglasReutter2018fusion},
no pentagonator occurs as structure of the monoidal category,
but it is strictified, as in Gray categories.

\paragraph{The 3-spherical axiom}
The 3-dimensional Turaev-Viro-Barrett-Westbury model
(\cite{TuraevViro:1992865, BarrettWestbury:1993Invariants})
is most generally defined for a \emph{spherical} fusion category.
The spherical axiom demands that the left and right traces defined by the pivotal structure are equal.
Given a diagram in $\R^2$ of a morphism in a spherical fusion category,
this axiom is equivalent to embedding the diagram in $S^2 \cong \R^2 \cup \infty$
and allowing any line to be isotoped past $\infty$.
Unsurprisingly, this is the lower-dimensional analogon of the 3-$\infty$-move from Section \ref{sec:slides}:
3-manifolds have handle decompositions,
and when depicting them in the plane,
there is a ``2-$\infty$-move'' between different diagrams of the same decomposition.
The graphical calculus in the category must correspondingly validate this move.
(One might argue that the move should thus be called the ``2-spherical'' move,
since it takes place on a 2-sphere.)

It is only natural to require a ``3-spherical axiom'' in any suitable definition of spherical fusion 2-categories which corresponds to the 3-$\infty$-move.
In those 2-categories which come from $G\times$-BSFCs,
this move is trivially true \cite[Example 2.3.6]{DouglasReutter2018fusion},
since group elements act trivially on the tensor unit of the category.

Spherical fusion 2-categories satisfy precisely this move \cite[Definition 2.3.2]{DouglasReutter2018fusion},
justifying its name and its relevance.
In this greater generality, it becomes a necessary ingredient to define dimensions of objects of 2-categories.
In the graphical calculus,
a trivially attached 3-handle (Figure \ref{fig:illustrations attaching regions})
would evaluate to the dimension of its label.
In $G\times$-BSFCs, these dimensions are always 1, and thus irrelevant,
but in general, 3-handles would play a bigger computational role.

\paragraph{Defect theories and orbifold data}
It seems fruitful to compare to orbifold theories of defect theories.
The Turaev-Viro-Barrett-Westbury model is an orbifold TQFT of the trivial 3d defect theory \cite{CarquevilleRunkelSchaumann2017Orbifolds},
and the parametrising spherical fusion categories are algebras in a tricategory.
Gray categories are naturally objects in a (yet to be rigorously defined) tetracategory,
and one would hope that the Gray category derived from a $G\times$-BSFC as well as the orbifold data of the trivial 4d defect theory carry the appropriate algebra structure satisfying the correct higher sphericality axiom.

\section*{Acknowledgements}

The author wishes to thank the Peters-Beer-Stiftung for generous financial support.
For helpful discussions and correspondence,
thank you,
Daniel Scherl, Gregor Schaumann, Nils Carqueville, Flavio Montiel Montoya, Sascha Biberhofer, Dominic Williamson, Alex Bullivant, Bruce Bartlett, Andras Juhász, Dominik Wradzidlo,
Chris Douglas, David Reutter, Shawn Cui, Noah Snyder, Marcel Bischof, Arun Debray,
and the anonymous reviewer.

\appendix

\section{Manifolds with corners}
\label{sec:corners}

A closed $n$-manifold is defined to be locally homeomorphic to $\R^n$,
whereas a manifold with boundary is locally homeomorphic to
(a neighbourhood in) $\R^{n-1} \times \R_+$,
with $\R_+ \coloneqq [0, \infty)$.
Closed manifolds have products,
but defining a product for manifolds with boundary is not straightforward since corners arise.

A (smooth) $n$-manifold with $k$-corners is locally homeomorphic to $\R^{n-k} \times \R^k_+$.
A closed manifold thus has 0-corners,
a manifold with boundary has 1-corners,
and the product of two manifolds with boundaries is naturally a manifold with 2-corners.

Since as topological spaces
$\R^{n-1} \times \R_+ \cong \R^{n-k} \times \R^k_+$,
the notion of corners only becomes relevant in the context of smooth manifolds.
The boundary of a smooth $n$-manifold with $k$-corners is subtle to define.
The naive definition, in which a point is on the boundary whenever any of the last $k$ coordinates is 0,
does not yield a smooth manifold,
only a topological manifold.
In a sensible definition like \cite[Definition 2.6]{Joyce2009ManifoldsCorners},
the boundary is a smooth $(n-1)$-manifold with $(k-1)$-corners,
but it is in general not a submanifold.
Rather, $j$ disjoint copies are made for every $j$-corner point,
and the boundary is immersed in the $n$-manifold.
For example, the archetypal corner manifold $\R^2_+$ has as boundary $\R_+ \times \{0\} \sqcup \{0\} \times \R_+$,
where the origin $(0, 0)$ appears twice.

In this light, handles and their diverse regions are much easier to understand conceptually.
The $k$-disk $D^k$ is a manifold with boundary for $k \geq 2$,
and thus an $n$-dimensional $k$-handle $h_k = D^k \times D^{n-k}$ is generally an $n$-manifold with 2-corners.
Its boundary is the disjoint union $S^{k-1} \times D^{n-k} \sqcup D^k \times S^{n-k-1}$,
and we have called the former ``attaching region'' and the latter ``remaining region''.
It is an $(n-1)$-manifold with 1-corners, or simply with boundary.

As a topological manifold, all $n$-dimensional $k$-handles are homeomorphic to $D^n$,
reflecting the fact that we can glue the attaching region and the remaining region together along their common boundary $S^{k-1} \times S^{n-k-1}$
to arrive at the sphere $S^{n-1}$.

Glueing two handles along their boundary regions results in a manifold with corners,
but usually these corners are of index 2.
There is a canonical way to ``smoothen'' these corners
\cite[Remark 1.3.3]{GompfStipsicz},
so a handle body again becomes a manifold with boundary after the handle attachment.

As a simple example, we define $n$-gons:
\begin{definition}
	\label{def:n-gons}
	The \textbf{$n$-gon} $G_n$ is the unique oriented, simply connected 2-manifold with corners whosel boundary $\partial G_n$ consists of $n$ copies of the interval.
	In particular:
	\begin{itemize}
		\item $G_0 = S^2$ is the 2-sphere.
		\item $G_1$ is the ``teardrop'' manifold (see e.g. \cite[Section 3.1, Figure 1]{Schommer-Pries:PhD})
		\item $G_3$ is the 2-simplex.
	\end{itemize}
	For $n > 0$, the underlying topological manifold is the 2-disc $D^2$,
	but its smooth corner structure is different.
\end{definition}

\section{Details on Kirby calculus and diagrams}
\label{sec:details on diagrams}

\subsection{Kirby convention}
\fxwarning{Maybe restate lemma here?}
\begin{proof}[{Proof of Lemma \ref{lem:establish Kirby conventions}}]
	We have to show that any handle decomposition with a single 0-handle and 4-handle can be isotoped such that it is \emph{regular}
	(Definition \ref{def:regular handle decomposition})
	and satisfies the \emph{single picture conventions}
	(Definition \ref{def:single picture conventions}).

	Regularity can always be achieved by isotoping the attaching regions,
	see e.g. \cite[Proposition 4.2.7]{GompfStipsicz}.

	Let us show the two conditions for the single picture conventions:
	\begin{description}
		\item [2-1, 3-1]
			This has been illustrated in Figure \ref{fig:pushing 1-handle}.
			Inside the remaining region $\remainingboundary h_1 = [-1, +1] \times S^2$ of a 1-handle,
			apply an isotopy such that all attachments intersect $\{0\} \times S^2$ transversely
			in a finite subset (for a 2-handle attachment) or a 1-dimensional manifold (for a 3-handle attachment).
			By transversality,
			this can be extended to a small neighbourhood $(-\varepsilon, +\varepsilon) \times S^2$,
			and then pushed out of $\remainingboundary h_1$ with an isotopy into the main drawing canvas.
		\item [3-2]
			Analogously to the situation for 1-handles,
			focus on $\{0\} \times S^1 \subset D^2 \times S^1 = \remainingboundary h_2$ inside the remaining region of a 2-handle,
			and isotope any 3-handle attachment such that it intersects this circle transversely in a finite set of points.
			Extend to $D^2_\varepsilon \times S^1$
			(where $D^2_\varepsilon$ is a disk around 0 with radius $\varepsilon$)
			and isotope out of the remaining region by enlarging this disk,
			possibly pushing all folds outside into the drawing canvas.
	\end{description}
\end{proof}

\subsection{Blackboard framing}
\label{sec:blackboard framing}
An embedded thickened ribbon is given,
up to isotopy,
by a framed ribbon,
which is a ribbon with a nonvanishing normal vector field.
Immersing an oriented ribbon into an oriented surface defines a framing
by arbitrarily choosing a vector field into the left hand side of the surface,
as viewed from the ribbon direction.

A regular diagram as in Definition \ref{def:regular diagram} defines a blackboard framing on its ribbons and fold arcs.
Similarly, a Kirby diagram defines a blackboard framing for its 2-handle attachments.
For any given framing, it is always possible to match it with the blackboard framing by repeatedly applying the first Reidemeister move.

The notion of blackboard framing is standard for Kirby diagrams without 3-handles,
but has not been described yet for diagrams with 3-handles.
The projection of the diagram onto the plane locally defines a right hand side and a left hand side for every ribbon.
To assume the blackboard framing for at 3-handle attachment means to require that it may not change the side,
so the 3-handle attaching sphere must stay parallel to the plane close to where it attaches to a 2-handle.

As with any convention,
it must be shown that later constructions do not depend on the choices made.
Here, we have to show that the invariant does not depend on whether a particular 3-handle is attached on the right hand side or the left hand side.
Assume that 3-handles labelled $g_1, g_2, \dots g_K$ are attached to a 2-handle on the right hand side (from top to bottom)
and $g_{K+1}, \dots g_N, g$ on the left hand side (from bottom to top),
which implies that the 2-handle is labelled with $\Omega_{g'g}$,
where $g' \coloneqq g_1 g_2 \cdots g_N$.
If the 3-handle labelled $g$ is isotoped around to the right side,
the grade changes from $g'g$ to $gg'$,
and the 3-handle attaching sphere now covers the 2-handle circle.
Since $\acts{g}{\Omega_{g'g}} \cong \Omega_{gg'}$,
the diagram still evaluates to the same quantity.

\subsection{Related graphical calculi}

A standard reference for graphical calculus is \cite[Sections 3 and 4]{TuraevHFTAndGCats2000}.
Like the one presented here,
it also resembles the calculus of braided categories,
but instead of using 2-dimensional sheets,
it implements the $G$-crossed structure by means of a group homomorphism from the diagram complement to $G$.
We chose a slightly different route in this article and defined a more elaborate diagram language which matches Kirby calculus with 3-handles more closely,
but is also similar to the graphical calculus of semistrict 3-categories with duals \cite{Schaumann:GrayCats},
as mentioned in Section \ref{sec:3-cats}.
Turaev's calculus can be recovered soundly from the one presented here in Section \ref{sec:Graphical calculus},
as will be shown now.

First, we will paraphrase the relevant definitions from \cite[Sections 3 and 4]{TuraevHFTAndGCats2000} in our conventions.
\begin{definition}
	Let $G$ be a group,
	and $\mathcal{C}$ a $G\times$-BSFC.
	A \textbf{$\mathcal{C}$-coloured $G$-link} consists of:
	\begin{itemize}
		\item A chosen base point $z \in S^3$.
		\item A framed link $L \in S^3$.
		\item A group homomorphism $g\colon \pi_1 \left( S^3 \backslash L \right) \to G$.
		\item For every homotopy class of paths $\gamma$ from $z$ to some point in $L$,
			an object $X_\gamma$ in $\mathcal{C}$.
	\end{itemize}
	The objects $X_\gamma$ are subject to certain coherence laws detailed in \cite[Section 3.1]{TuraevHFTAndGCats2000}.
	For example, the action of a loop $\beta\colon S^1 \to S^3 \backslash L$ on a path $\gamma$
	corresponds to the group action of $g([\beta])$ on the labelling object $X_\gamma$.
\end{definition}

Whenever a diagram (Definition \ref{def:unlabelled diagram}) is in particular a link,
a $C$-colouring can be recovered from a labelling (Definition \ref{def:labelling}):

\begin{definition}
	Let $\mathcal{D}$ be a labelled diagram with no points.
	Construct the associated $\mathcal{C}$-coloured $G$-link as follows:
	\begin{itemize}
		\item The base point is $\infty$.
		\item The link is given by the ribbons $\mathcal{D}_1$,
			which necessarily are all circles.
		\item The group homomorphism is defined on a loop around a circle as the product of group elements of the incident sheets on the circle,
			just as in the typing relations of \ref{def:labelling}.
		\item For each circle $r$,
			choose one path from $\infty$ to it,
			running below the diagram in the projection onto the plane.
			Its homotopy class is assigned the labelling object $X(r)$.
			Other homotopy classes of paths to the same circle are assigned objects according to the conditions in \cite[Section 3.1]{TuraevHFTAndGCats2000}.
	\end{itemize}
	It is a tedious exercise to see that the group homomorphism and the path colourings are indeed well-defined.
	\fxerror{Do that exercise}

	The calculus of $\mathcal{C}$-coloured links is generalised in \cite[Section 4]{TuraevHFTAndGCats2000} to graphs with coupons labelled with morphisms in the category.
	The definition can then be repeated with general labelled diagrams, i.e. with points.
\end{definition}

\begin{lemma}
	The evaluation of labelled diagrams
	(Definition \ref{def:labelled diagram})
	extends to a functor of $G\times$-BSFCs.
\end{lemma}
\begin{proof}
	A routine
\end{proof}
\begin{corollary}
	Given a labelled diagram,
	its evaluation (Definition \ref{def:labelled diagram}) coincides with the action of the unique functor defined in \cite[Theorems 3.6 and 4.5]{TuraevHFTAndGCats2000} on the associated $\mathcal{C}$-coloured graph.
\end{corollary}
\begin{proof}
	The translation of the ribbons to (crossed) braidings is well-known,
	see \cite[Chapter I]{turaev:QuantumInvariants} and \cite[Theorem 4.5]{TuraevHFTAndGCats2000}.
	It remains to check whether the group action evaluates the same.
	The group action of the calculus in this work differs by the one from \cite{TuraevHFTAndGCats2000}
	just by the extra flexibility arising from folds of the sheets.
	This gives rise to isomorphic, but different group actions,
	and thus to different placements of the coherence isomorphisms.
	An example is Figure \ref{fig:G-coherences}.
	But from the coherence theorem of $G\times$-BSFCs \cite{Mueger2010:StructureBraidedCrossedGCategories, Cui2016TQFTs} we know that the whole morphism,
	and thus the evaluation is invariant with respect to these choices.
\end{proof}
\fxwarning{Could work out more details}

\section{1-handle slide lemmas in fusion categories}

To the knowledge of the author,
the following lemmas are not mentioned explicitly in the literature.

\tikzmath{
	\boundingwidth = 0.7;
}

\begin{definition}
	Let $\mathcal{C}$ be a pivotal fusion category and $A \in \mathcal{C}$ an object.
	Assuming the summation conventions from Figure \ref{fig:dual bases},
	the \textbf{projection onto the monoidal unit}
	$\pi^\mathcal{I}_A\colon A \to A$
	is defined as the following morphism:
	\begin{equation}
		\pi^\mathcal{I}_A \coloneqq
		\begin{tikzpicture} [ baseline, ]
			\node [ plaquette ] (alpha)      at (0,  0.5) {$\alpha$};
			\node [ plaquette ] (alphatilde) at (0, -0.5) {$\tilde\alpha$};
			\draw [ thick, directed ]
				(alpha)      -- (0,  2) node [ above ] {$A$}
			;
			\draw [ thick, opdirected ]
				(alphatilde) -- (0, -2) node [ below ] {$A$};
			\useasboundingbox (-\boundingwidth, 0) -- (\boundingwidth, 0);
		\end{tikzpicture}
	\end{equation}
\end{definition}

\begin{lemma}
	\fxwarning{Fix labels}
	The projection onto the monoidal unit is a natural transformation,
	i.e. $\pi^\mathcal{I}_A$ is natural in $A$.
\end{lemma}
\begin{proof}
	Explicitly, let $f\colon A \to B$ be a morphism in $\mathcal{C}$.
	Then the following is valid:
	\begin{anfxwarning}{To fix in the pictures}
		\begin{itemize}
			\item Directions?
				But those conventions don't need them?
				How about the other pictures?
		\end{itemize}
	\end{anfxwarning}
	\begin{equation}
		\begin{tikzpicture} [ baseline, yscale = 1.5 ]
			\node [ plaquette ] (f)          at (0,  0) {$f$};
			\node [ plaquette ] (alpha)      at (0, -1) {$\alpha$};
			\node [ plaquette ] (alphatilde) at (0, -2) {$\tilde\alpha$};
			\draw [ thick, directed ]
				(f)     -- (0, 3) node [ above ] {$B$}
			;
			\draw [ thick, directed ]
				(alpha) -- node [ right ] {$A$} (f)
			;
			\draw [ thick, opdirected ]
				(alphatilde) -- (0, -3) node [ below ] {$A$}
			;
			\useasboundingbox (-\boundingwidth, 0) -- (\boundingwidth, 0);
		\end{tikzpicture}
		=
		\sum_{X \in \mathcal{O}(\mathcal{C})} \qdim{X}
		\begin{tikzpicture} [ baseline, yscale = 1.5 ]
			\node [ plaquette ] (beta)       at (0,  2) {$\beta$};
			\node [ plaquette ] (betatilde)  at (0,  1) {$\tilde\beta$};
			\node [ plaquette ] (f)          at (0,  0) {$f$};
			\node [ plaquette ] (alpha)      at (0, -1) {$\alpha$};
			\node [ plaquette ] (alphatilde) at (0, -2) {$\tilde\alpha$};
			\draw [ thick, directed ]
				(beta)      -- (0, 3) node [ above ] {$B$}
			;
			\draw [ thick, directed ]
				(betatilde) -- node [ right ] {$X$} (beta)
			;
			\draw [ thick, directed ]
				(f)         -- node [ right ] {$B$} (betatilde)
			;
			\draw [ thick, directed ]
				(alpha)     -- node [ right ] {$A$} (f)
			;
			\draw [ thick, opdirected ]
				(alphatilde) -- (0, -3) node [ below ] {$A$}
			;
			\useasboundingbox (-\boundingwidth, 0) -- (\boundingwidth, 0);
		\end{tikzpicture}
		=
		\begin{tikzpicture} [ baseline, yscale = 1.5 ]
			\node [ plaquette ] (beta)       at (0,  2) {$\beta$};
			\node [ plaquette ] (betatilde)  at (0,  1) {$\tilde\beta$};
			\node [ plaquette ] (f)          at (0,  0) {$f$};
			\node [ plaquette ] (alpha)      at (0, -1) {$\alpha$};
			\node [ plaquette ] (alphatilde) at (0, -2) {$\tilde\alpha$};
			\draw [ thick, directed ]
				(beta)    -- (0, 3) node [ above ] {$B$}
			;
			\draw [ thick, directed ]
				(f)       -- node [ right ] {$B$} (betatilde)
			;
			\draw [ thick, directed ]
				(alpha)   -- node [ right ] {$A$} (f)
			;
			\draw [ thick, opdirected ]
				(alphatilde) -- (0, -3) node [ below ] {$A$}
			;
			\useasboundingbox (-\boundingwidth, 0) -- (\boundingwidth, 0);
		\end{tikzpicture}
		=
		\sum_{X \in \mathcal{O}(\mathcal{C})} \qdim{X}
		\begin{tikzpicture} [ baseline, yscale = 1.5 ]
			\node [ plaquette ] (beta)       at (0,  2) {$\beta$};
			\node [ plaquette ] (betatilde)  at (0,  1) {$\tilde\beta$};
			\node [ plaquette ] (f)          at (0,  0) {$f$};
			\node [ plaquette ] (alpha)      at (0, -1) {$\alpha$};
			\node [ plaquette ] (alphatilde) at (0, -2) {$\tilde\alpha$};
			\draw [ thick, directed ]
				(beta)       -- (0, 3) node [ above ] {$B$}
			;
			\draw [ thick, directed ]
				(f)          -- node [ right ] {$B$} (betatilde)
			;
			\draw [ thick, directed ]
				(alpha)      -- node [ right ] {$A$} (f)
			;
			\draw [ thick, directed ]
				(alphatilde) -- node [ right ] {$X$} (alpha)
			;
			\draw [ thick, opdirected ]
				(alphatilde) -- (0, -3) node [ below ] {$A$}
			;
			\useasboundingbox (-\boundingwidth, 0) -- (\boundingwidth, 0);
		\end{tikzpicture}
		=
		\begin{tikzpicture} [ baseline, yscale = 1.5 ]
			\node [ plaquette ] (beta)       at (0,  2) {$\beta$};
			\node [ plaquette ] (betatilde)  at (0,  1) {$\tilde\beta$};
			\node [ plaquette ] (f)          at (0,  0) {$f$};
			\draw [ thick, directed ]
				(beta)     -- (0, 3) node [ above ] {$B$}
			;
			\draw [ thick, directed ]
				(f)        -- node [ right ] {$B$} (betatilde)
			;
			\draw [ thick, opdirected ]
				(f)        -- (0, -3) node [ below ] {$A$}
			;
			\useasboundingbox (-\boundingwidth, 0) -- (\boundingwidth, 0);
		\end{tikzpicture}
	\end{equation}
\end{proof}
\begin{remark}
	For $G\times$-BSFCs,
	the identity holds even if there is a group action on $\tilde\alpha$ or $\alpha$:
	In the middle diagram,
	the disconnected part $\alpha \circ f \circ \tilde\beta$ forms a closed diagram,
	and by Remark \ref{rem:3-handle closed diagram} it can be moved freely from into or out of a group action.
\end{remark}
\begin{corollary}
	\label{cor:1-1-slide}
	The 1-1-handle slide is valid in the graphical calculus of pivotal fusion categories.
	Explicitly, let $f\colon \mathcal{I} \to A$.
	Then the following is true:
	\begin{equation}
		\begin{tikzpicture} [ baseline ]
			\node [ plaquette ] (f)          at (-1, 1) {$f$};
			\node [ plaquette ] (alpha)      at ( 1, 0) {$\alpha$};
			\node [ plaquette ] (alphatilde) at ( 1,-1) {$\tilde\alpha$};
			\draw [ thick, directed ]
				(f)          -- +(0, 3) node [ above ] {$A$};
			\draw [ thick, directed ]
				(alpha)      -- +(0, 4) node [ above ] {$B$};
			\draw [ thick, opdirected ]
				(alphatilde) -- +(0,-1) node [ below ] {$B$};
			\useasboundingbox (1+\boundingwidth, 0);
		\end{tikzpicture}
		=
		\begin{tikzpicture} [ baseline ]
			\node [ plaquette ] (f)          at (-1, 1) {$f$};
			\node [ plaquette ] (beta)       at ( 0, 3) {$\beta$};
			\node [ plaquette ] (betatilde)  at ( 0, 2) {$\tilde\beta$};
			\draw [ thick, directed ]
				(f)         -- node [ above left ] {$A$} (betatilde);
			\draw [ thick, directed ]
				(beta)      -- +(-1, 1) node [ above ] {$A$};
			\draw [ thick, directed ]
				(beta)      -- +( 1, 1) node [ above ] {$B$};
			\draw [ thick, opdirected ]
				(betatilde) -- +( 1,-4) node [ below ] {$B$};
			\useasboundingbox (-1-\boundingwidth, 0);
		\end{tikzpicture}
	\end{equation}
\end{corollary}
\begin{proof}
	Apply the previous lemma to $f \otimes 1_B$.
\end{proof}

\printbibliography

\end{document}